\documentclass[reqno]{amsart}

\usepackage{mathrsfs}
\usepackage{amsmath}
\usepackage{amssymb}
\usepackage{cite}
\usepackage{latexsym}
\usepackage{graphicx}

\usepackage{color}

\theoremstyle{plain}
\newtheorem{thm}{Theorem}[section]
\newtheorem{lem}[thm]{Lemma}
\newtheorem{cor}[thm]{Corollary}

\newtheorem{prop}[thm]{Proposition}
\newtheorem{rmk}[thm]{Remark}

\def\L{\mathscr{L}}

\def\U{\mathscr{U}}

\def\c{\mathrm{c}}
\def\d{\mathrm{d}}

\def\Cset{\mathbb{C}}
\def\Nset{\mathbb{N}}
\def\Pset{\mathbb{P}}
\def\Rset{\mathbb{R}}
\def\Sset{\mathbb{S}}
\def\Zset{\mathbb{Z}}

\def\id{\mathrm{id}}

\def\epsilon{\varepsilon}

\makeatletter
 \@addtoreset{equation}{section}
\makeatother
\def\theequation{\arabic{section}.\arabic{equation}}

\begin{document}


\title[Bifurcations of twisted solutions in the Kuramoto model]%
{Bifurcations of twisted solutions
 in a continuum limit for the Kuramoto model on nearest neighbor graphs}

\author{Kazuyuki Yagasaki}

\address{Department of Applied Mathematics and Physics, Graduate School of Informatics,
Kyoto University, Yoshida-Honmachi, Sakyo-ku, Kyoto 606-8501, JAPAN}
\email{yagasaki@amp.i.kyoto-u.ac.jp}

\date{\today}

\begin{abstract}
We study bifurcations of twisted solutions
 in a continuum limit (CL) for the Kuramoto model (KM)
 of identical oscillators defined on nearest neighbor graphs,
 which may be deterministic dense, random dense or random sparse,
 when it may have phase-lag.
We use the center manifold reduction,
 which is a standard technique in dynamical systems theory,
 and prove that the CL suffers bifurcations
 at which the one-parameter family of twisted solutions becomes unstable
 and a stable or unstable two-parameter family
 of modulated twisted solutions that oscillate or not
 depending on whether the phase-lag exists or not is born.
We demonstrate the theoretical results by numerical simulations
 for the KM on deterministic dense, random dense and random sparse graphs.
\end{abstract}

\maketitle


\section{Introduction}

We consider the Kuramoto model (KM) \cite{K75,K84} of identical oscillators
 on a graph $G_{n}=\langle V(G_n),E(G_n),W(G_n)\rangle$,
\begin{align}
\frac{\d}{\d t} u_k^n (t)
=\omega
+\frac{1}{n\alpha_{n}}\sum_{j=1}^{n}w_{kj}^n&\sin\left( u_j^n(t)-u_k^n(t)+\sigma\right),
\notag\\
& \quad k\in [n]:=\{1,2,\ldots,n\},
\label{eqn:dsys}
\end{align}
where $u_k^n:\Rset\rightarrow\Sset^1:=\Rset/2\pi\Zset$
 stands for the phase of oscillator at the node $k\in [n]$,
 $\omega$ is the natural frequency, 
 $\sigma\in(-\tfrac{1}{2}\pi,\tfrac{1}{2}\pi)$ is the phase-lag parameter,
 and $\alpha_{n}>0$ is a scaling factor that is one if $G_{n}$ is dense,
 and less than one with $\alpha_{n}\searrow 0$ and $n \alpha_{n} \to\infty$
 as $n\rightarrow \infty$, if $G_{n}$ is sparse.
Here $V(G_{n})=[n]$ and $E(G_{n})$ represent the sets of nodes and edges, respectively,
 and $W(G_{n})$ is an $n\times n$ weight matrix given by
\begin{equation*}
(W(G_{n}))_{kj}=
\begin{cases}
w_{kj}^{n} & \mbox{if $(k,j)\in E(G_{n})$};\\
0 &\rm{otherwise}.
\end{cases}
\end{equation*}
So we have
\[ 
E(G_{n})=\{(k,j)\in[n]^2\mid (W(G_{n}))_{kj}\neq 0\}, 
\] 
where each edge is represented by an ordered pair of nodes $(k,j)$, 
 which is also denoted by $j\to k$, and a loop is allowed.
If $W(G_{n})$ is symmetric,
 then $G_{n}$ represents an undirected weighted graph 
 and each edge is also denoted by $k\sim j$ instead of $j\to k$. 
When $G_{n}$ is a simple graph, 
 $W(G_{n})$ is a matrix whose elements are $\{0,1\}$-valued. 
When $G_{n}$ is a random graph, 
 $W(G_{n})$ is a random matrix. 
We say that $G_{n}$ is a \emph{dense} graph 
 if $\# E(G_{n})/(\# V(G_{n}))^2>0$ as $n \rightarrow \infty$. 
If $\# E(G_{n})/(\# V(G_{n}))^2\rightarrow 0$ as $n \rightarrow \infty$, 
 then we call it a \emph{sparse} graph.

Moreover, the weight matrix $W(G_{n})$ is given as follows.
Let $I=[0,1]$ and let $W^n\in L^2 (I^2)$ be a nonnegative function.
If $G_{k}$ is a deterministic dense graph, then
\begin{equation}
w_{kj}^{n} = \langle W^n\rangle_{kj}^{n}
:= n^2 \int_{I_k^n \times I_j^n}W^n(x,y) \d x\d y,
\label{eqn:ddg}
\end{equation}
where
\[
I_k^n:=
\begin{cases}
  [(k-1)/n,k/n) & \mbox{for $k<n$};\\
  [(n-1)/n,1] & \mbox{for $k=n$}.
\end{cases}
\]
If $G_{n}$ is a random dense graph,
 then $w_{kj}^{n}=1$ with probability 
\begin{equation}
\Pset(j\rightarrow k) = \langle W^n\rangle_{kj}^{n}, 
\label{eqn:rdg}
\end{equation}
where the range of $W^n$ is contained in $I$. 
If $G_{n}$is a random sparse graph, 
then $w_{kj}^{n}=1$ with probability 
\begin{equation}
\Pset(j \rightarrow k) = \alpha_{n} \langle \tilde{W}^n \rangle_{kj}^n, 
\quad \tilde{W}^n(x,y) :=\alpha_{n}^{-1} \wedge W^n(x,y), 
\label{eqn:rsg} 
\end{equation} 
where $\alpha_{n} =n^{-\gamma}$ with $\gamma\in(0,\frac{1}{2})$,
 and $a\wedge b=\min(a,b)$ for $a,b\in\Rset$.
Here $w_{kj}^n$, $k,j\in[n]$, are also assumed
 to be independently distributed in $j\in[n]$ for each $k\in[n]$
 when $G_n$ is a random graph, whether dense or sparse.
The function $W^n(x,y)$ is usually called a \emph{graphon} \cite{L12}.
Such a construction of a random graph where $W^n(x,y)$ does not depend on $n$
 was given in \cite{M19} and used in \cite{IY23,Y24b}.
We assume that there exists a measurable function $W\in L^2(I^2)$ such that 
\begin{equation}
\|W(x,y)-W^n(x,y)\|_{L^2(I^2)}=\int_{I^2}|W(x,y)-W^n(x,y)|^2\d x\d y\to 0
\label{eqn:W}
\end{equation}
as $n\to\infty$.

Such coupled oscillators in complex networks
 have recently attracted much attention and have been extensively studied.
They provide many mathematical models in various fields
 such as physics, chemistry, biology, social sciences and engineering.
Among them, the KM is one of the most representative models
 and has been generalized in several directions,
 e.g., to contain time delay or control force
 or to be defined on multiple graphs or a lattice.
It has very frequently been subject to research,
 especially  to discuss the synchronization phenomenon. 
See \cite{S00,PRK01,ABVRS05,ADKMZ08,DB14,PR15,RPJK16}
 for the reviews of vast literature on coupled oscillators in complex networks
 including the KM and its generalizations.

In \cite{IY23,KM17,M14a,M14b,M19},
 coupled oscillator networks including \eqref{eqn:dsys} were studied
 and shown to be well approximated by the corresponding continuum limits (CLs),
 for instance, which are given by
\begin{align}
\frac{\partial}{\partial t}u(t,x)
=& \omega+\int_I W(x,y) \sin(u(t,y)-u(t,x)+\sigma)\d y, \quad x \in I,
\label{eqn:csys}
\end{align}
for \eqref{eqn:dsys}. 
In particular, more general cases in which the networks depend on two or more graphs
 or the natural frequency of each oscillator is different
 were discussed in \cite{IY23}.
Similar results for such networks defined on single graphs
 and with the same natural frequency at each node
 were obtained earlier in \cite{KM17,M14a,M14b,M19}.
Such a CL was introduced for the classical KM,
 which depends on the single complete simple graph
 but may have natural frequencies depending on each oscillator,
 without a rigorous mathematical guarantee very early in \cite{E85}.
Similar CLs were utilized for the KM with nonlocal coupling and a single or zero natural frequency
 in \cite{GHM12,M14c,MW17,WSG06}.

Very recently, in \cite{Y24a},
 bifurcations and stability of synchronized solutions
 in the classical KM with equally distributed natural frequencies were studied,
 and they were shown to be very different from those in the corresponding CL.
Moreover,  in \cite{Y24b},
 bifurcations of completely synchronized solutions in the CL
 for the KM with two mode interaction depending on two graphs,
 one of which is uniform but may be deterministic, random dense or random sparse
 and the other is a deterministic nearest neighbor graph,
 were analyzed by using the center manifold reduction \cite{GH83,HI11,K04},
 which is a standard technique in dynamical systems,
 and it was proved that the CL suffers bifurcations
 at which the one-parameter family of completely synchronized state becomes unstable
 and a stable two-parameter family of $\ell$-humped sinusoidal shape stationary solutions
 ($\ell\ge 2$) appears.
The occurrence of such bifurcations was also suggested
 in a stability analysis and numerical simulation results a little earlier
 in \cite{IY23}.

In this paper we choose as the graphons $W^n(x,y)$ and $W(x,y)$
\[
W^n(x,y)=\begin{cases}
p & \mbox{if $(x,y)\in I_k^n\times I_j^n$ with $|k-j|\le n\kappa$ or $|k-j|\ge n(1-\kappa)$};\\
0 & \mbox{otherwise},
\end{cases}
\]
and
\begin{equation}
W(x,y)=\begin{cases}
p & \mbox{if $|x-y|\le\kappa$ or $|x-y|\ge1-\kappa$};\\
0 & \mbox{otherwise},
\end{cases}
\label{eqn:nn}
\end{equation}
with $p\in(0,1]$ and $0<\kappa<\tfrac{1}{2}$, which correspond
 to a nearest (more specifically, $\lfloor n\kappa\rfloor$-nearest) neighbor graph,
 where $\lfloor z\rfloor$ represents the maximum integer
 that is not greater than $z\in\Rset$,
 and study bifurcations of the \emph{$q$-twisted solutions},
\begin{equation}
u(t,x)=2\pi qx+\Omega t+\theta,\quad
\theta\in\Sset^1,
\label{eqn:tsol}
\end{equation}
with $q\in{\color{black}\Zset}$ in the CL \eqref{eqn:csys} for the KM \eqref{eqn:dsys},
 where $\Omega$ is a constant given by
\begin{equation}
\Omega=\omega+p\int_{x-\kappa}^{x+\kappa}\sin(2\pi q(y-x)+\sigma)\d y
=\omega+\frac{p\sin 2\pi q\kappa\sin\sigma}{\pi q}.
\label{eqn:Omega}
\end{equation}
Here  the graph $G_n$ may be deterministic dense, random dense or random sparse.
{\color{black}When $q=0$, Eq.~\eqref{eqn:tsol} represents the completely synchronized solutions.}
Note that $W^n(x,y),W(x,y)\to p$ for any $(x,y)\in I$ as $\kappa\to \tfrac{1}{2}$.

Substituting \eqref{eqn:tsol} into \eqref{eqn:csys},
 we easily see that Eq.~\eqref{eqn:tsol} gives a one-parameter family of solutions,
 which rotates with the speed $\Omega\neq 0$ when $\sigma\neq 0$ even if $\omega=0$,
 in the CL \eqref{eqn:csys} {\color{black}with \eqref{eqn:nn}}.
We take $\kappa$ as a control parameter
 and prove that the CL \eqref{eqn:csys} suffers bifurcations
 at which the one-parameter family \eqref{eqn:tsol} of twisted solutions becomes unstable
 and a stable or unstable two-parameter family of modulated twisted solutions
 that oscillate or not, depending on whether $\sigma\neq 0$ or not are born.
It follows from the results of \cite{IY23,Y24a,Y24b} that 
 such bifurcations also occur in the KM \eqref{eqn:dsys}.
We demonstrate our theoretical results
 by numerical simulations for the KM \eqref{eqn:dsys}
 on nearest neighbor graphs which may be deterministic dense,
 random dense or random sparse.

{\color{black}
Substituting $u(t,x)=2\pi qx+\Omega t+\theta+\hat{u}(t,x)$ and \eqref{eqn:nn} into \eqref{eqn:csys},
 we have
\begin{align}
\frac{\partial}{\partial t}\hat{u}(t,x)
=&p\int_{x-\kappa}^{x+\kappa}\sin(2\pi q(y-x)+\hat{u}(t,y)-\hat{u}(t,x)+\sigma)\d y\notag\\
& -\frac{p\sin2\pi q\kappa\sin\sigma}{\pi q},
\label{eqn:hatcsys}
\end{align}
where the domain of $\hat{u}(t,x)$ has been extended to $\Rset\times[-\kappa,1+\kappa]$
 such that it is periodic of period $1$ in $x$.
If $q<0$, then we change $q$ and $\hat{u}(t,x)$ to $-q$ and $-\hat{u}(t,x)$, respectively,
 so that Eq.~\eqref{eqn:hatcsys} becomes
\begin{align*}
\frac{\partial}{\partial t}\hat{u}(t,x)
=&p\int_{x-\kappa}^{x+\kappa}\sin(2\pi q(y-x)+\hat{u}(t,y)-\hat{u}(t,x)-\sigma)\d y\notag\\
& +\frac{p\sin2\pi q\kappa\sin\sigma}{\pi q},
\end{align*}
which is the same as \eqref{eqn:hatcsys} when $\sigma$ is replaced with $-\sigma$.
So we only consider the case of $q\ge 0$ below.
}

Twisted solutions and their stability
 in the KM \eqref{eqn:dsys} and CL \eqref{eqn:csys} with $\sigma=0$
 were studied for deterministic nearest neighbor graphs in \cite{GHM12,WSG06}
 and for a little more general random graphs in \cite{M14c,MW17}.
Twisted states have also been recognized
 as an important class of stationary solutions in the KM \eqref{eqn:dsys}
 on these network graphs
 since they provide valuable insights into the phase space structure
 and help us understand more complex spatial patterns
 such as chimera states \cite{GHM12,M14c,MM22}.
To the author's knowledge,
 twisted solutions and their stability for $\sigma\neq 0$
 and bifurcations of twisted solutions for $\sigma=0$ or $\neq 0$
 in the KM \eqref{eqn:dsys} and CL \eqref{eqn:csys}
 have not been reported previously.
Related results on feedback control of the KM \eqref{eqn:dsys}
 defined on deterministic nearest neighbor graphs and the complete simple graph
 and the CL \eqref{eqn:csys} will be reported  in \cite{Y24c}.

The outline of this paper is as follows:
In Section~2 we briefly review the previous fundamental results
 of \cite{IY23} and \cite{Y24a} in the context of the KM \eqref{eqn:dsys} and CL \eqref{eqn:csys}.
We analyze the associated linear eigenvalue problem and bifurcations of twisted solutions
 for the CL \eqref{eqn:csys} in Sections~3 and 4, respectively.
Numerical simulation results of the KM \eqref{eqn:dsys}
 on the nearest neighbor graphs are given in Section~5.


\section{Previous Fundamental Results}

We first review the results of \cite{IY23,KM17,M19,Y24a}
 on relationships between couples oscillator networks and their CLs
 in the context of  \eqref{eqn:dsys} and \eqref{eqn:csys}.
See Section~2 and Appendices~A and B of \cite{IY23}
 and Section~2 of \cite{Y24a} for more details
 including the proofs of the theorems stated below.
The theory can be extended to more general cases.

Let $g(x)\in L^2(I)$
 and let $\mathbf{u}:\Rset\to L^2(I)$ stand for an $L^2(I)$-valued function.
We have the following on the existence and  uniqueness of solutions
 to the initial value problem (IVP) of the CL \eqref{eqn:csys}
 (see Theorem~2.1 of \cite{IY23} or Theorem~3.1 of \cite{KM17}).
 
\begin{thm}
\label{thm:2a}
There exists a unique solution $\mathbf{u}(t)\in C^1(\Rset,L^2(I))$
 to the IVP of \eqref{eqn:csys} with
\begin{equation*}
u(0,x)=g(x).
\end{equation*}
Moreover, the solution depends continuously on $g$.
\end{thm}

We next consider the IVP of the KM \eqref{eqn:dsys}
 and turn to the issue on convergence of solutions in \eqref{eqn:dsys}
 to those in the CL \eqref{eqn:csys}.
Since the right-hand side of \eqref{eqn:dsys} is Lipschitz continuous in $u_k^n$, $i\in[n]$,
 we see by a fundamental result of ordinary differential equations
 (e.g., Theorem~2.1 of Chapter~1 of \cite{CL55})
 that the IVP of \eqref{eqn:dsys} has a unique solution.
Given a solution $u_n(t)=(u_1^n(t),\ldots, u_n^n(t))$ to the IVP of \eqref{eqn:dsys},
 we define an $L^2(I)$-valued function $\mathbf{u}_n:\Rset\to L^2(I)$ as
\begin{equation*}
\mathbf{u}_n(t) = \sum^{n}_{k=1} u_k^n(t) \mathbf{1}_{I_k^n},
\end{equation*}
where $\mathbf{1}_{I_k^n}$ represents the characteristic function of $I_k^n$, $k\in[n]$.
Let $\|\cdot\|$ denote the norm in $L^2(I)$.
In our setting as stated in Section~1,
 we slightly modify the arguments given in the proof of Theorem~2.3 of \cite{IY23}
 or Theorem~3.1 of \cite{M19} to obtain the following.

\begin{thm}
\label{thm:2b}
If $\mathbf{u}_{n}(t)$ is the solution to the IVP of \eqref{eqn:dsys} with the initial condition
\[
\lim_{n\to\infty}\|\mathbf{u}_n(0)-\mathbf{u}(0)\|=0\quad\mbox{a.s.},
\]
then for any $T > 0$ we have
\[
\lim_{n \rightarrow \infty}\max_{t\in[0,T]}\|\mathbf{u}_n(t)-\mathbf{u}(t)\|=0\quad\mbox{a.s.},
\]
where $\mathbf{u}(t)$ represents the solution
 to the IVP of the CL \eqref{eqn:csys}.
\end{thm}

Note that the statements of Theorem~\ref{thm:2b} and the remainings are valid
 for undirected random graphs 
 since $w_{kj}^n$, $k,j\in[n]$, may be only assumed
 to be independently distributed in $j\in[n]$ for each $k\in[n]$.

For $\theta\in\Sset^1$,
 let $\boldsymbol{\theta}$ represent the constant function $u=\theta$ in $L^2(I)$.
If $\bar{\mathbf{u}}_n(t)$ is a solution to the KM \eqref{eqn:dsys},
 then so is $\bar{\mathbf{u}}_n(t)+\boldsymbol{\theta}$ for any $\theta\in\Sset^1$.
Similarly, if $\bar{\mathbf{u}}(t)$ is a solution to the CL \eqref{eqn:csys},
 then so is $\bar{\mathbf{u}}(t)+\boldsymbol{\theta}$ for any $\theta\in\Sset^1$.
Let $\U_n=\{\bar{\mathbf{u}}_n(t)+\boldsymbol{\theta}\mid\theta\in\Sset^1\}$
 and $\U=\{\bar{\mathbf{u}}(t)+\boldsymbol{\theta}\mid\theta\in\Sset^1\}$
 denote the families of solutions to \eqref{eqn:dsys} and \eqref{eqn:csys}, respectively.
We say that $\U_n$ (resp. $\U$) is \emph{stable}
 if solutions starting in its (smaller) neighborhood
 remain in its (larger) neighborhood for $t\ge 0$,
 and \emph{asymptotically stable} if $\U_n$ (resp. $\U$) is stable
 and the distance between such solutions and $\U_n$ (resp. $\U$) in $ L^2(I)$
 converges to zero as $t\to\infty$.
We also obtain the following result,
 slightly modifying the proofs of Theorem~2.7 in \cite{IY23}
 and Theorem~2.3 of \cite{Y24a}.

\begin{thm}
\label{thm:2c}
Suppose that the KM \eqref{eqn:dsys} and CL \eqref{eqn:csys}
 have solutions $\bar{\mathbf{u}}_n(t)$ and $\bar{\mathbf{u}}(t)$, respectively, such that
\begin{equation}
\lim_{n\to\infty}\|\bar{\mathbf{u}}_n(t)-\bar{\mathbf{u}}(t)\|=0\quad\mbox{a.s.}
\label{eqn:2c}
\end{equation}
for any $t\in[0,\infty)$.
Then the following hold$:$
{\color{black}
\begin{enumerate}
\setlength{\leftskip}{-1.8em}
\item[\rm(i)]
If for any $\epsilon>0$, there exist $\delta_1>0$
 such that for $n>0$ sufficiently large,
 any solution $u_k^n(t)$, $k\in[n]$, to the KM \eqref{eqn:dsys} with
\[
\min_{\theta\in\Sset^1}|u_k^n(0)-\bar{u}_k^n(0)-\theta|<\delta_1,
\quad k\in[n],
\] 
satisfies
\[
\min_{\theta\in\Sset^1}|u_k^n(t)-\bar{u}_k^n(t)-\theta|<\epsilon,
\quad k\in[n],\quad\mbox{a.s.},
\] 
then $\U$ is stable.
Moreover, if for any $\epsilon>0$, there exists $\delta_2>0$
 such that for $n>0$ sufficiently large,
 any solution $u_k^n(t)$, $k\in[n]$, to the KM \eqref{eqn:dsys} with
\[
\min_{\theta\in\Sset^1}|u_k^n(0)-\bar{u}_k^n(0)-\theta|<\delta_2,\quad k\in[n],\quad\mbox{a.s.}
\] 
converges to $\U_n$, $k\in[n]$,
 then $\U$ is asymptotically stable.
\item[\rm(ii)]
If $\U$ is stable, then for any $\epsilon,T>0$ there exists $\delta>0$
 such that for $n>0$ sufficiently large,
 if $\mathbf{u}_n(t)$ is a solution to the KM \eqref{eqn:dsys} satisfying
\begin{equation*}
\min_{\theta\in\Sset^1}\|\mathbf{u}_n(0)-\bar{\mathbf{u}}_n(0)-\boldsymbol{\theta}\|
 <\delta,
\end{equation*}
then
\begin{equation*}
\min_{\theta\in\Sset^1}\|\mathbf{u}_n(t)-\bar{\mathbf{u}}_n(t)-\boldsymbol{\theta}\|
 <\epsilon\quad\mbox{a.s.}
\end{equation*}
Moreover, if $\U$ is asymptotically stable, then
\begin{equation*}
\lim_{t\to\infty}\lim_{n\to\infty}\min_{\theta\in\Sset^1}
 \|\mathbf{u}_n(t)-\bar{\mathbf{u}}_n(t)-\boldsymbol{\theta}\|=0\quad\mbox{a.s.},
\end{equation*}
where $\mathbf{u}_n(t)$ is any solution to \eqref{eqn:dsys}
 such that $\mathbf{u}_n(0)$ is contained in the basin of attraction for $\U$.
\end{enumerate}}
\end{thm}

\begin{rmk}\
\label{rmk:2a}
\begin{enumerate}
\setlength{\leftskip}{-1.8em}
\item[\rm(i)]
In Theorem~$2.3$ of {\rm\cite{Y24a}} only complete simple graphs were treated
 but Theorem~{\rm\ref{thm:2c}} can be proven similarly
 since its proof relies only on Theorem~$2.2$ of {\rm\cite{Y24a}},
 of which extension to \eqref{eqn:dsys} and \eqref{eqn:csys} is Theorem~{\rm\ref{thm:2b}}.
This is also the case
 in Corollary~$\ref{cor:2a}$ and Theorems~$\ref{thm:2d}$ and $\ref{thm:2e}$.
\item[\rm(ii)]
$\U_n$ may not be stable or asymptotically stable in the KM \eqref{eqn:dsys}
 for $n>0$ sufficiently large even if so is $\U$ in the CL \eqref{eqn:csys}.
In the definition of stability and asymptotic stability of solutions to the CL \eqref{eqn:csys},
 we cannot distinguish two solutions that are different only in a set with the Lebesgue measure zero.
\end{enumerate}
\end{rmk}

We have the following as a corollary of Theorem~\ref{thm:2c},
 without assuming the existence of the solution $\bar{\mathbf{u}}_n(t)$
 to the KM \eqref{eqn:dsys} satisfying \eqref{eqn:2c}
 (see also Corollary~2.6 of \cite{Y24a}).

\begin{cor}
\label{cor:2a}
Suppose that the CL \eqref{eqn:csys} has a solution $\bar{\mathbf{u}}(t)$
 and $\U=\{\bar{\mathbf{u}}(t)+\boldsymbol{\theta}\mid\theta\in\Sset^1\}$ is stable.
Then for any $\epsilon,T>0$ there exists $\delta>0$ such that for $n>0$ sufficiently large,
 if $\mathbf{u}_n(t)$ is a solution to the KM \eqref{eqn:dsys} satisfying
\[
\min_{\theta\in\Sset^1}\|\mathbf{u}_n(0)-\bar{\mathbf{u}}(0)-\boldsymbol\theta\|
<\delta,
\]
then
\[
\max_{t\in[0,T]}\min_{\theta\in\Sset^1}
 \|\mathbf{u}_n(t)-\bar{\mathbf{u}}(t)-\boldsymbol\theta\|
 <\epsilon\quad\mbox{a.s.}
\]
Moreover, if $\U$ is asymptotically stable, then
\[
\lim_{t\to\infty}\lim_{n\to\infty}{\color{black}\min_{\theta\in\Sset^1}}
 \|\mathbf{u}_n(t)-\bar{\mathbf{u}}(t)-\boldsymbol\theta\|=0\quad\mbox{a.s.},
\]
where $\mathbf{u}_n(t)$ is any solution to \eqref{eqn:dsys}
 such that $\mathbf{u}_n(0)$ is contained in the basin of attraction for $\U$.
\end{cor}

Corollary~$\ref{cor:2a}$ says that $\U$ behaves
 as if it is an $($asymptotically$)$ stable family of solutions in the KM \eqref{eqn:dsys}.
Finally, we obtain the following results,
 slightly modifying the proofs of Theorems~2.7 and 2.9 in \cite{Y24a}.
 
\begin{thm}
\label{thm:2d}
Suppose that the hypothesis of Theorem~$\ref{thm:2c}$ holds.
Then the following hold$:$
\begin{enumerate}
\setlength{\leftskip}{-1.8em}
\item[\rm(i)]
If $\U_n$ is unstable a.s. for $n>0$ sufficiently large
 and no stable family of solutions to the KM \eqref{eqn:dsys}
 converges to $\U$ a.s. as $n\to\infty$, then $\U$ is unstable$;$
\item[\rm(ii)]
If $\U$ is unstable, then so is $\U_n$ a.s. for $n>0$ sufficiently large.
\end{enumerate}
\end{thm}

\begin{thm}
\label{thm:2e}
If $\U$ is unstable,
 then for any $\epsilon,\delta>0$ there exists $\tau>0$ such that for $n>0$ sufficiently large
\[
\min_{\theta\in\Sset^1}\|\mathbf{u}_n(\tau)-\bar{\mathbf{u}}(\tau)-\boldsymbol{\theta}\|
>\epsilon\quad\mbox{a.s.},
\]
where $\mathbf{u}_n(t)$ is a solution to the KM \eqref{eqn:dsys} satisfying
\[
\min_{\theta\in\Sset^1}\|\mathbf{u}_n(0)-\bar{\mathbf{u}}(0)-\boldsymbol{\theta}\|
<\delta\quad\mbox{a.s.}
\]
\end{thm}

\begin{rmk}\
\begin{enumerate}
\setlength{\leftskip}{-1.6em}
\item[\rm(i)]
Only under the hypothesis of Theorem~{\rm\ref{thm:2c}},
 $\U$  is not necessarily unstable
  even if $\U_n$ is unstable a.s. for $n>0$ sufficiently large.
Moreover, $\U$ may be asymptotically stable
 even if $\U_n$ is unstable a.s. for $n>0$ sufficiently large.
Such an example was found in {\rm\cite{Y24a}}.
\item[\rm(ii)]
Theorem~$\ref{thm:2e}$ implies that $\U$ behaves
 as if it is an unstable family of solutions in the KM \eqref{eqn:dsys}.
\end{enumerate}
\end{rmk}


\section{Linear Eigenvalue Problem}

We next analyze the eigenvalue problem
 for the linearized equation of the CL \eqref{eqn:csys}.
We first note that if $u(t,x)$ is a solution to the CL \eqref{eqn:csys},
 then so is $u(t,x+x_0)+\theta$ for any $x_0\in I$ and $\theta\in\Sset^1$
 (cf. the statement just before Theorem~\ref{thm:2c}).

Let $q\in\Nset$.
We consider the eigenvalue problem
 for the linear operator $\L:L^2(I)\to L^2(I)$ given by
\begin{align}
\L\phi(x)
=& \int_I W(x,y)\cos(2\pi q(y-x)+\sigma)(\phi(y)-\phi(x))\d y\notag\\
=& p\int_{x-\kappa}^{x+\kappa}\cos(2\pi q(y-x)+\sigma)\phi(y)\d y
-\frac{p\cos\sigma\sin 2\pi q\kappa}{\pi q}\phi(x),
\label{eqn:ep}
\end{align}
which is the linearization of \eqref{eqn:csys} around the solution \eqref{eqn:tsol},
 where we have used 
\begin{equation*}
\int_{x-\kappa}^{x+\kappa}\cos(2\pi q(y-x)+\sigma)\d y=\frac{\cos\sigma\sin 2\pi q\kappa}{\pi q}.
\end{equation*}
We compute the integral in the right-hand side of \eqref{eqn:ep}
 for $\phi(x)=\cos2\pi\ell x$ and $\sin2\pi\ell x$, $\ell\in\Nset$, as
\begin{align*}
&
\int_{x-\kappa}^{x+\kappa}\cos(2\pi q(y-x)+\sigma)\cos 2 \pi\ell y\,\d y\\
&=\left(\left(\chi_1(\kappa;\ell, q)+\frac{\sin 2\pi q\kappa}{\pi q}\right)\cos\sigma\cos 2 \pi\ell x
+\chi_2(\kappa;\ell,q)\sin\sigma\sin 2 \pi\ell x\right)
\end{align*}
and
\begin{align*}
&
\int_{x-\kappa}^{x+\kappa}\cos(2\pi q(y-x)+\sigma)\sin 2 \pi\ell y\,\d y\\
&
=\left(\left(\chi_1(\kappa;\ell, q)+\frac{\sin 2\pi q\kappa}{\pi q}\right)\cos\sigma\sin2 \pi\ell x
-\chi_2(\kappa;\ell,q)\sin\sigma\cos 2 \pi\ell x\right),
\end{align*}
respectively, where
\begin{align*}
\chi_1(\kappa;\ell, q)=\begin{cases}
\displaystyle
\kappa+\frac{\sin 4\pi q\kappa}{4\pi q}-\frac{\sin 2\pi q\kappa}{\pi q} & \mbox{if $\ell= q$;}\\[2ex]
\displaystyle
\frac{\sin 2\pi(\ell-q)\kappa}{2\pi(\ell-q)}+\frac{\sin 2\pi(\ell+q)\kappa}{2\pi(\ell+q)}
 -\frac{\sin 2\pi q\kappa}{\pi q} & \mbox{otherwise}
\end{cases}
\end{align*}
and
\begin{align*}
\chi_2(\kappa;\ell, q)=\begin{cases}
\displaystyle
\kappa-\frac{\sin 4\pi q\kappa}{4\pi q} & \mbox{if $\ell= q$;}\\[2ex]
\displaystyle
\frac{\sin 2\pi(\ell-q)\kappa}{2\pi(\ell-q)}-\frac{\sin 2\pi(\ell+q)\kappa}{2\pi(\ell+q)}
  & \mbox{otherwise}.
\end{cases}
\end{align*}

Obviously, $\phi(x)=1$ is an eigenfunction for the zero eigenvalue.
Moreover,  if $\sigma=0$, then
\[
\phi(x)=\cos 2 \pi\ell x,\quad
\sin 2\pi\ell x
\]
are eigenfunctions for the eigenvalue
\[
\lambda=p\chi_1(\kappa;\ell, q)
\]
for each $\ell\in\Nset$, and if $\sigma\neq 0$, then
\[
\phi(x)=\cos 2 \pi\ell x\pm i\sin 2\pi\ell x
\]
are eigenfunctions for the eigenvalue
\[
\lambda=p\chi_1(\kappa;\ell, q)\cos\sigma\mp ip\chi_2(\kappa;\ell,q)\sin\sigma
\]
for each $\ell\in\Nset$, where the upper or lower signs are taken simultaneously.
These eigenvalues are the only ones of $\L$
 since the Fourier expansion of any function in $L^2(I)$ converges a.e.
 by Carleson's theorem \cite{C66}.

We also see that the linear operator $\L$ has no continuous spectrum
 unlike the mean-field limit \cite{SM91,C15},
 in which a continuous spectrum exists on the imaginary axis.
Indeed, for any $h(x)\in L^2(I)$, the equation
\begin{align*}
&
(\L-\lambda\,\id)\phi(x)\\
&
=p\int_{x-\kappa}^{x+\kappa}\cos(2\pi q(y-x)+\sigma)\phi(y)\d y
-\left(\frac{p\cos\sigma\sin 2\pi q\kappa}{\pi q}+\lambda\right)\phi(x)
=h(x),
\end{align*}
where $\lambda\in\Cset$ is a constant and `$\id$' is the identity operator,
has a solution $\phi(x)\in L^2(I)$
i.e., $\L-\lambda\,\id$ is surjective if $\lambda$ is not an eigenvalue given above,
 since $\phi(x),h(x)\in L^2(I)$ are expanded in Fourier series,
 and the Fourier series of $\phi(x)$ is determined if those of $h(x)$ are given.

Thus, if
\begin{equation}
\chi_1(\kappa;\ell, q)=0
\label{eqn:ezero}
\end{equation}
for some $\ell\in\Nset$, then $\L$ has a zero eigenvalue of geometric multiplicity three
 when $\sigma=0$,
 and a simple zero and a pair of purely imaginary eigenvalues
 when $\sigma\neq 0$ and $\chi_2(\kappa;\ell,q)\neq 0$,
 so that a bifurcation may occur in the CL \eqref{eqn:csys}.
{\color{black}
On the other hand, when $q=0$,
 the zero eigenvalue of the linear operator $\L$ is always simple 
 and the remaining eigenvalues are negative or positive,
 depending on whether $\cos\sigma$ is positive or negative,
 as shown in Appendix~A.
Thus, no bifurcation of the solutions \eqref{eqn:tsol} with $q=0$ occurs
 in the CL \eqref{eqn:csys}.
Henceforth we only consider the case of $q\in\Nset$. 
}

Let
\[
\varphi(\zeta)
=\frac{\sin\zeta}{\zeta}(2-\cos\zeta).
\]
We have
\begin{equation}
\lim_{\eta\to+0}\varphi(\zeta)=1.
\label{eqn:3a}
\end{equation}
Let $\{\zeta_j\}_{j=1}^\infty$ be a strictly increasing sequence
 consisting of all $\zeta=\zeta_j>0$, $j\in\Nset$, that satisfy
\[
\frac{\sin\zeta}{\zeta}=\frac{2\cos^2\zeta-2\cos\zeta-1}{\cos\zeta-2}.
\]
Noting that
\[
\varphi'(\zeta)=\frac{1}{\zeta}\left((\cos\zeta-2)\frac{\sin\zeta}{\zeta}
 -(2\cos^2\zeta-2\cos\zeta-1)\right),
\]
we see that $\varphi(\zeta)$ only has local maxima at $\zeta=\zeta_{2j-1}$,
 and local minima at $\zeta=\zeta_{2j}$, $j\in\Nset$.
We have the following (see Appendix~B for a proof).

\begin{lem}
\label{lem:3a}
$\zeta_j\in((j-1)\pi,j\pi)$ for $j\in\Nset$.
\end{lem}

Since $\zeta_1=1.39535\ldots<\pi$, $\zeta_2=4.18392\ldots>\pi$,
 $\varphi(\zeta_1)=1.28815\ldots>1$ and $\varphi(\zeta_2)=-0.51688\ldots<1$,
 it follows from Lemma~\ref{lem:3a} and \eqref{eqn:3a}
 that there exists a unique root of $\varphi(\zeta)=1$ on $(0,\zeta_2)$.
We write it as $\zeta=\zeta_0$ and estimate
\[
\zeta_0=2.1391\ldots<\pi.
\]

\begin{prop}\
\label{prop:3a}
\begin{enumerate}
\setlength{\leftskip}{-1.6em}
\item[\rm(i)]
As $\kappa\to\tfrac{1}{2}-0$,
 $\chi_1(\kappa;q, q)\to\tfrac{1}{2}$
 and $\chi_1(\kappa;\ell,q)\to0$ for $\ell\neq q$,
 while $\chi_1(\kappa;\ell, q)\to0$ as $\kappa\to+0$ for any $\ell,q\in\Nset$.
\item[\rm(ii)]
$\chi_1(\kappa; q, q)<0$ for $\kappa\in(0,\zeta_0/2\pi q)$
 and $\chi_1(\kappa; q, q)>0$ for  $\kappa\in(\zeta_0/2\pi q,\tfrac{1}{2})$.
\item[\rm(iii)]
For any $\ell\in\Nset$,
 $\chi_1(\kappa; \ell, q)<0$ when $\kappa>0$ is sufficiently small.
\item[\rm(iv)]
{\color{black}
If $\ell\ge2q$,} 
 then
\begin{equation}
\chi_1(\kappa;q, q)>\chi_1(\kappa;\ell, q).
\label{eqn:prop3a}
\end{equation}
\item[\rm(v)]
If $\chi_1(\kappa; \ell, q)=0$ and $2\ell\kappa$ or $2q\kappa\not\in\Nset$,
 then $\chi_2(\kappa; \ell, q)\neq 0$.
\end{enumerate}
\end{prop}

\begin{proof}
It is easy to show part~(i).
By \eqref{eqn:3a},
 we have $\varphi(\zeta)>1$ on $(0,\zeta_0)$, so that
\[
\chi_1(\kappa;q,q)=\kappa(1-\varphi(2\pi q\kappa))<0
\]
for $\kappa\in(0,\zeta_0/2\pi q)$.
On the other hand, since for $\zeta>\pi$
\[
|\varphi(\zeta)|\le\frac{|\sin\zeta|}{\zeta}(2+|\cos\zeta|)\le\frac{3}{\pi}<1,
\]
we have
\[
\chi_1(\kappa;q, q)=1-\varphi(2\pi q\kappa)>0
\]
when $2\pi q\kappa>\pi$.
Since the above inequality also holds
 for $2\pi q\kappa\in(\zeta_0,\zeta_2)$ by Lemma~\ref{lem:3a},
 we obtain part~(ii).
 
We turn to part~(iii).
It follows from part~(ii) when $\ell=q$.
So we assume that $\ell\neq q$.
We compute
\[
\frac{\d\chi_1}{\d\kappa}(0;\ell,q)=0,\quad
\frac{\d^2\chi_1}{\d\kappa^2}(0;\ell,q)=0
\]
and
\[\frac{\d^3\chi_1}{\d\kappa^3}(0;\ell,q)
=-4\pi^2(\ell-q)^2-4\pi^2(\ell+q)^2+8\pi^2 q^2=-8\pi^2\ell^2<0.
\]
This yields part~(iii) since $\chi_1(0;\ell,q)=0$.

Let $\psi_1(\zeta)=\sin\zeta/\zeta$ as in Appendix~B.
We easily show the following.

\begin{lem}
There exists a monotonically increasing positive sequence $\{\bar{\zeta}_j\}_{j=1}^\infty$
 such that $\bar{\zeta}_j\in(j\pi,(j+1)\pi)$,
 $\psi_1(\zeta)$ is monotonically  increasing or decreasing
 on $(\bar{\zeta}_j,\bar{\zeta}_{j+1})$
 and has a local minimum or maximum at $\zeta=\bar{\zeta}_j$,
 depending on whether $j$ is odd or even.
Moreover,
\[
\psi_1(\bar{\zeta}_1)<\psi_1(\bar{\zeta}_3)<\cdots<0
 <\cdots<\psi_1(\bar{\zeta}_4)<\psi_1(\bar{\zeta}_2).
\]
\end{lem}

We turn to the proof of part~(iv) and assume that $\ell\ge 2q$.
We have
\begin{align*}
\chi_1(\kappa;q,q)-\chi_1(\kappa;\ell,q)
=&\frac{\sin4\pi q\kappa}{4\pi q} 
 -\frac{\sin2\pi(\ell+q)\kappa}{2\pi(\ell+q)}
 +\kappa-\frac{\sin2\pi(\ell-q)\kappa}{2\pi(\ell-q)}\\
>&
\frac{\sin4\pi q\kappa}{4\pi q} 
 -\frac{\sin2\pi(\ell+q)\kappa}{2\pi(\ell+q)}
\end{align*}
for $\kappa>0$, since $\psi_1(\zeta)<1$ for $\zeta>0$.
In addition, since $\psi_1(\zeta)$ is monotonically decreasing on $[0,\pi]$ and
\begin{equation}
\psi_1(\tfrac{4}{5}\pi)=0.23387\ldots
 >\psi_1(\bar{\zeta}_2)=0.12837\ldots\ge\psi_1(\zeta)\quad
 \mbox{for $\zeta>\pi$},
\label{eqn:prop3a1}
\end{equation}
we have
\[
\frac{\sin4\pi q\kappa}{4\pi q\kappa} 
 -\frac{\sin2\pi(\ell+q)\kappa}{2\pi(\ell+q)\kappa}>0,
\]
so that Eq.~\eqref{eqn:prop3a} holds,
 when $4\pi q\kappa\le \tfrac{4}{5}\pi$.
On the other hand, 
 since $\psi_1(\zeta)\ge0$ for $\zeta\le\pi$ and
\[
\psi_1(\zeta)<\psi_1({\color{black}\tfrac{2}{5}}\pi)=0.75682\ldots
\]
for $\zeta>\tfrac{2}{5}\pi$, we have
\[
\frac{\sin4\pi q\kappa}{4\pi q\kappa} 
 -\frac{\sin2\pi(\ell+q)\kappa}{2\pi(\ell+q)\kappa}
 -\frac{\sin2\pi(\ell-q)\kappa}{2\pi(\ell-q)\kappa}>-0.885201\ldots
\]
by \eqref{eqn:prop3a1}, so that Eq.~\eqref{eqn:prop3a} holds,
 when $4\pi q\kappa\in(\tfrac{4}{5}\pi,\pi]$.
Moreover, since
\[
\psi_1(\zeta)\ge\psi_1(\bar{\zeta}_1)=-0.21723\ldots
\]
for $\zeta>\pi$, and
\[
\psi_1(\zeta)<\psi_1(\tfrac{1}{2}\pi)=0.63661\ldots
\]
for $\zeta>\tfrac{1}{2}\pi$, we have
\[
\frac{\sin4\pi q\kappa}{4\pi q\kappa} 
 -\frac{\sin2\pi(\ell+q)\kappa}{2\pi(\ell+q)\kappa}
 -\frac{\sin2\pi(\ell-q)\kappa}{2\pi(\ell-q)\kappa}>-0.98222\ldots
\]
by \eqref{eqn:prop3a1}, so that Eq.~\eqref{eqn:prop3a} holds,
 when $4\pi q\kappa>\pi$.
Thus, we obtain part~(iv).

We turn to part~(v).
We easily see that $\chi_2(\kappa;q,q)\neq 0$ for $\kappa\neq 0$
 since if $\chi_2(\kappa;q,q)=0$, then
\[
\sin4\pi q\kappa=4\pi q\kappa,
\]
which yields $\kappa=0$.
So we assume that $\ell\neq q$
 and that $\chi_j(\kappa;\ell,q)=0$, $j=1,2$.
Then
\[
\frac{\sin 2\pi(\ell-q)\kappa}{2\pi(\ell-q)}
=\frac{\sin 2\pi(\ell+q)\kappa}{2\pi(\ell+q)}
=\frac{\sin 2\pi q\kappa}{2\pi q}.
\]
Letting $\zeta=2\pi q\kappa\neq 0$ and $\zeta'=2\pi \ell\kappa\neq 0$,
 we rewrite the above relation as
\[
\frac{\sin(\zeta-\zeta')}{\zeta-\zeta'}
=\frac{\sin(\zeta+\zeta')}{\zeta+\zeta'}
=\frac{\sin\zeta}{\zeta}=c,
\]
where $c$ is some constant.
From the last equality in the above equation
 we obtain $\sin\zeta=c\zeta$
 and consequently $\cos\zeta=\pm\sqrt{1-c^2\zeta^2}$,
 so that from the other equalities
\begin{equation}
\begin{split}
&
c\zeta\cos\zeta'\mp\sqrt{1-c^2\zeta^2}\sin\zeta'=c(\zeta-\zeta'),\\
&
c\zeta\cos\zeta'\pm\sqrt{1-c^2\zeta^2}\sin\zeta'=c(\zeta+\zeta'),
\end{split}
\label{eqn:prop3a2}
\end{equation}
where the upper or lower signs are taken simultaneously.
Hence, if $c\neq 0$, then $\cos\zeta'=1$. so that
\[
c\zeta=c(\zeta-\zeta')=c(\zeta+\zeta'),
\]
which yields $c=0$.
So we have $c=0$ so that by \eqref{eqn:prop3a2}
 $\sin\zeta,\sin\zeta'=0$, i.e., $2q\kappa,2\ell\kappa\in\Nset$.
Thus we complete the proof.
\end{proof}

\begin{rmk}\
\label{rmk:3a}
\begin{enumerate}
\setlength{\leftskip}{-1.8em}
\item[\rm(i)]
It follows from Proposition~{\rm\ref{prop:3a}(ii)} and {\rm(iv)} that
 $\chi_1(\kappa;\ell,q)<0$ for $\ell\ge 2q$ if $\kappa<\zeta_0/2\pi q$.
Moreover, if $\kappa<\zeta_0/2\pi q$ and $\chi_1(\kappa;\ell,q)<0$ for $\ell<2q$,
 then $\chi_1(\kappa;\ell,q)<0$ for any $\ell\in\Nset$
 so that the one-parameter family \eqref{eqn:tsol} of $q$-twisted solutions
 is linearly stable.
\item[\rm(ii)]
From Proposition~{\rm\ref{prop:3a}(iii)}
 we see that the one-parameter family \eqref{eqn:tsol} of $q$-twisted solutions
 is linearly stable when $\kappa>0$ is sufficiently small. 
So there are many stable $q$-twisted solutions near $\kappa=0$.
See also {\rm\cite{GHM12}} for such multi-stability of twisted solutions.
\end{enumerate}
\end{rmk}

\begin{figure}
\includegraphics[scale=0.54]{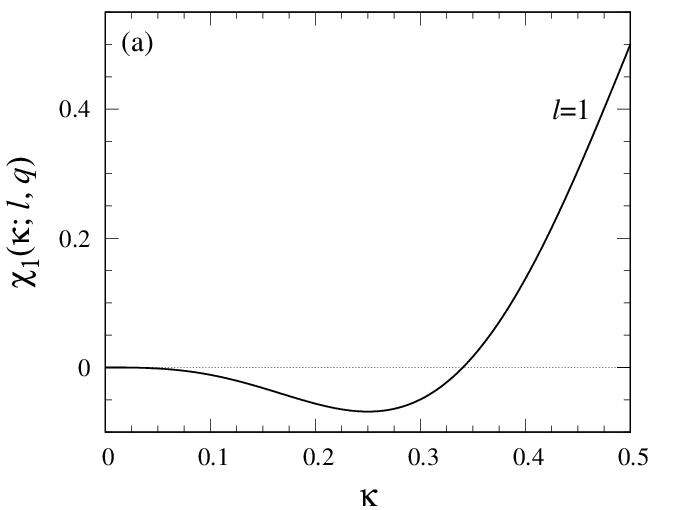}\
\includegraphics[scale=0.54]{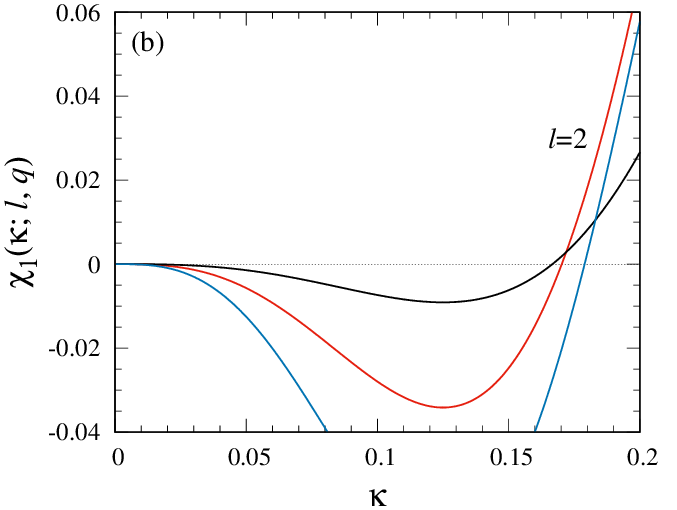}\\[1ex]
\includegraphics[scale=0.54]{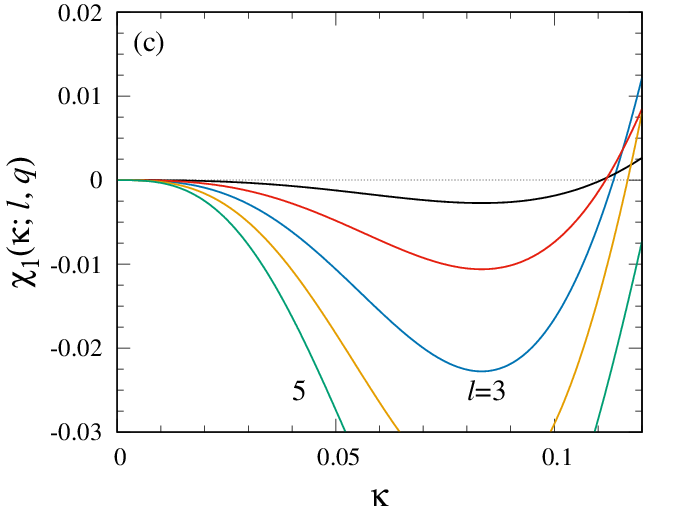}\
\includegraphics[scale=0.54]{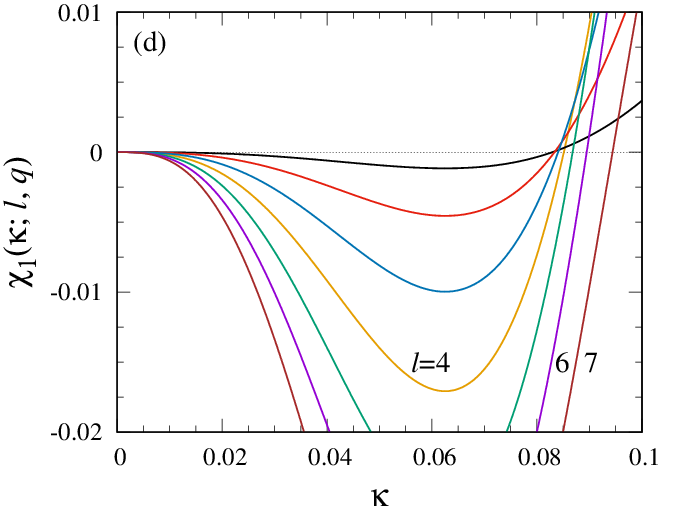}
\caption{Dependence of $\chi_1(\kappa;l,q)$ on $\kappa$ for $l<2q$:
(a) $q=1$; (b) $q=2$; (c) $q=3$; (d) $q=4$.
It is plotted as the lines of which color is black for $l=1$, red for $l=2$, blue for $l=3$,
 orange for $l=4$, green for $l=5$, purple for $l=6$ and brown for $l=7$.
\label{fig:3a}}
\end{figure}

Figure~\ref{fig:3a} displays the dependence of $\chi_1(\kappa;\ell,q)$, $\ell<2q$, on $\kappa$
 for $q\in[4]$.
We see that the graphs of $\chi_1(\kappa;\ell,q)$ intersect the zero-axis
 and bifurcations may occur at these zeros of $\kappa$.
In particular, the graph $\chi_1(\kappa;1,q)$ first intersects the zero-axis
 when $\kappa$ increases from $\kappa=0$ for $q\in[4]$,
 and the one-parameter family \eqref{eqn:tsol} of $q$-twisted solutions
 is linearly stable if $\kappa$ is smaller than the zero,
 as stated in Remark~\ref{rmk:3a}(i).


\section{Bifurcations}

In this section, we take $\kappa$ as a control parameter
 and analyze bifurcations of the twisted state solutions \eqref{eqn:tsol}
 in the CL \eqref{eqn:csys}.
Let $\kappa_{\ell q}$ denote the value of $\kappa$
 satisfying \eqref{eqn:ezero}, i.e., $\chi_1(\kappa_{\ell q};\ell,q)=0$.
It follows from the analysis of Section~3 that
 bifurcations occur at $\kappa=\kappa_{\ell q}$, $\ell\in\Nset$.
In particular,  when $\sigma=0$,
 the related linear operator $\L$ always has a zero eigenvalue,
 which is especially of geometric multiplicity three
 and very degenerate at $\kappa=\kappa_{\ell q}$
 due to the translation symmetry stated at the beginning of Section~3.
Henceforth, we assume that $\ell=1$ for simplicity,
Actually, $\kappa_{1q}$ is the smallest of $\kappa_{\ell q}$, $\ell\in\Nset$, for $q\in[4]$,
 as stated above. 
By Proposition~\ref{prop:3a}(ii), $\kappa_{1q}=\zeta_0/2\pi$ for $q=1$.

\subsection{Center manifold reduction}
Consider the case of $\kappa\approx\kappa_{1q}$ for some $q\in\Nset$
 and introduce the parameter $\mu\approx 0$ as
\begin{equation}
\mu=p\bar{\chi}_{1q}'(\kappa-\kappa_{1q}),
\label{eqn:mu}
\end{equation}
where $\bar{\chi}_{1q}'=\displaystyle\frac{\partial\chi_1}{\partial\kappa}(\kappa_{1q};1,q)$.
Let
\begin{equation}
u(t,x)=2\pi qx+\Omega t+\xi_0(t)
 +\sum_{j=1}^\infty(\xi_j(t)\cos 2\pi jx+\eta_j(t)\sin 2\pi jx).
\label{eqn:solex}
\end{equation}
Regarding $\mu$ as a state variable,
 we substitute \eqref{eqn:solex} into \eqref{eqn:csys},
 and integrate the resulting equation from $x=0$ to $1$
 directly or after multiplying it with $\cos 2\pi jx$ or $\sin 2\pi jx$, $j\in\Nset$, to obtain
\begin{equation}
\begin{split}
\dot{\xi}_0=& p\rho_0\sin\sigma(\xi_1^2+\eta_1^2)+\cdots,\\
\dot{\xi}_1=&\mu\xi_1-\nu_1\eta_1
 +p\cos\sigma(-\beta_1(\xi_1^2+\eta_1^2)\xi_1
 +\delta_1(\xi_1\eta_2-\xi_2\eta_1))\\
& +p\sin\sigma(-\beta_2(\xi_1^2+\eta_1^2)\eta_1
 +\delta_2(\xi_1\xi_2+\eta_1\eta_2))+\cdots,\\
\dot{\eta}_1=&\nu_1\xi_1+\mu\eta_1
 -p\cos\sigma(\beta_1(\xi_1^2+\eta_1^2)\eta_1
+\delta_1(\xi_1\xi_2+\eta_1\eta_2))\\
& +p\sin\sigma(\beta_2(\xi_1^2+\eta_1^2)\xi_1
 +\delta_2(\xi_1\eta_2-\xi_2\eta_1))+\cdots,\\
\dot{\xi}_2=&\mu_2\xi_2-\nu_2\eta_2
 -2p\rho_1\xi_1\eta_1\cos\sigma
 +p\rho_2\sin\sigma(\xi_1^2-\eta_1^2)+\cdots,\\
\dot{\eta}_2=&\nu_2\xi_2+\mu\eta_2
 +p\rho_1\cos\sigma(\xi_1^2-\eta_1^2)
 +2p\rho_2\xi_1\eta_1\sin\sigma+\cdots,\\
\dot{\xi}_j=&\mu_j\xi_j-\nu_j\eta_j+\cdots,\quad
\dot{\eta}_j=\nu_j\xi_j+\mu_j\eta_j+\cdots,\quad
j\neq1,2,\\
\dot{\mu}=&0,
\end{split}
\label{eqn:ifex}
\end{equation}
where  `$\cdots$' represents higher-order terms of
\[
O\left(\sqrt{\xi_0^8+\xi_1^8+\eta_1^8+\xi_2^4+\eta_2^4
 +\sum_{j=3}^\infty(\xi_j^2+\eta_j^2)^{4/3}+\mu^4}\right)
\]
for the first, second and third equations of \eqref{eqn:ifex} ,
\[
O\left(\sqrt{\xi_0^6+\xi_1^6+\eta_1^6+\sum_{j=2}^\infty(\xi_j^4+\eta_j^4)+\mu^4}\right)
\]
for the other equations, and
\begin{align*}
&
\beta_1=\tfrac{3}{8}a_2(q,0)-\tfrac{1}{2}a_2(q,1)+\tfrac{1}{8}a_2(q,2),\quad
\beta_2=\tfrac{1}{4}a_1(q,1)-\tfrac{1}{8}a_1(q,2),\\
&
\delta_1=a_1(q,1)-\tfrac{1}{2}a_1(q,2),\quad
\delta_2=\tfrac{1}{2}a_2(q,0)-\tfrac{1}{2}a_2(q,2),\\
&
\rho_0=\tfrac{1}{2}(a_2(q,0)-a_2(q,1)),\quad
\rho_1=\tfrac{1}{2}a_1(q,1)-\tfrac{1}{4}a_1(q,2),\\
&
\rho_2=\tfrac{1}{4}a_2(q,0)-\tfrac{1}{2}a_2(q,1)+\tfrac{1}{4}a_2(q,2),\\
&
\mu_j=p\chi_1(\kappa_{1q};j,q)\cos\sigma,\quad
j\in\Nset\setminus\{1\},\\
&
\nu_j=p\chi_2(\kappa_{1q};j,q)\sin\sigma,\quad
j\in\Nset,
\end{align*}
with
\begin{align*}
&
a_1(q,j)=\begin{cases}
\displaystyle
\frac{\sin4\pi q\kappa_{1q}}{4\pi q}-\kappa_{1q}
& \mbox{for $j=q$};\\[2ex]
\displaystyle
\frac{q\sin2\pi j\kappa_{1q}\cos2\pi q\kappa_{1q}-j\cos2\pi j\kappa_{1q}\sin2\pi q\kappa_{1q}}{\pi(q^2-j^2)}
& \mbox{for $j\neq q$},
\end{cases}\\
&
a_2(q,j)=\begin{cases}
\displaystyle
-\frac{\sin4\pi q\kappa_{1q}}{4\pi q}-\kappa_{1q}
& \mbox{for $j=q$};\\[2ex]
\displaystyle
\frac{j\sin2\pi j\kappa_{1q}\cos2\pi q\kappa_{1q}-q\cos2\pi j\kappa_{1q}\sin2\pi q\kappa_{1q}}{\pi(q^2-j^2)}
& \mbox{for $j\neq q$}.
\end{cases}
\end{align*}
See Appendix~C for the derivation of \eqref{eqn:ifex}.

Suppose that $\mu_j<0$ for any $j\neq 1$.
This situation occurs for $q\in[4]$,
 as seen from Fig.~\ref{fig:3a} and Proposition~\ref{prop:3a}(iv)
 (see also Remark~\ref{rmk:3a}(i)).
Then the origin in the infinite-dimensional system \eqref{eqn:ifex} is an equilibrium
 having a three-dimensional center manifold $W^\c$, even if $\sigma\neq 0$.
Using the standard approach \cite{GH83,HI11,K04} on center manifold theory,
 we obtain the following.

\begin{prop}
\label{prop:4a}
The center manifold is expressed as 
\begin{align*}
W^\c=\{
\xi_2=&p\bar{\xi}_2(\xi_1,\eta_1)+O(3),\\
&
\eta_2=p\bar{\eta}_2(\xi_1,\eta_1)+O(3),\xi_j=O(3),\eta_j=O(3),j>2\}
\end{align*}
near the origin, where $O(k)$ represents higher-order terms of
\[
O\left(\sqrt{\xi_0^{2k}+\xi_1^{2k}+\eta_1^{2k}+\mu^4}\right)
\]
and
\[
\bar{\xi}_2(\xi_1,\eta_1)
=c_1(\xi_1^2-\eta_1^2)+2c_2\xi_1\eta_1,\quad
\bar{\eta}_2(\xi_1,\eta_1)
=-c_2(\xi_1^2-\eta_1^2)+2c_1\xi_1\eta_1
\]
with
\[
c_1
=\frac{p(2\nu_1-\nu_2)\rho_1\cos\sigma-\mu_2\rho_2\sin\sigma}
{\mu_2^2+(2\nu_1-\nu_2)^2},\quad
c_2
=\frac{p\mu_2\rho_1\cos\sigma-(2\nu_1-\nu_2)\rho_2\sin\sigma}
{\mu_2^2+(2\nu_1-\nu_2)^2}.
\]
\end{prop}

Based on Proposition~\ref{prop:4a},
 we apply the center manifold reduction \cite{HI11} to \eqref{eqn:ifex}
 and obtain
\begin{equation}
\begin{split}
\dot{\xi}_0=& p\rho_0\sin\sigma(\xi_1^2+\eta_1^2)+O(4),\\
\dot{\xi}_1=&\mu\xi_1-\nu_1\eta_1\\
&
+p\cos\sigma(-\beta_1(\xi_1^2+\eta_1^2)\xi_1
 +\delta_1(\xi_1\bar{\eta}_2(\xi_1,\eta_1)-\eta_1\bar{\xi}_2(\xi_1,\eta_1)))\\
&
+p\sin\sigma(-\beta_2(\xi_1^2+\eta_1^2)\eta_1
 +\delta_2(\xi_1\bar{\xi}_2(\xi_1,\eta_1)+\eta_1\bar{\eta}_2(\xi_1,\eta_1)))+O(4),\\
\dot{\eta}_1=&\nu_1\xi_1+\mu\eta_1\\
&
-p\cos\sigma(\beta_1(\xi_1^2+\eta_1^2)\eta_1
 +\delta_1(\xi_1\bar{\xi}_2(\xi_1,\eta_1)+\eta_1\bar{\eta}_2(\xi_1,\eta_1)))\\
&
+p\sin\sigma(\beta_2(\xi_1^2+\eta_1^2)\xi_1
 +\delta_2(\xi_1\bar{\eta}_2(\xi_1,\eta_1)-\eta_1\bar{\xi}_2(\xi_1,\eta_1)))+O(4),\\
\dot{\mu}=&0
\end{split}
\label{eqn:cmex}
\end{equation}
on $W^\c$.
See Appendix~B of \cite{Y24b} for the validity of application of the center manifold theory
 on infinite-dimensional dynamical systems \cite{HI11}.
The origin $(\xi_0,\xi_1,\eta_1)=(0,0,0)$ is always an equilibrium in \eqref{eqn:cmex}.
This is because the family \eqref{eqn:tsol} of twisted solutions
 necessarily satisfies the CL \eqref{eqn:csys}. 
 
\subsection{Case of $\sigma=0$}

\begin{figure}
\includegraphics[scale=0.55]{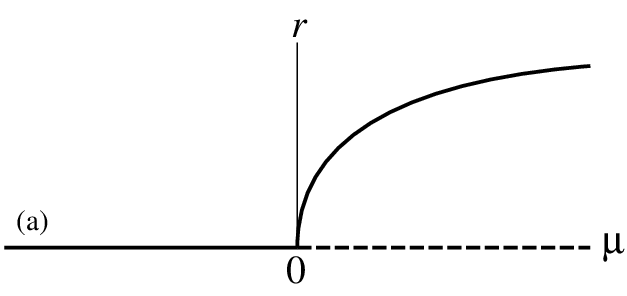}\qquad
\includegraphics[scale=0.55]{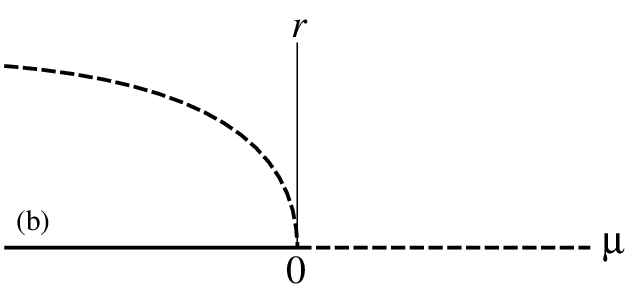}
\caption{Bifurcation diagrams for \eqref{eqn:r}:
(a) $\beta_0>0$; (b) $\beta_0<0$.
The solid and broken lines represent stable and unstable equilibria, respectively.
\label{fig:4a}}
\end{figure}

We set $\sigma=0$, so that $c_1=0$ and $c_2=\rho_1/\chi_1(\kappa_{1q};2,q)$.
The zero eigenvalue is of geometric multiplicity three
 and very degenerate at ${\color{black}\kappa}=\kappa_{1q}$, as stated above.
Letting $r=\sqrt{\xi_1^2+\eta_1^2}\ge 0$,
 we rewrite \eqref{eqn:cmex} as
\begin{equation}
\dot{\xi}_0=O(\sqrt{r^8+\mu^4}),\quad
\dot{r}=\mu r-p\beta_0 r^3
 +O(\sqrt{r^8+\mu^4}),\quad
\dot{\mu}=0,
\label{eqn:r}
\end{equation}
where
\begin{equation}
\beta_0=\beta_1+\frac{p\delta_1\rho_1}{\mu_2}.
\label{eqn:beta0}
\end{equation}
{\color{black}
Recall that, as stated at the beginning of Section~3,
 if $u(t,x)$ is a solution to the CL \eqref{eqn:csys},
 then so is $u(t,x+x_0)+\theta$ for any $x_0\in I$ and $\theta\in\Sset^1$.
Hence,} by the rotation and translation symmetry,
 Eq.~\eqref{eqn:r} must depend only on $r$ and $\mu$,
 even if the higher-order terms are included.
We easily show the following for the second equation of \eqref{eqn:r}
 when $\beta_0>0$ (resp. $\beta_0<0$):
 
\begin{enumerate}
\setlength{\leftskip}{-1.8em}
\item[(i)]
The equilibrium $r=0$ is stable for $\mu<0$ and unstable for $\mu>0$;
\item[(ii)]
There exists another stable (resp. unstable) equilibrium at
\begin{equation}
r=\sqrt{\frac{\mu}{p\beta_0}}+O(\mu)
\label{eqn:ic1}
\end{equation}
for $\mu>0$ (resp. $\mu<0$).
\end{enumerate}
See Fig.~\ref{fig:4a} for the bifurcation diagrams for \eqref{eqn:r}.
From this result,
 we obtain the following for the CL \eqref{eqn:csys}.

\begin{thm}
\label{thm:4a}
Fix $q\in\Nset$
 and suppose that $\beta_0\neq 0$ and $\mu_j=p\chi_1(\kappa_{1q};j,q)<0$ for any $j\neq 1$.
If $\bar{\chi}_{1q}'>0$ $($resp. $\bar{\chi}_{1q}'<0)$,
 then a bifurcation of the one-parameter family of twisted solutions
\[
\U_q=\{u=2\pi qx+\theta\mid\theta\in\Sset^1\}
\]
occurs at $\kappa=\kappa_{1q}$ in the CL \eqref{eqn:csys} with $\sigma=0$
 as follows$:$
\begin{enumerate}
\setlength{\leftskip}{-1.8em}
\item[(i)]
The family $\U_q$ is stable $($resp. unstable$)$ for $\kappa<\kappa_{1q}$
 and unstable $($resp. stable$)$ for $\kappa>\kappa_{1q}$
 near  $\kappa=\kappa_{1q};$
\item[(ii)]
There exists a stable or unstable two-parameter family
 of modulated twisted solutions
\begin{align}
\U_{1q}
 =\biggl\{u=2\pi qx+\sqrt{\frac{\bar{\chi}_{1q}'(\kappa-\kappa_{1q})}{\beta_0}}&\sin(2\pi x+\psi)\notag\\
& +\theta+O(\kappa-\kappa_{1q})\Big|\theta,\psi\in\Sset^1\biggr\}
\label{eqn:thm4a}
\end{align}
for $\kappa>\kappa_{1q}$ or $\kappa<\kappa_{1q}$
 $($resp.  $\kappa<\kappa_{1q}$ or $\kappa>\kappa_{1q})$ near  $\kappa=\kappa_{1q}$,
 depending on whether $\beta_0>0$ or $<0$. 
\end{enumerate}
Here $\beta_0=O(1)$ is given in \eqref{eqn:beta0}.
\end{thm}

\begin{proof}
We rewrite the system \eqref{eqn:cmex} in the rotational frame
 with speed $-\tilde{\omega}<0$ in the $\xi_0$-direction,  so that
\[
\dot{\xi}_0=\tilde{\omega}+O(\sqrt{r^8+\mu^4})
\] 
and take the Poincar\'e section
\begin{equation}
\Sigma=\{(\xi_0,\xi_1,\eta_1,\mu)\in\Sset^1\times\Rset^3\mid\xi_0=0\}.
\label{eqn:Psec}
\end{equation}
Here the other components of \eqref{eqn:cmex} do not change under this transformation
 by its rotation symmetry.
So the corresponding Poincar\'e map
 has an invariant set of fixed points given by \eqref{eqn:ic1}
 and with the same stability type as stated above  on $\Sigma$.
Thus, we obtain the desired result.
\end{proof}

\begin{rmk}\
\label{rmk:4a}
\begin{enumerate}
\setlength{\leftskip}{-1.8em}
\item[\rm(i)]
For each $\ell\ge 2$,
 a bifurcation similar to one detected in Theorem~$\ref{thm:4a}$ also occurs
 at $\kappa=\kappa_{\ell q}$ if $\kappa_{\ell q}\neq\kappa_{j q}$ for any $j\neq\ell$,
 although the two-parameter family of perturbed twisted solutions
 born there is necessarily unstable.
\item[\rm(ii)]
From Fig.~$\ref{fig:3a}$ and Proposition~{\rm\ref{prop:3a}(iv)}
 we see  for $q\le 4$ at least that
 the one-parameter family $\U_q$ of twisted solutions is stable
 when $\kappa<\kappa_{1q}$.
See also Remark~{\rm\ref{rmk:3a}(i)}.
\end{enumerate}
\end{rmk}

\begin{table}
\caption{
Constants appearing in Eq.~\eqref{eqn:beta0} and Theorem~\ref{thm:4a} for $q\in[4]$.
The numbers are rounded up to the fifth decimal point.
\label{tab:4a}}
\begin{tabular}{c|c|c|c|c|c|c|c}
$q$ & $\kappa_{1q}$ &  $\bar{\chi}_{1q}'$
 & $\beta_1$ & $\delta_1$ & $\rho_1$ & $\mu_2/p$ & $\beta_0$\\
\hline
1 & $0.34046$ & $1.65602$
 & $0.01588$ & $-0.34915$ & $-0.17457$ & $-0.12703$ & $-0.46397$\\
2 & $0.16667$ & $0.50000$
 & $-0.01222$ & $-0.03727$ & $-0.01863$ & $-0.00562$ & $-0.13572$\\
3 &$0.110727$ & $0.22949$
 & $ 0.00010$ & $-0.00807$ & $-0.00403$ & $-0.04062$ & $-0.00070$\\
4 & $0.08295$ & $0.13053$ 
 & $0.00002$ & $-0.00263$ & $-0.00132$ & $-0.00019$ & $-0.01818$ 
\end{tabular}
\end{table}

See Table~\ref{tab:4a} for the values of constants
 appearing in Eq.~\eqref{eqn:beta0} and Theorem~\ref{thm:4a} for $q\in[4]$.
In particular, {\color{black}for $q\in[4]$,}
 the two-parameter family $\U_{1q}$ born at the bifurcation is unstable,
{\color{black}so that solutions starting near $\U_{q}$ converge to the other stable states
 which may be $\U_{q'}$ with $q'<q$, synchronized or desynchronized ones
 when the one-parameter family $\U_{q}$ becomes unstable.
}

\subsection{Case of $\sigma\neq 0$}

We next consider the case of $\sigma\neq 0$.
Recall that by Proposition~\ref{prop:3a}(v) $\nu_1\neq 0$ at $\kappa=\kappa_{1q}$.
Letting $\xi_1=r\cos\psi$ and $\eta_1=r\sin\psi$,
 we rewrite \eqref{eqn:cmex} as
\begin{equation}
\begin{split}
\dot{\xi}_0=&p\rho_0 r^2\sin\sigma+O(\sqrt{r^8+\mu^4}),\\
\dot{r}=&\mu r-p\beta_\sigma r^3+O(\sqrt{r^8+\mu^4}),\\
\dot{\psi}=&\nu_1+O(\sqrt{r^2+\mu^2}),\quad
\dot{\mu}=0,
\end{split}
\label{eqn:polar}
\end{equation}
where
\begin{align}
\beta_\sigma
=&\beta_1\cos\sigma+\frac{p}{2(\mu_2^2+(2\nu_1-\nu_2)^2)}
 (\mu_2(\delta_1\rho_1+\delta_2\rho_2)\notag\\
& +\mu_2(\delta_1\rho_1-\delta_2\rho_2)\cos2\sigma
 +(2\nu_1-\nu_2)(\delta_1\rho_2-\delta_2\rho_1)\sin2\sigma).
\label{eqn:beta}
\end{align}
Here by the rotation and translation symmetry,
 Eq.~ \eqref{eqn:polar} must depend only on $r$ and $\mu$,
 even if the higher-order terms are included.
We easily show that a Hopf bifurcation \cite{GH83,HI11,K04} occurs
 for the second and third equations of \eqref{eqn:polar}
 when $\beta_\sigma>0$ (resp. $\beta_\sigma<0$) as follows:
 
\begin{enumerate}
\setlength{\leftskip}{-1.8em}
\item[(i)]
The equilibrium $r=0$ is stable for $\mu<0$ and unstable for $\mu>0$;
\item[(ii)]
There exists a stable (resp. unstable) periodic orbit given by
\begin{equation}
r=\sqrt{\frac{\mu}{p\beta_\sigma}}+O(\mu),\quad
\psi=\nu_1 t+O(\sqrt{\mu}),
\label{eqn:ic2}
\end{equation}
for $\mu>0$ (resp. $\mu<0$).
\end{enumerate}
From this result,
 we obtain the following for the CL \eqref{eqn:csys}.

\begin{thm}
\label{thm:4b}
Fix $q\in\Nset$
 and suppose that $\beta_\sigma\neq 0$ and $\mu_j=p\chi_1(\kappa_{1q};j,q)\cos\sigma<0$ for any $j\neq 1$.
If $\bar{\chi}_{1q}'<0$ $($resp. $\bar{\chi}_{1q}'>0$),
 then a  bifurcation of the one-parameter family of twisted solutions
\[
\tilde{\U}_q=\{u=2\pi qx+\Omega t+\theta\mid\theta\in\Sset^1\}
\]
occurs at $\kappa=\kappa_{1q}$ in the CL \eqref{eqn:csys} with $\sigma\neq 0$
 as follows$:$
\begin{enumerate}
\setlength{\leftskip}{-1.8em}
\item[(i)]
The family $\tilde{\U}_q$
 is stable $($resp. unstable$)$ for $\kappa<\kappa_{1q}$
 and unstable $($resp. stable$)$ for $\kappa>\kappa_{1q}$ near  $\kappa=\kappa_{1q};$
\item[(ii)]
There exists a stable or unstable two-parameter family
 of oscillating twisted solutions
\begin{align}
\tilde{\U}_{1q}
=\biggl\{u=2\pi qx+&\sqrt{\frac{\bar{\chi}_{1q}'(\kappa-\kappa_{1q})}{\beta_\sigma}}
 \sin(2\pi x+\tilde{\psi}(t)+\psi_0)\notag\\
&+\tilde{\Omega}t+\theta+O(\kappa-\kappa_{1q})
 \mid\theta,\psi_0\in\Sset^1\biggr\}
\label{eqn:thm4b}
\end{align}
for $\kappa>\kappa_{1q}$ or $\kappa<\kappa_{1q})$
 $($resp.  $\kappa<\kappa_{1q}$ or $\kappa>\kappa_{1q})$ near  $\kappa=\kappa_{1q}$,
 depending on whether $\beta_\sigma>0$ or $<0$, where
\[
\tilde{\Omega}=\Omega
 +p\rho_0\sin\sigma\left(\frac{\bar{\chi}_{1q}'(\kappa-\kappa_{1q})}{\beta_\sigma}\right)
 +O((\kappa-\kappa_{1q})^2)
\]
and $\tilde{\psi}(t)\in\Sset^1$ is a periodic function whose period is approximately $2\pi/\nu_1$.
\end{enumerate}
Here $\Omega$ and $\beta_\sigma=O(1)$ are given in \eqref{eqn:Omega} and \eqref{eqn:beta},
 respectively.
\end{thm}
 
\begin{proof}
We proceed as in the proof of Theorem~\ref{thm:4a}.
So we rewrite the system \eqref{eqn:cmex} in the rotational frame
 with speed $-\tilde{\omega}$ in the $\xi_0$-direction, so that
\[
\dot{\xi}_0=\tilde{\omega}+p\rho_0\sin\sigma r^2+O(\sqrt{r^8+\mu^4}),
\] 
where $\tilde{\omega}>0$ is sufficiently large,
 and take the Poincar\'e section $\Sigma$ defined by \eqref{eqn:Psec}.
The other components of \eqref{eqn:cmex} do not change under this transformation
 by its rotation symmetry, again.
The corresponding Poincar\'e map
 has an invariant circle which consists of quasi-periodic orbits given by \eqref{eqn:ic2}
 and has the same stability type as stated above, on $\Sigma$.
This yields the desired result.
\end{proof}

\begin{rmk}\
\label{rmk:4b}
\begin{enumerate}
\setlength{\leftskip}{-1.8em}
\item[\rm(i)]
For each $\ell\in\Nset$,
 a bifurcation similar to one detected in Theorem~$\ref{thm:4b}$ also occurs
 at $\kappa=\kappa_{\ell q}$ if $\kappa_{\ell q}\neq\kappa_{j q}$ for any $j\neq\ell$,
 although the two-parameter family of oscillating twisted solutions born there
 is necessarily unstable.
see also Remark~{\rm\ref{rmk:4a}(i)}.
\item[\rm(ii)]
From Fig.~$\ref{fig:3a}$ and Proposition~{\rm\ref{prop:3a}(iv)}
 we see  for $q\le 4$ at least that
 the one-parameter family $\tilde{\U}_q$ of twisted solutions is stable
 when $\kappa<\kappa_{1q}$.
See also Remarks~{\rm\ref{rmk:3a}(i)} and {\rm\ref{rmk:4a}(ii)}.
\end{enumerate}
\end{rmk}

\begin{table}
\caption{Constants appearing in Eq.~\eqref{eqn:beta} and Theorem~\ref{thm:4b} for $q\in[4]$.
The numbers are rounded up to the fifth decimal point.	
\label{tab:4b}}
\begin{tabular}{c|c|c|c|c}
$q$ & $\delta_2$ & $\rho_2$ & $\nu_1/p\sin\sigma$ & $\nu_2/p\sin\sigma$\\
\hline
1 & $-0.06351$ & $0.03176$ & $0.41266$ & $0.12703$\\
2 & $0.04888$ & $-0.02444$ & $0.13783$ & $0.20113$\\
3 & $-0.00039$ & $0.00020$ & $0.06433$ & $0.11253$\\
4 & $-0.03474$ & $0.00005$ & $0.03680$ & $0.06834$
\end{tabular}
\end{table}

\begin{figure}
\includegraphics[scale=0.52]{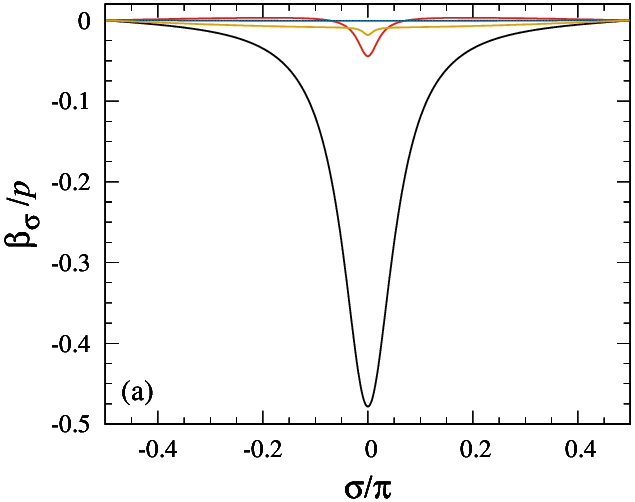}\quad
\includegraphics[scale=0.52]{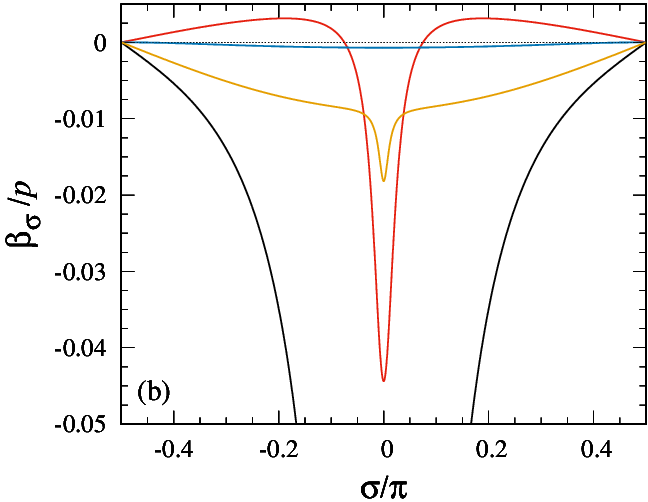}
\caption{Dependence of $\beta_\sigma$ on $\sigma$ for $q\in[4]$:
Figure~(b) is an enlargement of Fig.~(a) in the ordinate axis direction.
The black, red, blue and orange lines
 represent the results for $q=1,2,3$ and $4$, respectively, in both figures.
\label{fig:4b}}
\end{figure}

The values of constants appearing
 in Eq.~\eqref{eqn:beta} and Theorem~\ref{thm:4b}
 are provided in Table~\ref{tab:4b} for $q\in[4]$.
See Table~\ref{tab:4a} for the other constants.
Figure~\ref{fig:4b} shows the dependence of $\beta_\sigma$ on $\sigma$ for $q\in[4]$.
In particular, the two-parameter family $\tilde{\U}_{1q}$ born at the bifurcation
 is stable in some range  of $\sigma$ for $q=2$,
 while it is unstable for any $\sigma\in(-\tfrac{1}{2}\pi.\tfrac{1}{2}\pi)$
 for $q=1,3,4$.


\section{Numerical Simulations}

\begin{figure}[t]
\begin{minipage}[t]{0.495\textwidth}
\begin{center}
\includegraphics[scale=0.3]{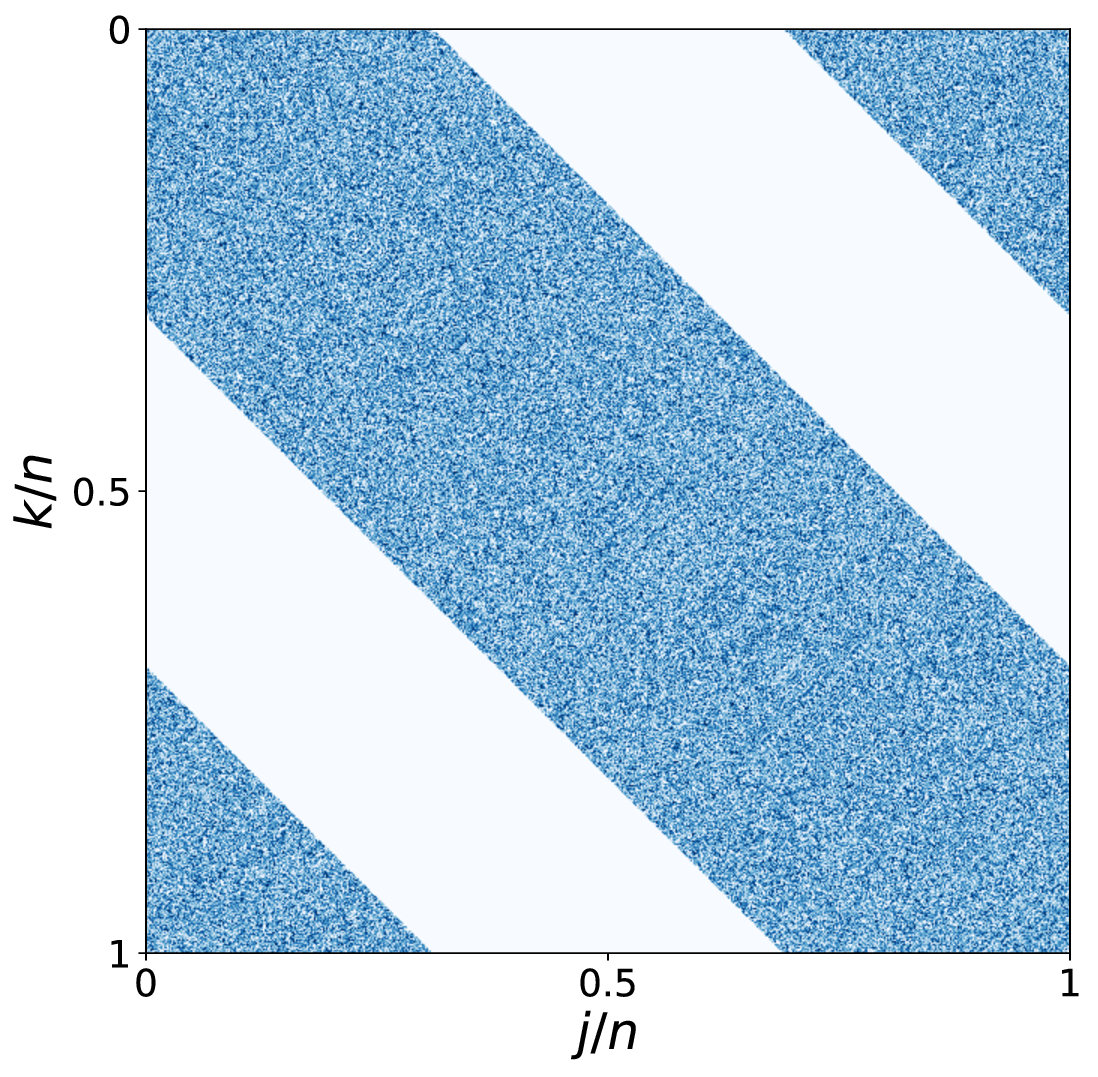}\\
{\footnotesize(a)}
\end{center}
\end{minipage}
\begin{minipage}[t]{0.495\textwidth}
\begin{center}
\includegraphics[scale=0.3]{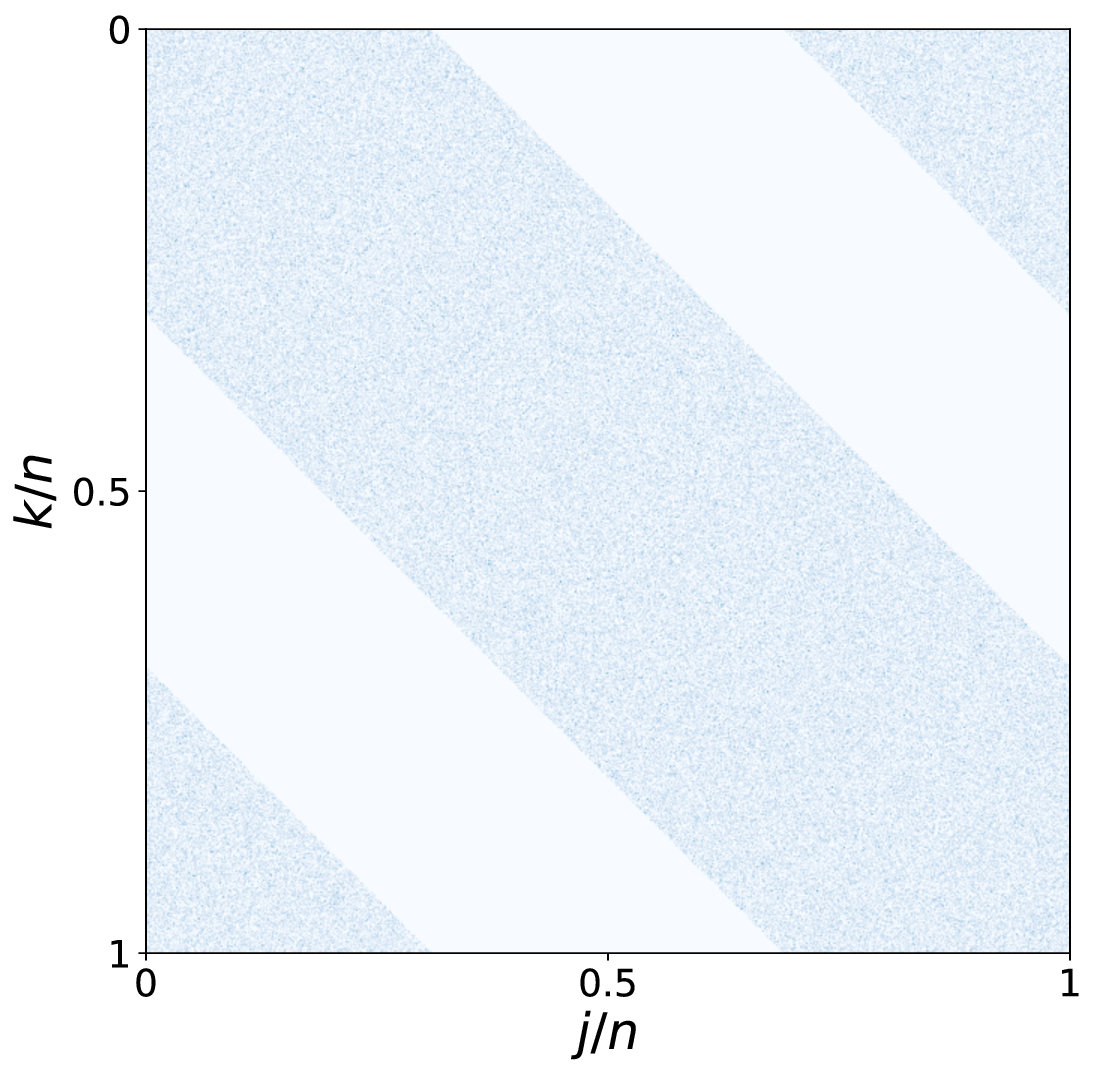}\\
{\footnotesize(b)}
\end{center}
\end{minipage}
\caption{Pixel picture of sampled weight matrices of the random undirected graphs
 in cases (ii) and (iii) for $\kappa=0.31$:
(a) Dense graph for $n=1000$ and $p=0.5$;
(b) sparse graph for $n=2000$,  $p=1$ and $\gamma=0.3$.
The color of the corresponding pixel is blue if $w_{kj}^n=1$
 and it is light blue otherwise.
\label{fig:5a}}
\end{figure}

Finally, we give numerical simulation results for the KM \eqref{eqn:dsys} with
\begin{equation}
\omega=-\frac{p\sin2\pi q\kappa\sin\sigma}{\pi q},
\label{eqn:omega}
\end{equation}
Note that $\omega=0$ for $\sigma=0$
 and that the twisted solution \eqref{eqn:tsol} in the CL \eqref{eqn:csys}
 has $\Omega=0$ by \eqref{eqn:Omega}.
We consider the following three cases
 for the $\lfloor n\kappa\rfloor$-nearest neighbor graph $G_{n}$:
\begin{enumerate}
\setlength{\leftskip}{-1.8em}
\item[(i)]
Deterministic dense graph with $p=1$, i.e.,
\begin{equation*}
w_{kj}=
\begin{cases}
1 & \mbox{if $|k-j|\leq\kappa n$ or $|k-j|\geq (1-\kappa)n$;}\\
0 & \mbox{otherwise;}
\end{cases}
\end{equation*}
\item[(ii)]
Random undirected dense graph in which $w_{kj}^n=1$ with probability
\begin{equation*}
\Pset(j\sim k)=p,\quad
k,j\in[n],
\end{equation*}
if
\begin{equation}
|k-j|\leq\kappa n\mbox{ or }|k-j|\geq (1-\kappa)n;
\label{eqn:con1}
\end{equation}
\item[(iii)]
Random undirected space graph in which $w_{kj}^n=1$ with probability
\begin{equation*}
\Pset(j\sim k)=n^{-\gamma}p,\quad
k,j\in[n],
\end{equation*}
where $\gamma=0.3$, if condition~\eqref{eqn:con1} holds.
\end{enumerate}
In case~(ii) (resp. case (iii)),
 we have $w_{kj}=0$ with probability $1-p$ (resp. $1-n^{-\gamma}p$) or with probability one,
 depending on whether condition~\eqref{eqn:con1} holds or not.
Figure~\ref{fig:5a} provides the weight matrices
 for numerically computed samples of the random undirected dense and sparse graphs
 with $(n,p)=(1000,0.5)$ and $(n,p,\gamma)=(2000,1,0.3)$, respectively, for $\kappa=0.31$.

\begin{figure}[t]
\begin{minipage}[t]{0.495\textwidth}
\begin{center}
\includegraphics[scale=0.27]{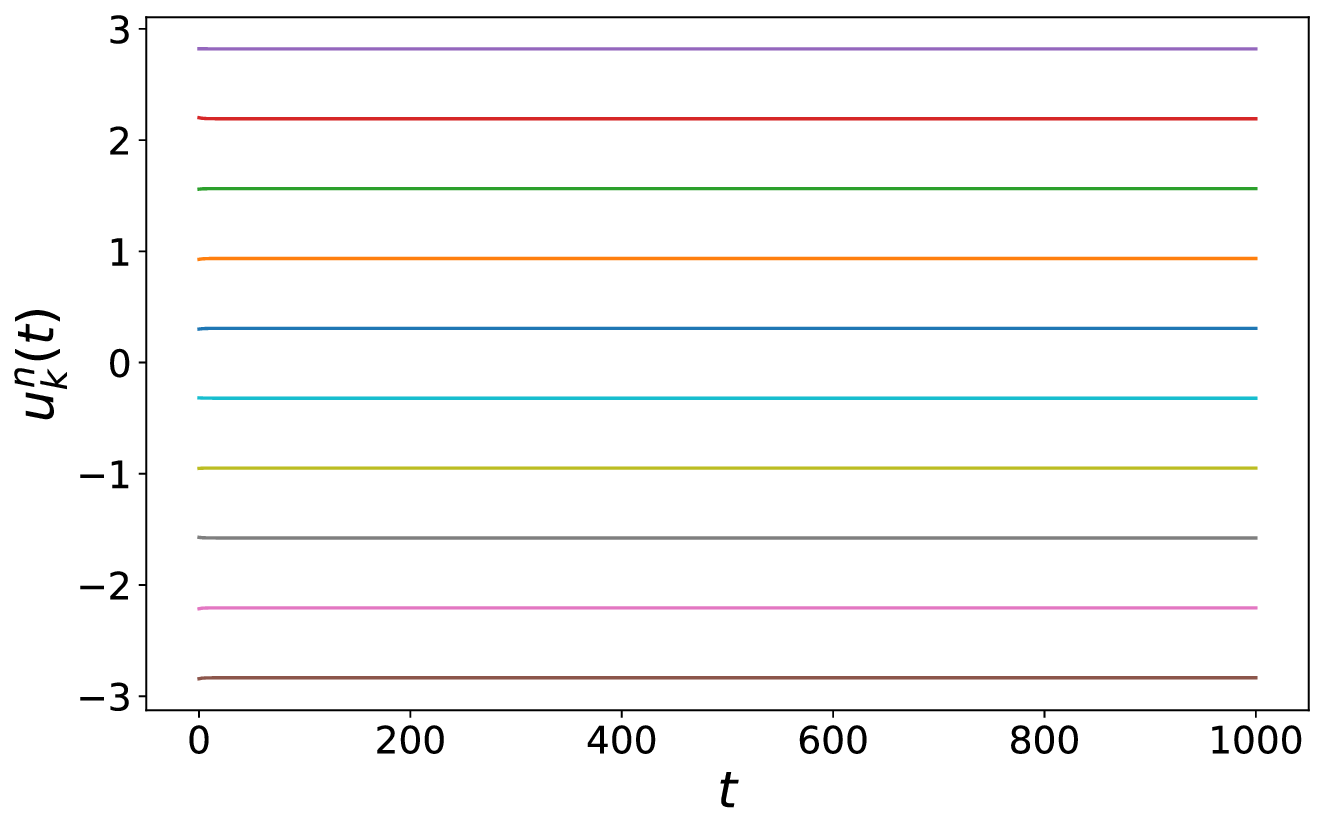}\\[-1ex]
{\footnotesize(a)}
\end{center}
\end{minipage}
\begin{minipage}[t]{0.495\textwidth}
\begin{center}
\includegraphics[scale=0.27]{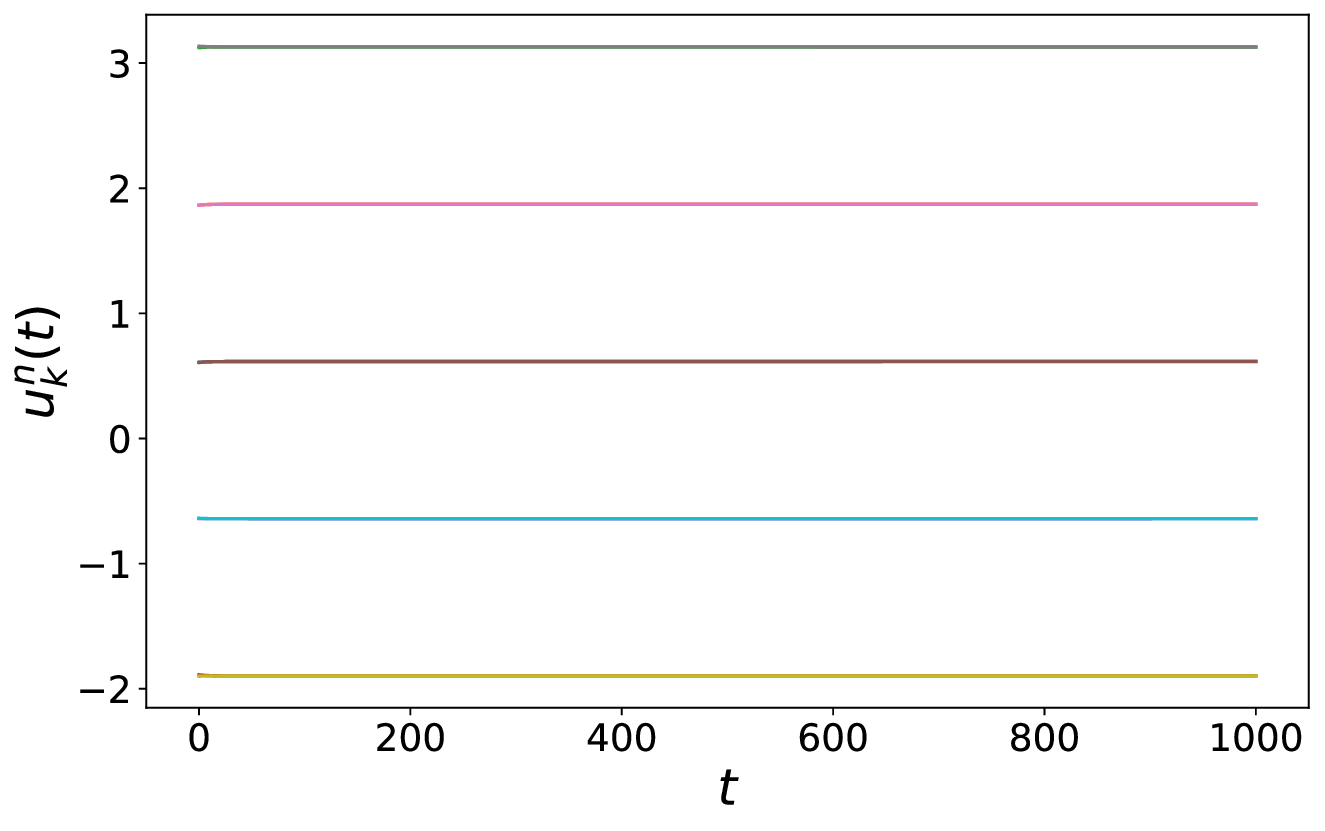}\\[-1ex]
{\footnotesize(b)}
\end{center}
\end{minipage}
\vspace*{1ex}

\begin{minipage}[t]{0.495\textwidth}
\begin{center}
\includegraphics[scale=0.27]{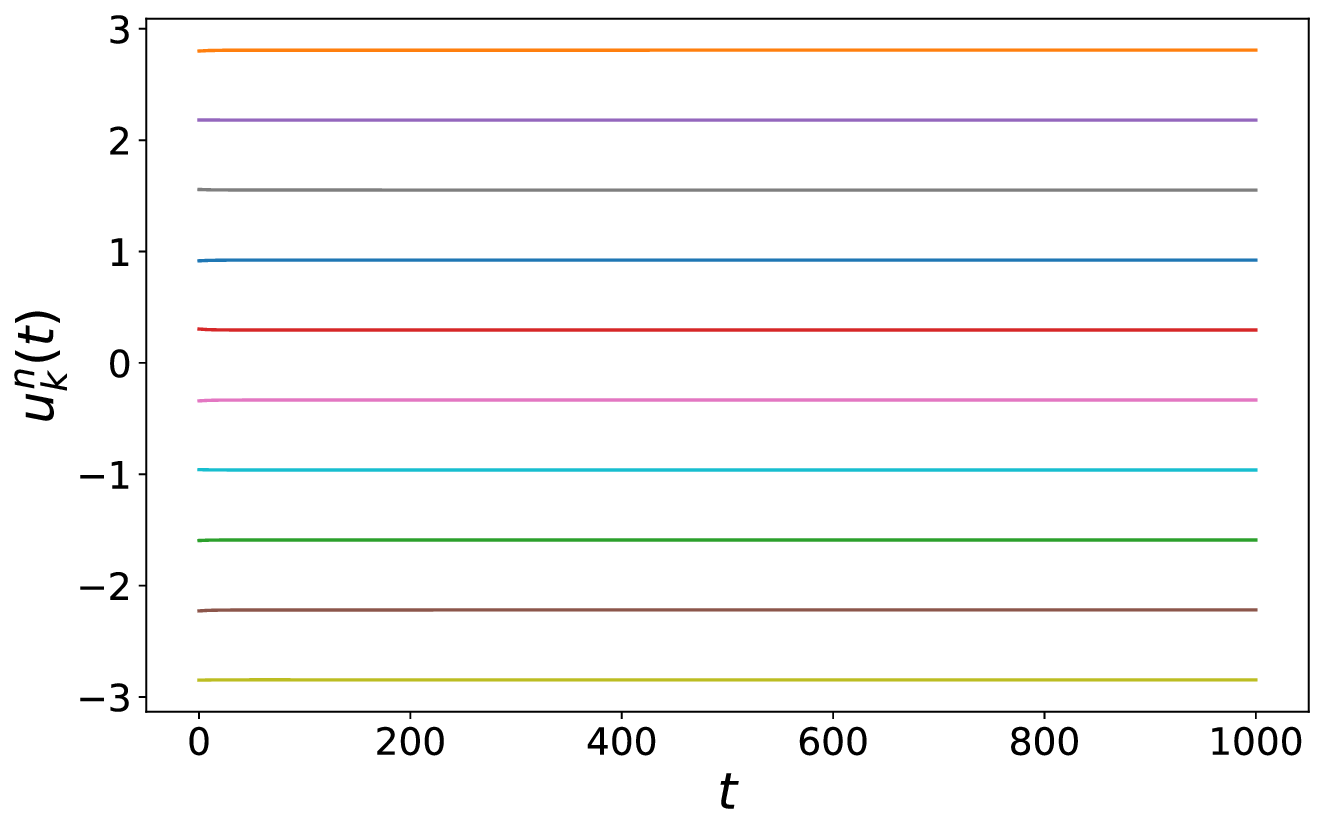}\\[-1ex]
{\footnotesize(c)}
\end{center}
\end{minipage}
\begin{minipage}[t]{0.495\textwidth}
\begin{center}
\includegraphics[scale=0.27]{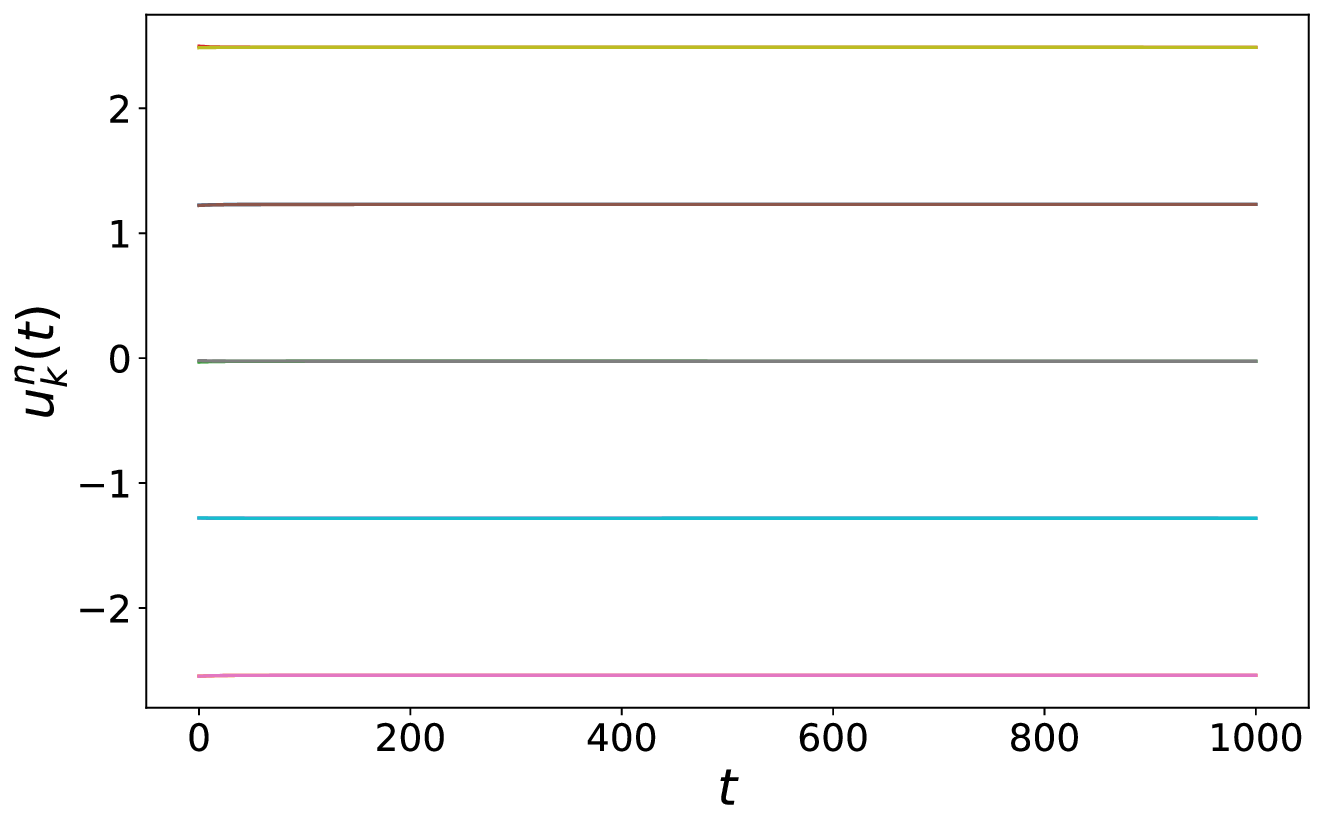}\\[-1ex]
{\footnotesize(d)}
\end{center}
\end{minipage}
\caption{Time histories of the KM \eqref{eqn:dsys}
 with $n=1000$, $\omega=0$ and $\sigma=0$ in case~(i):
(a) $(q,\kappa)=(1,0.31)$;
(b) $(2,0.16)$;
(c)  $(3,0.1)$;
(d) $(4,0.08)$.
They are plotted for every $100$th node (from 50th to 950th).
The five pairs of two lines coincide almost completely in Figs.~(b) and (d) for $q=2,4$.}
\label{fig:5b1}
\end{figure}

\begin{figure}[t]
\begin{center}
\includegraphics[scale=0.27]{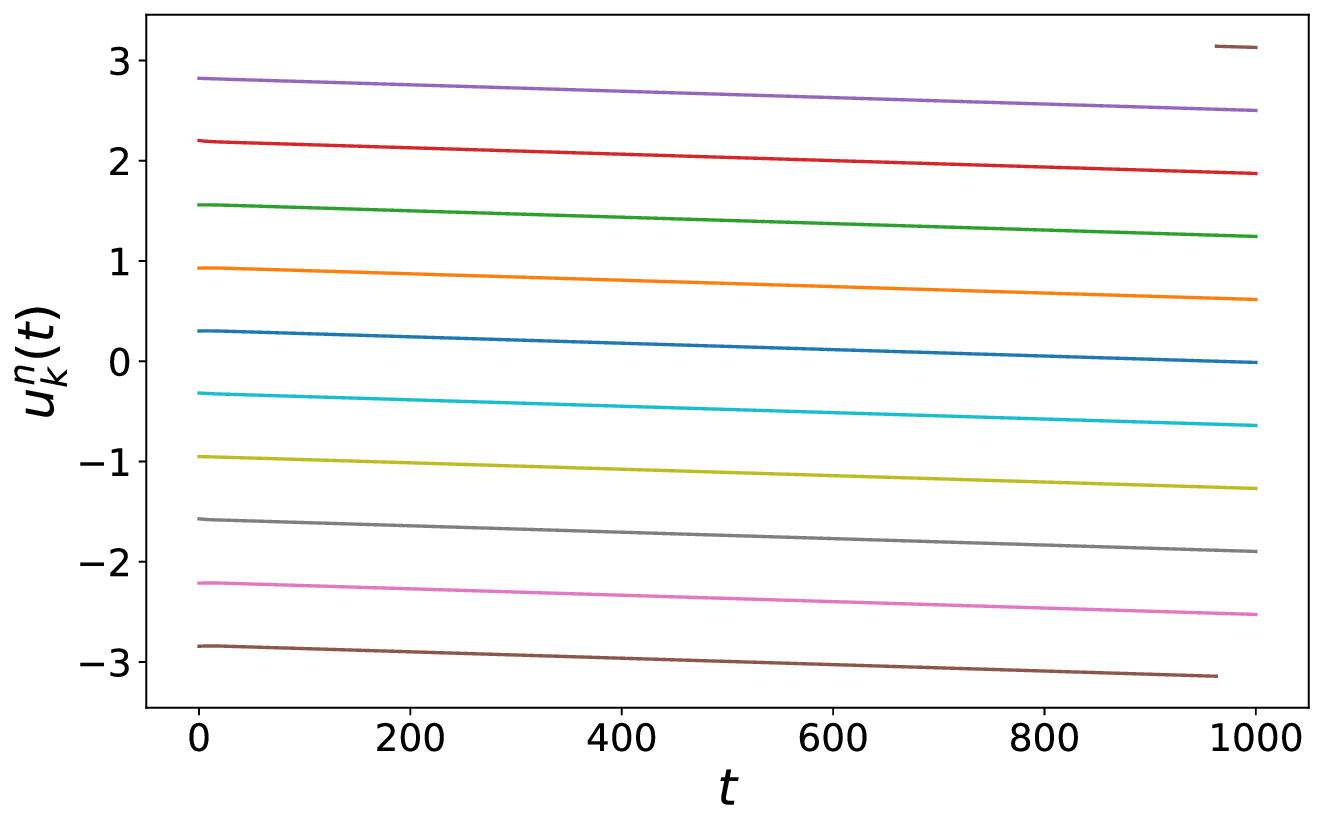}\\[-1ex]
{\footnotesize(a)}
\end{center}
\begin{minipage}[t]{0.495\textwidth}
\begin{center}
\includegraphics[scale=0.27]{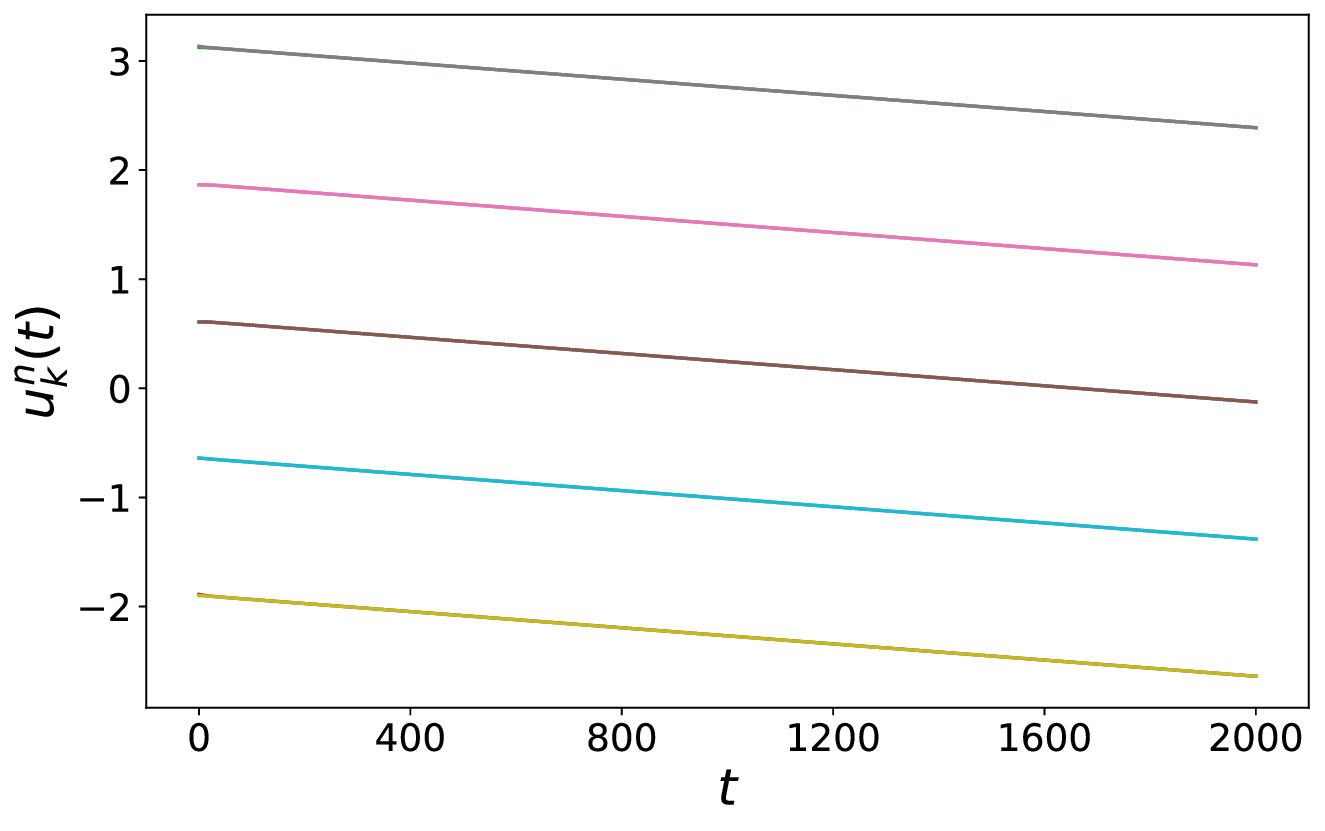}\\[-1ex]
{\footnotesize(b)}
\end{center}
\end{minipage}
\begin{minipage}[t]{0.495\textwidth}
\begin{center}
\includegraphics[scale=0.27]{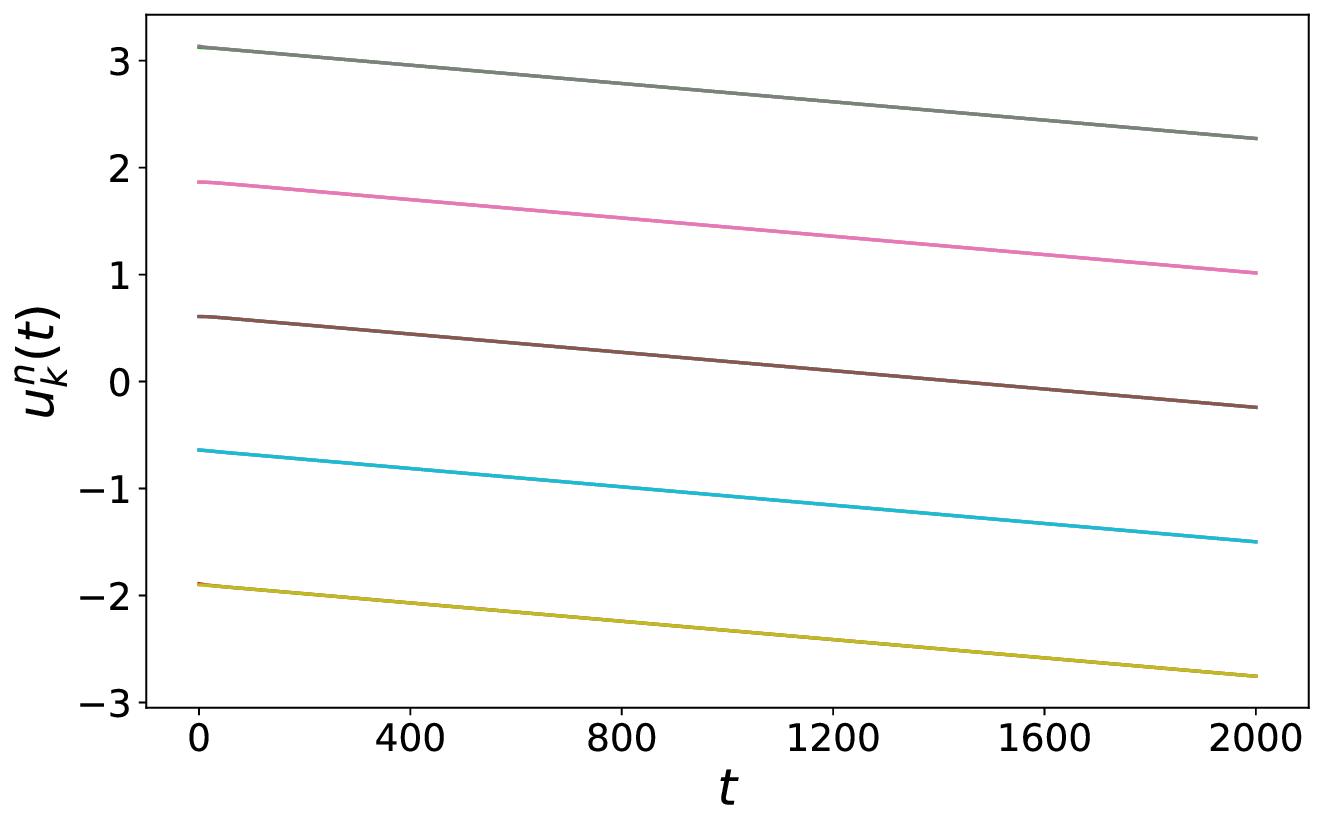}\\[-1ex]
{\footnotesize(c)}
\end{center}
\end{minipage}
\caption{Time histories of the KM \eqref{eqn:dsys}
 with $n=1000$ and $\sigma=\pi/3$ in case~(i):
(a) $(q,\kappa)=(1,0.31)$;
(b) $(2,0.16)$;
(c) $(2,0.166)$.
The natural frequency $\omega\neq 0$ was given by \eqref{eqn:omega}.
See also the caption of Fig.~\ref{fig:5b1}.}
\label{fig:5c1}
\end{figure}

We carried out numerical simulations for the KM \eqref{eqn:dsys},
 using the DOP853 solver \cite{HNW93}, for $q\in[4]$ or $q=1,2$.
We mainly took $n=1000$ but used $n=2000$ 
 for some cases.
The initial values $u_k^n(0)$, $k\in[n]$, were independently randomly chosen
 according to the uniform distribution on $[-10^{-2},10^{-2}]$
 around the $q$-twisted states $u_k^n=2\pi qk/n$, $q\in[4]$.
So if the $q$-twisted states are asymptotically stable,
 then the responses of \eqref{eqn:dsys} are expected to converge to them as $t\to\infty$.

\subsection{Case~(i): Deterministic graphs}

We first give numerical results for case~(i).
Figures~\ref{fig:5b1} and \ref{fig:5c1}
 show the time-histories of every $100$th node (from 50th to 950th)
 for $\sigma=0$ and $\pi/3$, respectively.
Here the initial values were chosen near the twisted state $u_k^n=2\pi qk/n$
 for $q\in[4]$ and $q=1,2$ in Figs.~\ref{fig:5b1} and \ref{fig:5c1}, respectively.
The values of $\kappa$
 are considered to be bigger in Fig.~\ref{fig:5c1}(c)
 and smaller in the other figures than the bifurcation points,
 which are approximated by $\kappa_{1q}$, $q\in[4]$.
We observe that the responses remain near the twisted states
 in Figs.~\ref{fig:5b1} and \ref{fig:5c1},
 and exhibit slow rotation due to a finite size effect in Fig.~\ref{fig:5c1}
 although $\Omega=0$ for the twisted solutions \eqref{eqn:tsol}
 in the CL \eqref{eqn:csys}.
As predicted by Theorems~\ref{thm:4a} and \ref{thm:4b} and Table~\ref{tab:4a}
 with the assistance of Corollary~\ref{cor:2a},
 the oscillating twisted solutions approximately given by \eqref{eqn:thm4b}
 were observed for $q=2$ as will be clarified below.

\begin{figure}[t]
\begin{minipage}[t]{0.495\textwidth}
\begin{center}
\includegraphics[scale=0.3]{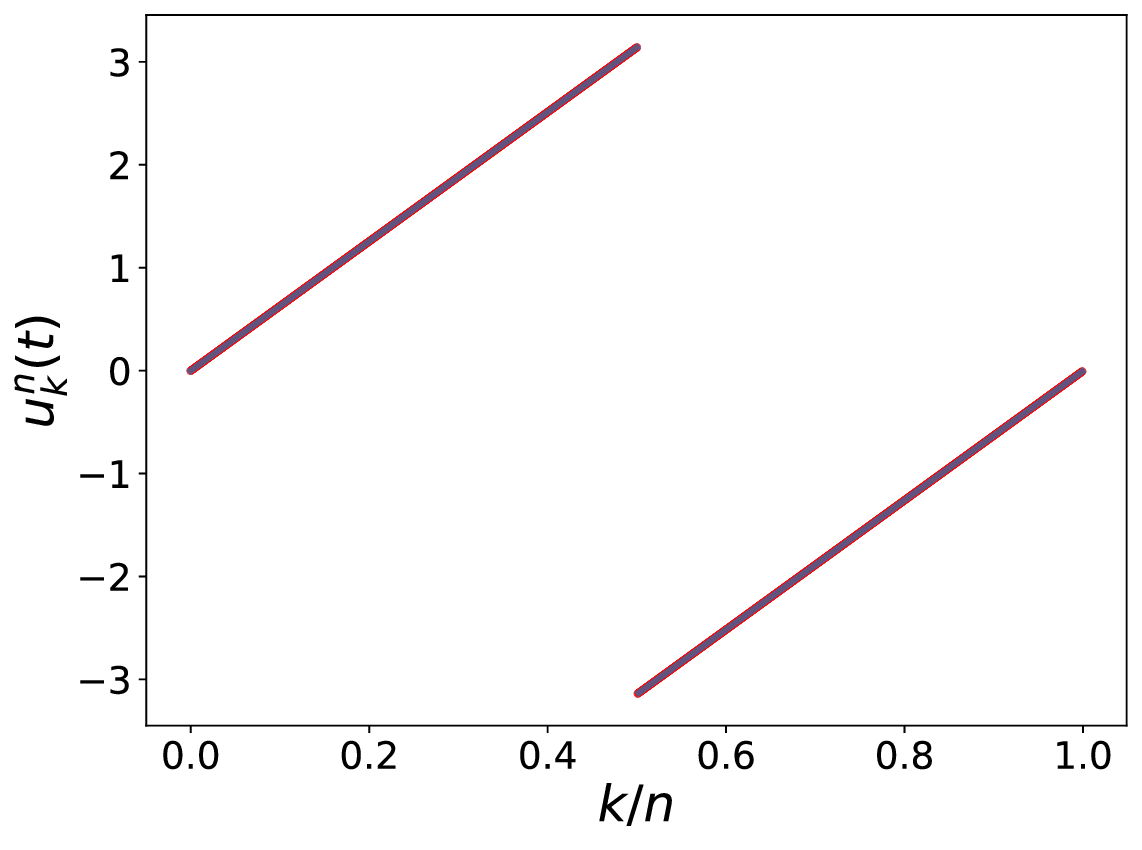}\\[-1ex]
{\footnotesize(a)}
\end{center}
\end{minipage}
\begin{minipage}[t]{0.495\textwidth}
\begin{center}
\includegraphics[scale=0.3]{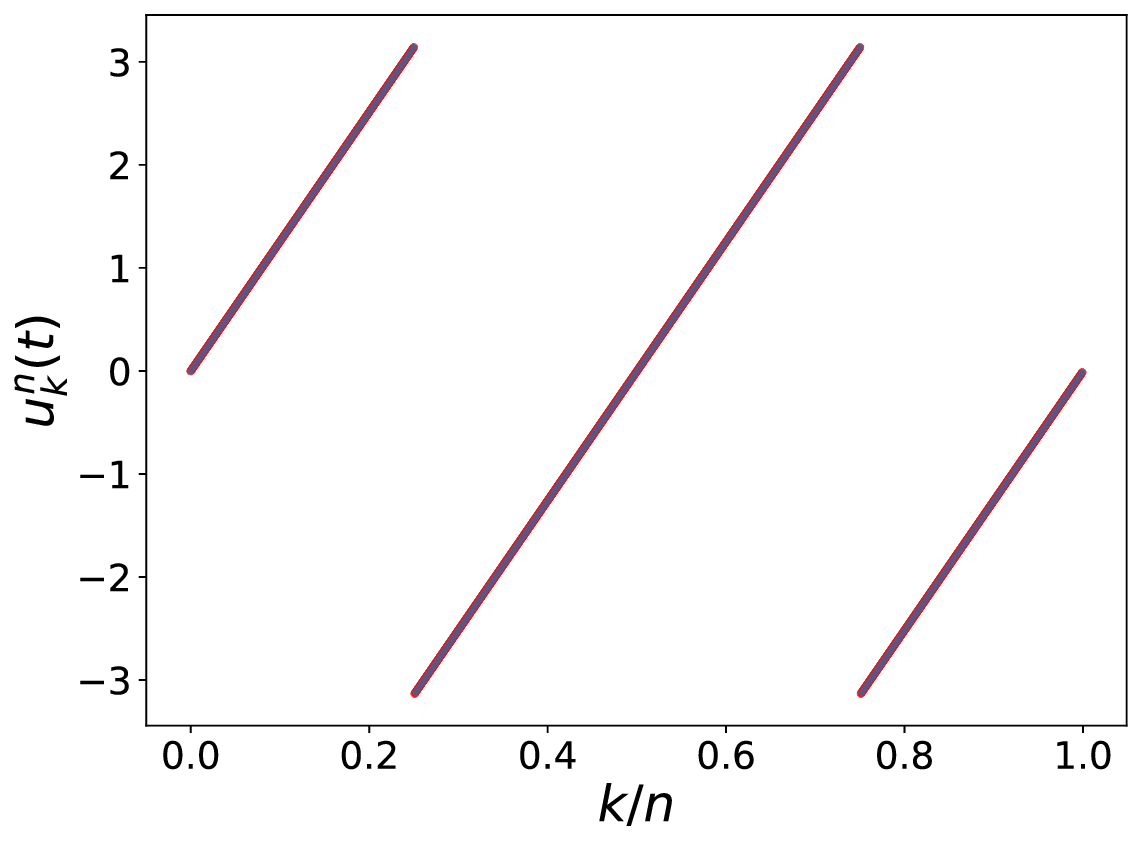}\\[-1ex]
{\footnotesize(b)}
\end{center}
\end{minipage}
\vspace*{1ex}

\begin{minipage}[t]{0.495\textwidth}
\begin{center}
\includegraphics[scale=0.3]{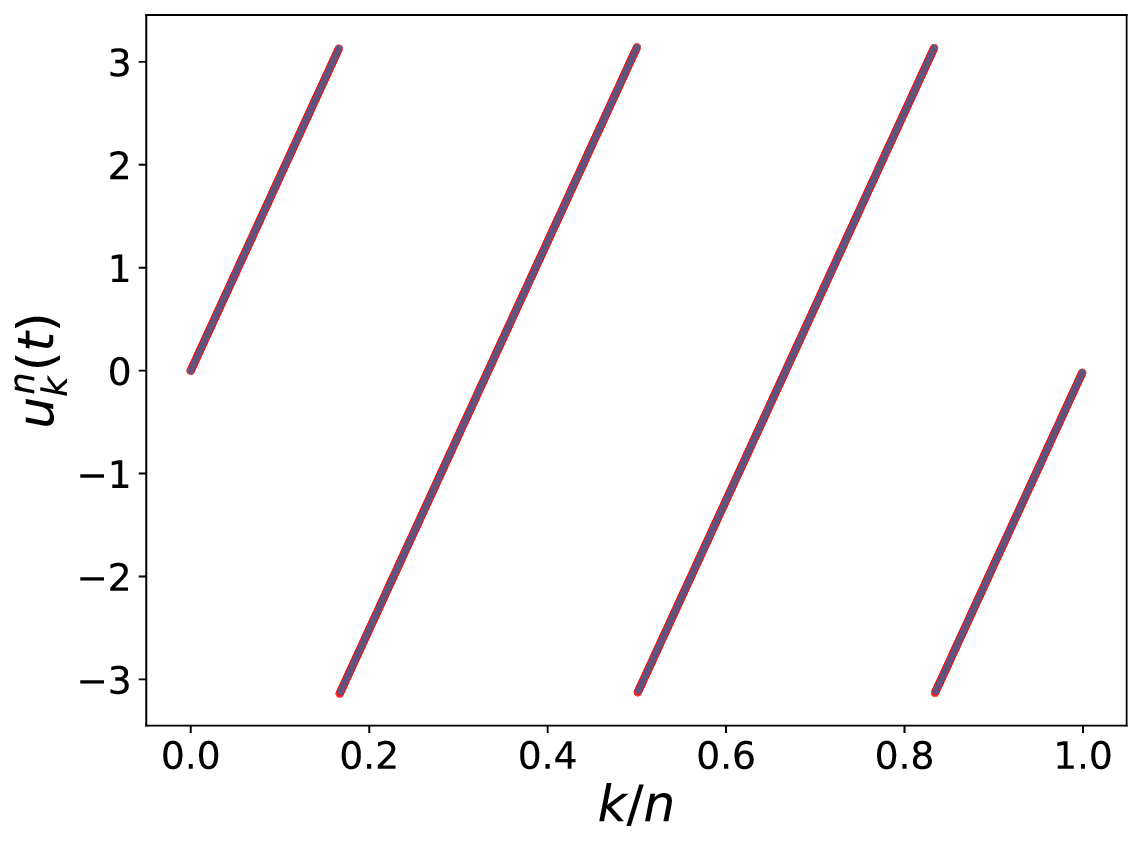}\\[-1ex]
{\footnotesize(c)}
\end{center}
\end{minipage}
\begin{minipage}[t]{0.495\textwidth}
\begin{center}
\includegraphics[scale=0.3]{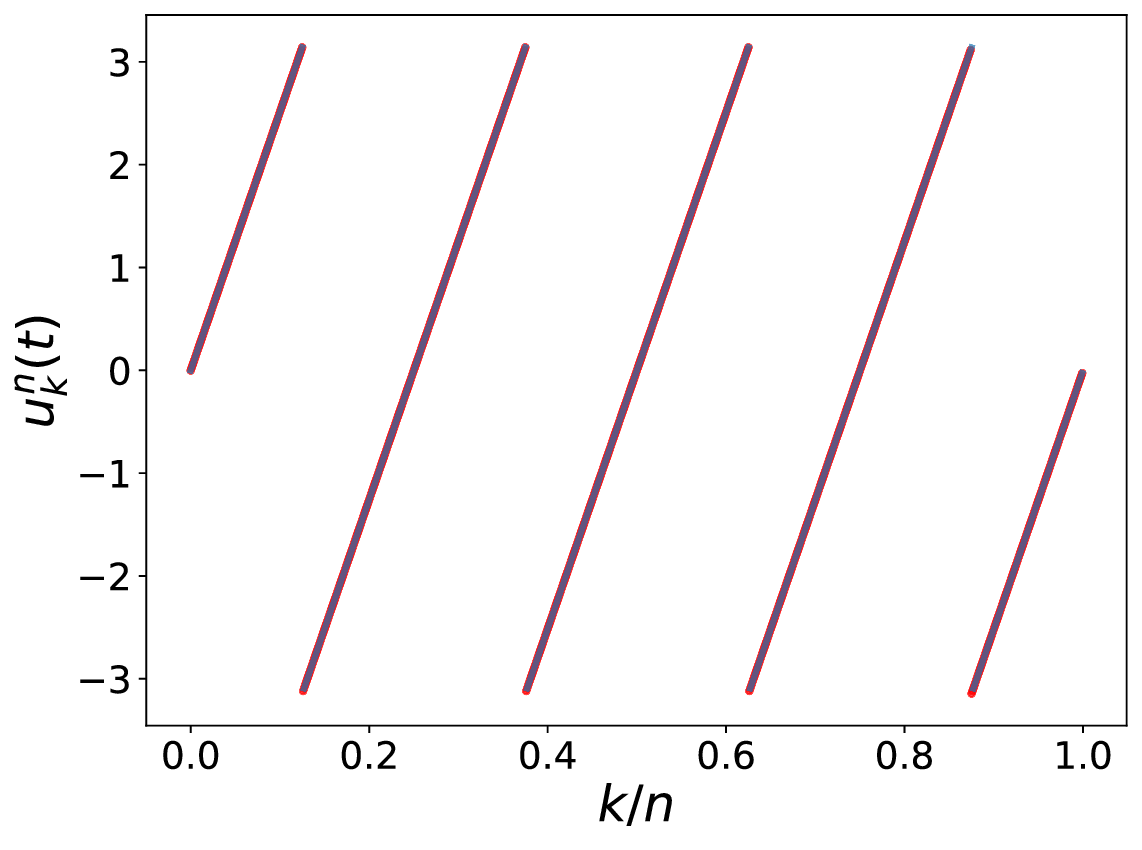}\\[-1ex]
{\footnotesize(d)}
\end{center}
\end{minipage}
\caption{Steady states of the KM \eqref{eqn:dsys}
 with $n=1000$, $\omega=0$, $\sigma=0$ and $t=1000$ in case~(i):
(a) $(q,\kappa)=(1,0.31)$;
(b) $(2,0.16)$;
(c)  $(3,0.1)$;
(d) $(4,0.08)$.
The simulation results are plotted as small red disks
 and the most probable twisted states estimated from them as blue lines
 (see the text for more details) although they coincide almost completely.}
\label{fig:5b2}
\end{figure}

\begin{figure}[t]
\begin{center}
\includegraphics[scale=0.3]{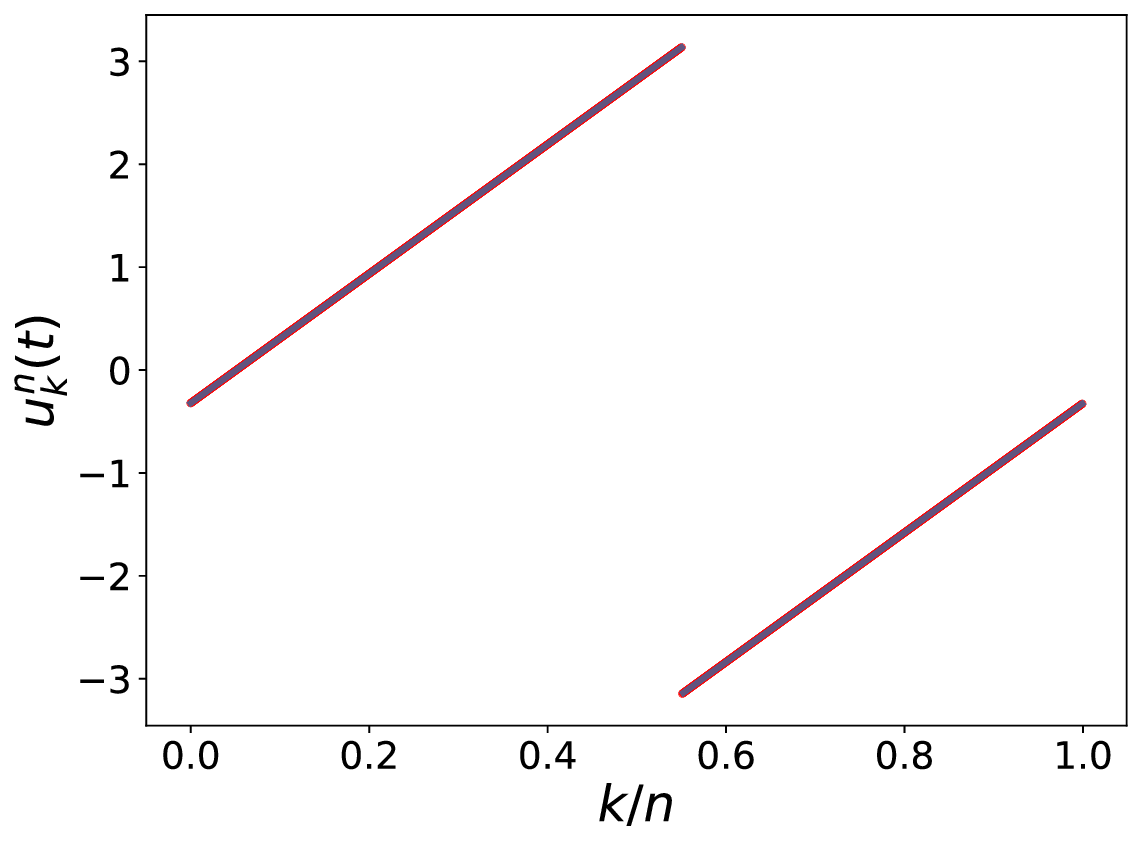}\\[-1ex]
{\footnotesize(a)}
\end{center}
\begin{minipage}[t]{0.495\textwidth}
\begin{center}
\includegraphics[scale=0.3]{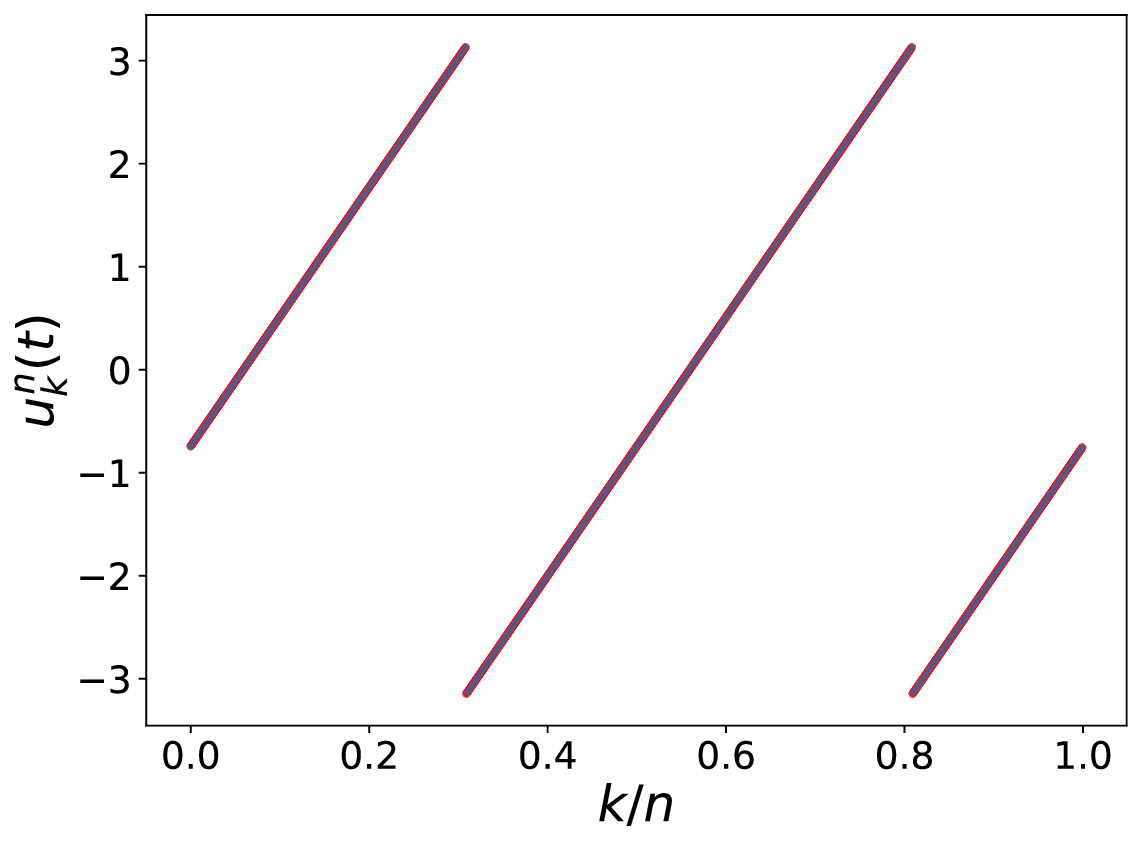}\\[-1ex]
{\footnotesize(b)}
\end{center}
\end{minipage}
\begin{minipage}[t]{0.495\textwidth}
\begin{center}
\includegraphics[scale=0.3]{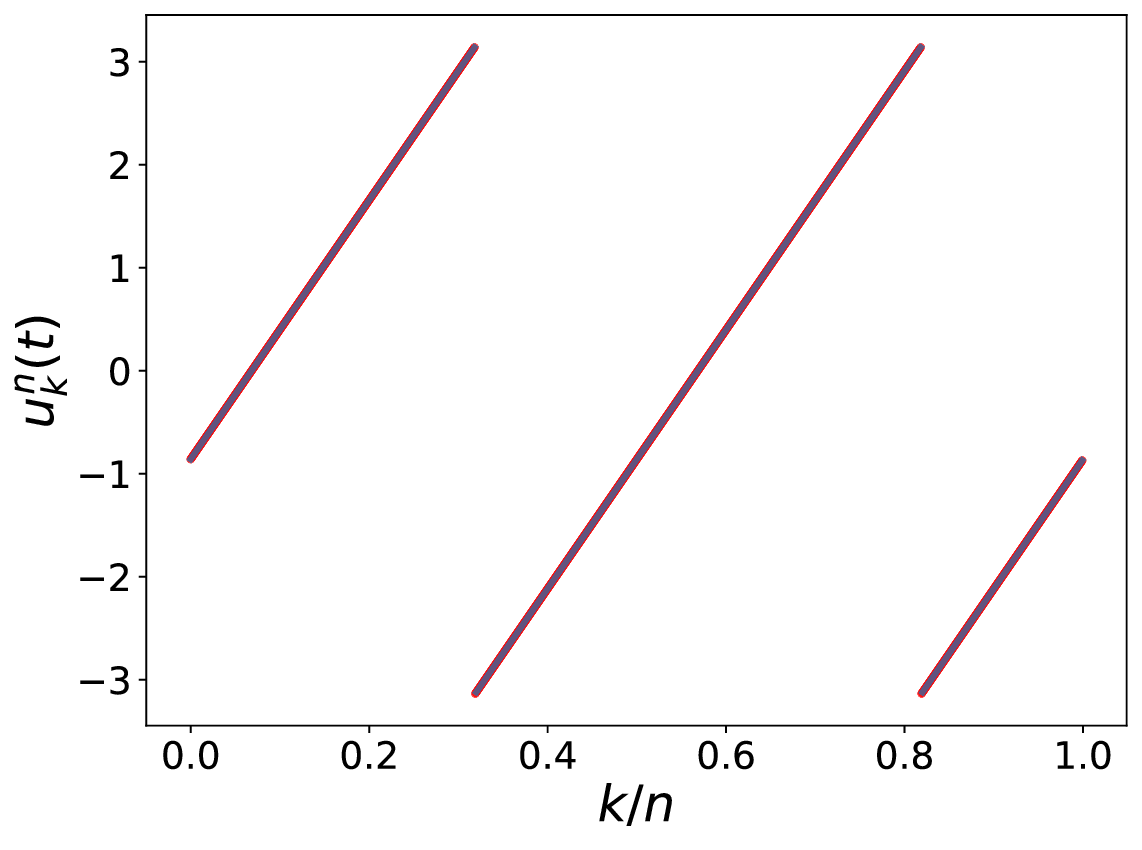}\\[-1ex]
{\footnotesize(c)}
\end{center}
\end{minipage}
\caption{Steady states of the KM \eqref{eqn:dsys}
 with $n=1000$ and $\sigma=\pi/3$ in case~(i):
(a) $(q,\kappa,t)=(1,0.31,1000)$;
(b) $(2,0.16,2000)$;
(c) $(2,0.166,2000)$.
See also the captions of Figs.~\ref{fig:5c1} and \ref{fig:5b2}.}
\label{fig:5c2}
\end{figure}

In Figs.~\ref{fig:5b2} and \ref{fig:5c2},
 $u_k^n(t)$, $k\in[n]$, are plotted as small red disks
 for $\sigma=0$ and $\pi/3$, respectively, 
 where the time $t=1000$ or $2000$ was chosen
 such that they may be regarded as the steady state responses,
 and the same values as in Figs.~\ref{fig:5b1} and \ref{fig:5c1}
 for $n$, $\kappa$ and $u_k^n(0)$, $k\in[n]$, were used.
The most probable twisted state \eqref{eqn:tsol}
 in the CL \eqref{eqn:csys} obtained from the numerical result
 is also plotted as a blue line in each figure, 
 where it was estimated by using the least mean square method as
\[
u_k^n(t)=2\pi qk/n+\overline{u_k^n(t)}-\pi q,
\]
where
\[
\overline{u_k^n(t)}=\frac{1}{n}\sum_{k=1}^n u_k^n(t)
\]
is the mean of $u_k^n(t)$, $k\in[n]$.
Both results coincide almost completely even in Fig.~\ref{fig:5c2}(c)
 with $q=2$ for $\kappa=166$, which is considered to be bigger than the bifurcation points
 such that the twisted state is unstable,
 as well as in the other figures
 as detected by Theorems~\ref{thm:4a} and \ref{thm:4b} for the CL \eqref{eqn:csys}
 with Corollary~\ref{cor:2a}.

\begin{figure}[t]
\begin{center}
\includegraphics[scale=0.3]{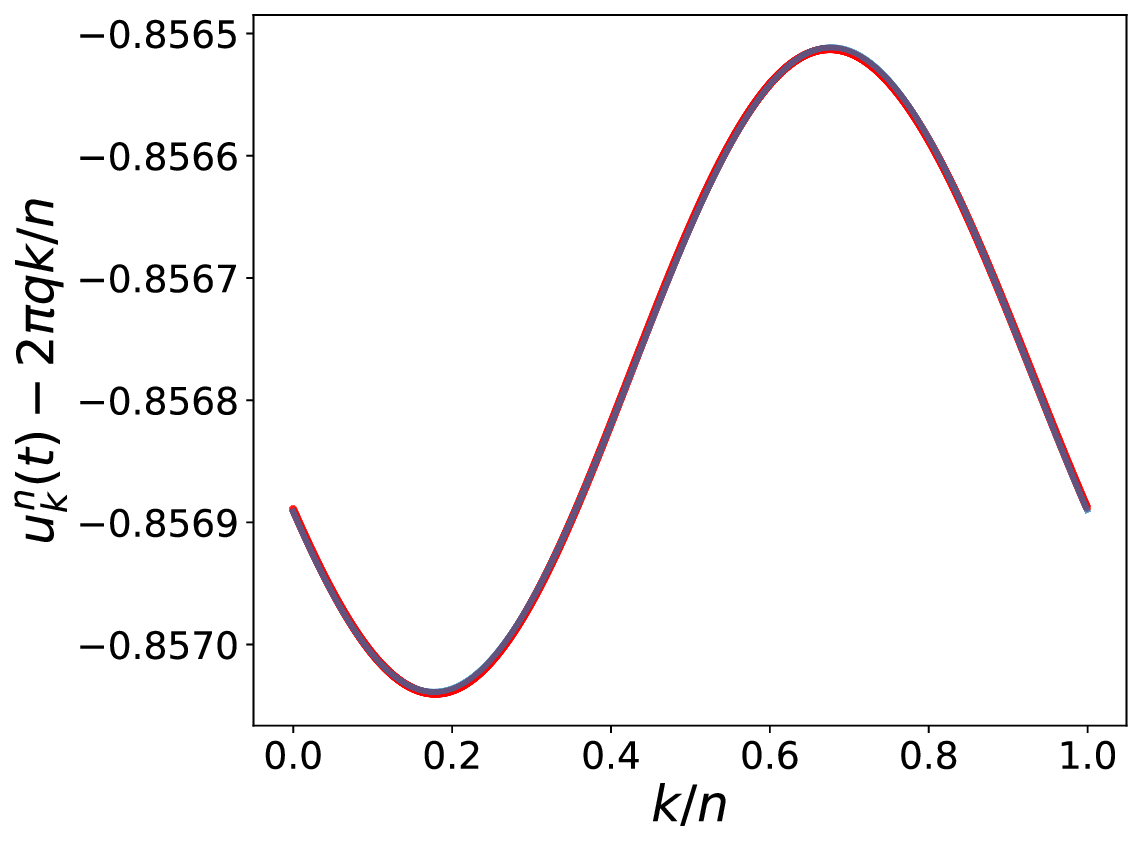}
\end{center}
\caption{Deviation from the $q$-twisted states
 in the steady states of the KM \eqref{eqn:dsys}
 with $n=1000$, $\kappa=0.166$, $\sigma=\pi/3$ and $t=2000$ for $q=2$ in case~(i):
The simulation result is plotted as small red disks
 and the most probably leading term
 of the oscillating twisted solution \eqref{eqn:thm4b}
 estimated from it is plotted as blue lines (see the text for more details)
 although they coincide almost completely.}
\label{fig:5c3}
\end{figure}

In Fig.~\ref{fig:5c3}, the deviation, $u_k^n(t)-2\pi q k/n$,
 of the steady state in Fig.~\ref{fig:5c2}(c) for $\kappa=0.166$,
 which is considered to be bigger than the bifurcation points,
 from the $q$-twisted one with $q=2$ in the KM \eqref{eqn:dsys} with $\sigma=\pi/3$
 is plotted as small red disks.
The most probably leading term,
\begin{equation}
u(x)=2\pi qx+r(t)\sin(2\pi x+\psi(t))+\tilde{\Omega}t+\theta,
\label{eqn:ss}
\end{equation}
of the oscillating twisted solution \eqref{eqn:thm4b}
 estimated from the numerical result
 by using the least mean square method as
\[
\tilde{\Omega}t+\theta+\pi q=\frac{1}{n}\sum_{k=1}^nv_k^n(t),\quad
r(t)=2\sqrt{c(t)^2+s(t)^2}
\]
and
\begin{align*}
\psi(t)=&\arctan\frac{s(t)}{c(t)}\quad
\left(\mbox{resp. }\arctan\frac{s(t)}{c(t)}+pi\mbox{ or } 
\arctan\frac{s(t)}{c(t)}-\pi\right)
\end{align*}
for $c(t)>0$ (resp. $c(t)<0$ and  $s(t)>0$ or  $s(t)<0$) with
\[
v_k^n(t)=u_k^n(t)-\frac{2\pi q k}{n}\mod 2\pi
\]
and
\[
c(t)=\frac{1}{n}\sum_{k=1}^n v_k^n(t)\cos\frac{2\pi k}{n},\quad
s(t)=\frac{1}{n}\sum_{k=1}^n v_k^n(t)\sin\frac{2\pi k}{n},
\]
is also plotted as the blue line in each figure.
The agreement between both results is almost perfect.

\begin{figure}[t]
\begin{minipage}[t]{0.495\textwidth}
\begin{center}
\includegraphics[scale=0.26]{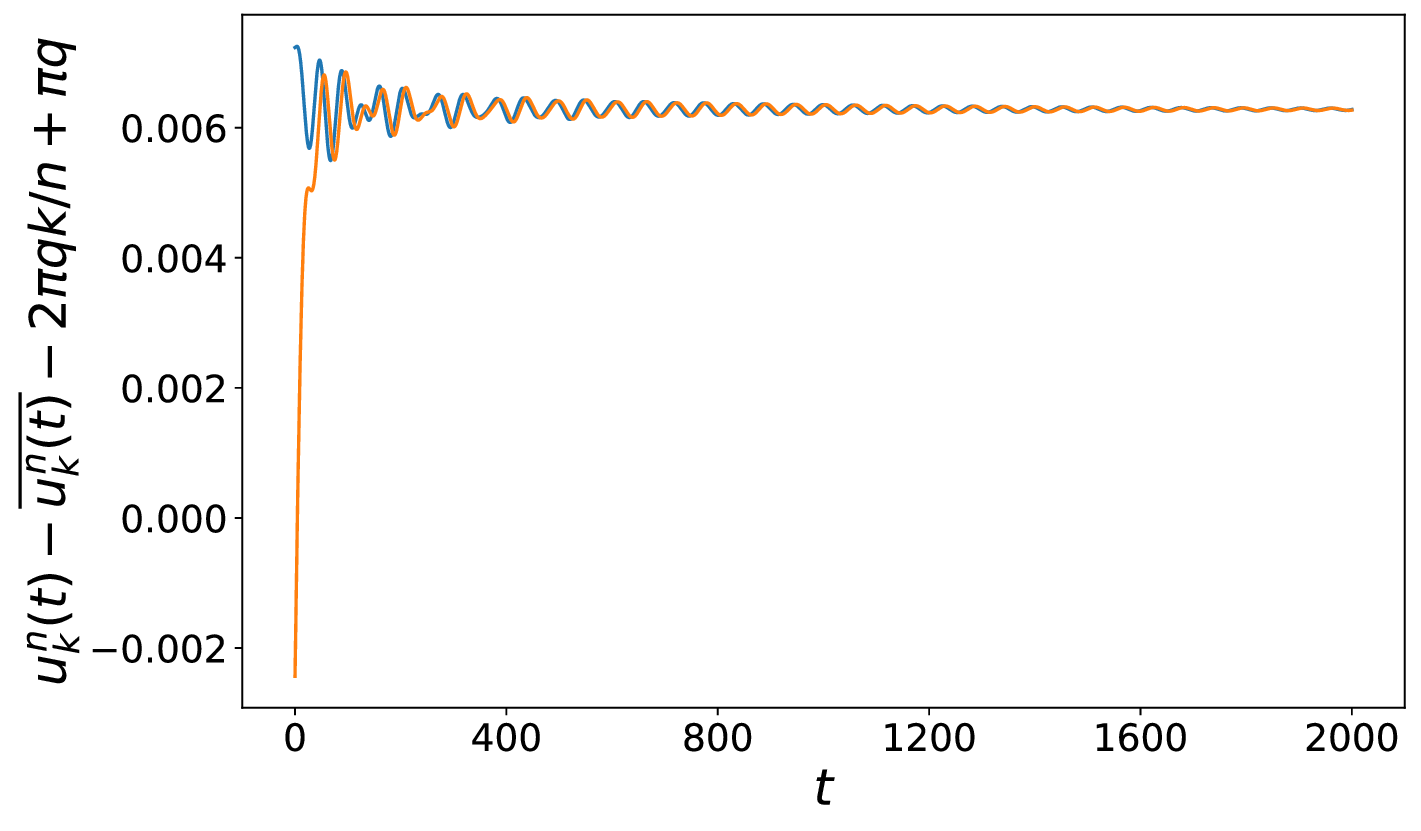}\\[-1ex]
{\footnotesize(a)}
\end{center}
\end{minipage}
\begin{minipage}[t]{0.495\textwidth}
\begin{center}
\includegraphics[scale=0.26]{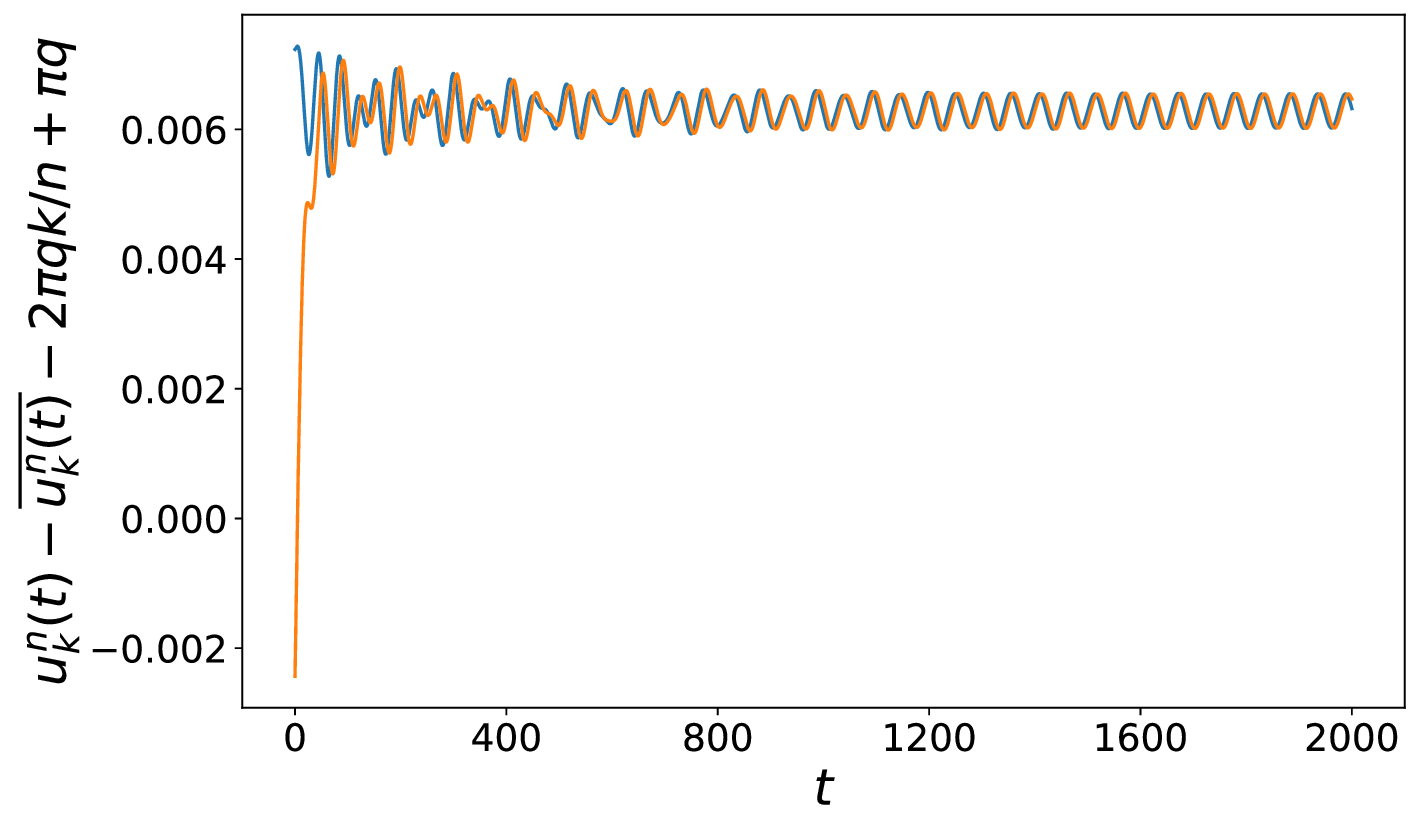}\\[-1ex]
{\footnotesize(b)}
\end{center}
\end{minipage}
\caption{Small oscillations in the steady states of the KM \eqref{eqn:dsys}
 with $n=1000$ and $\sigma=\pi/3$ for $q=2$ in case~(i):
(a) $\kappa=0.16$;
(b) $0.166$.
They are plotted as blue and red lines
 for $k=450$ and $550$, respectively, in Fig.~(a) (resp. in Fig.~(b)).}
\label{fig:5c4}
\end{figure}

In Fig.~\ref{fig:5c4}, the deviations of the oscillator phases at $k=450$ and  $550$,
 respectively, from the $q$-twisted states with $q=2$,
\[
u_k^n(t)-\overline{u_k^n(t)}-2\pi q kn+\pi q,
\]
in the KM \eqref{eqn:dsys} with $n=1000$ and $\sigma=\pi/3$
 for the numerical results for $\kappa=0.16,166$
 are plotted as blue and orange lines.
The same values of $\kappa$ and $u_k^n(0)$, $k\in[n]$, as  in Figs.~\ref{fig:5c1}(b) and (c)
 were used in Figs.~\ref{fig:5c4}(a) and (b), respectively.
We observe that small oscillation continues in Fig.~\ref{fig:5c4}(b) for $\kappa=0.166$
 as predicted by Theorem~\ref{thm:4b}(ii) with Corollary~\ref{cor:2a}
 for the bifurcated solutions in the CL \eqref{eqn:csys},
 while it does not in Fig.~\ref{fig:5c4}(a) for $\kappa=0.16$.

\begin{figure}[t]
\begin{minipage}[t]{0.495\textwidth}
\begin{center}
\includegraphics[scale=0.3]{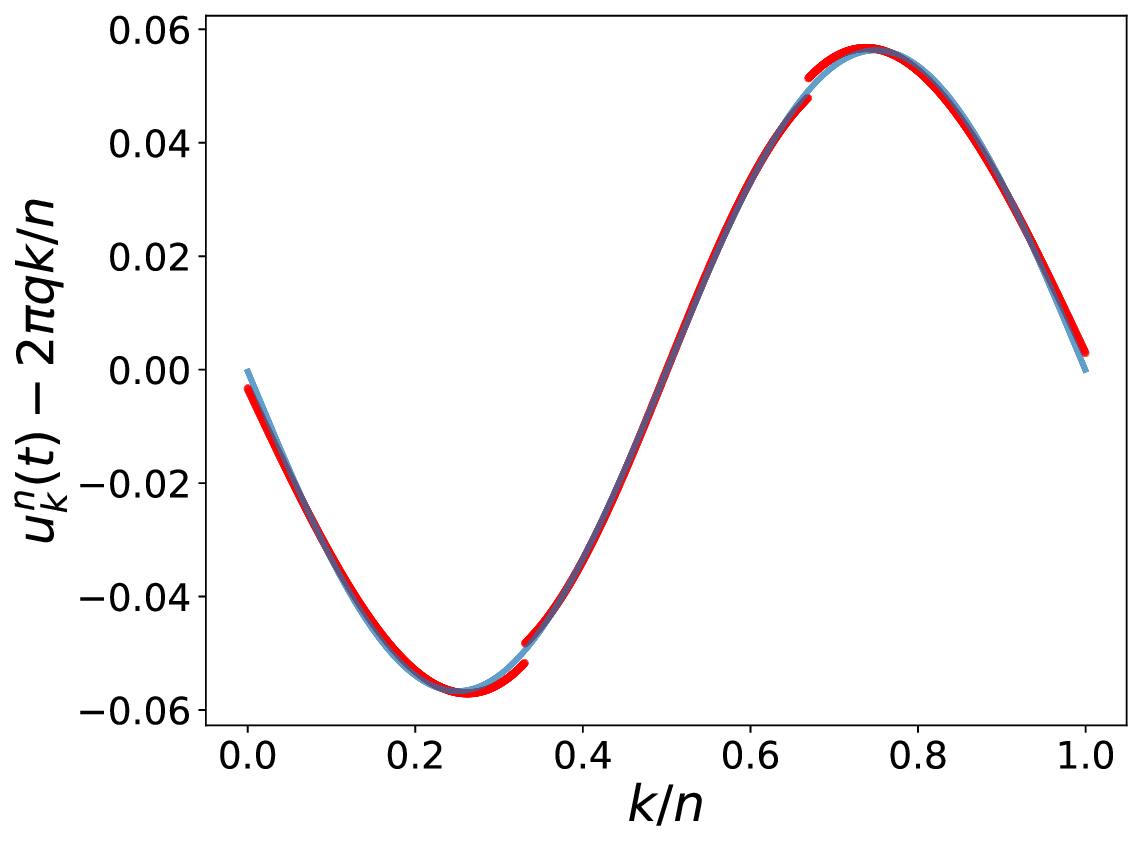}\\[-1ex]
{\footnotesize(a)}
\end{center}
\end{minipage}
\begin{minipage}[t]{0.495\textwidth}
\begin{center}
\includegraphics[scale=0.3]{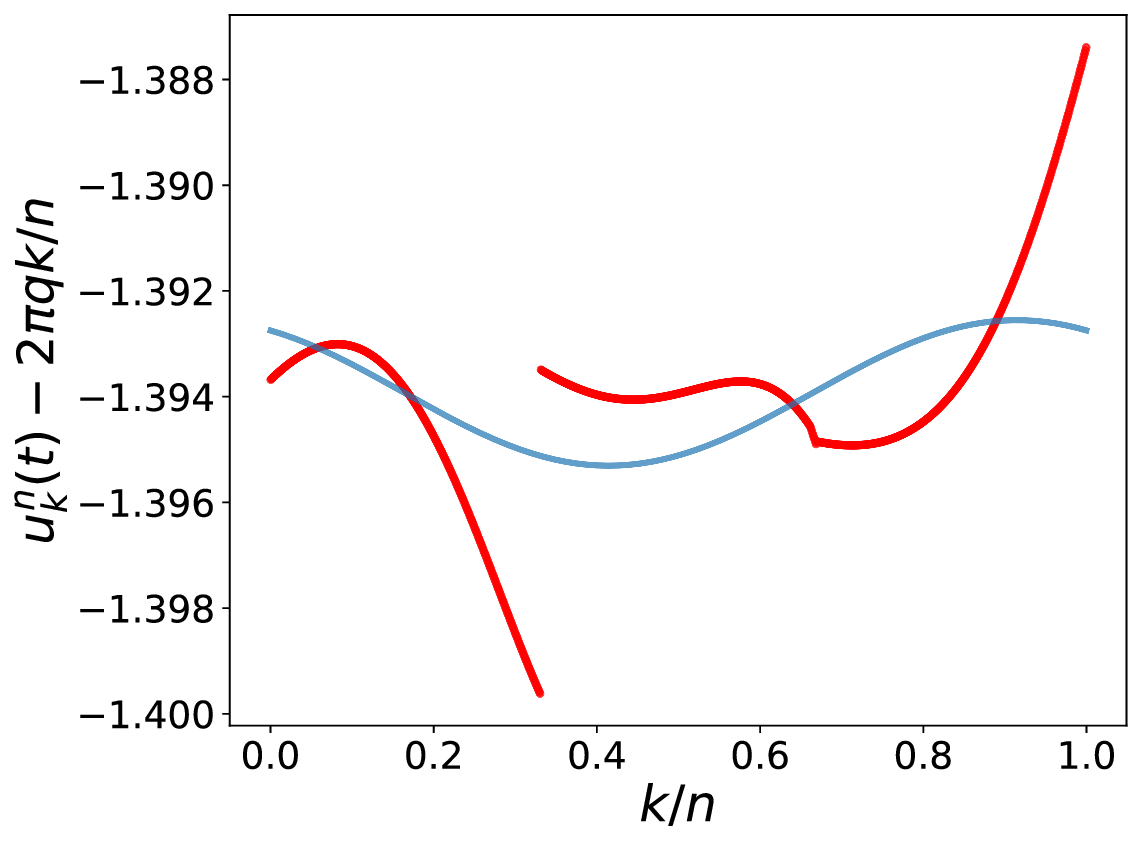}\\[-1ex]
{\footnotesize(b)}
\end{center}
\end{minipage}
\caption{Deviation from the $q$-twisted states
 in the steady states of the KM \eqref{eqn:dsys}
 with $n=1000$ and $\kappa=0.33$ at $t=2000$ for $q=1$ in case~(i):
(a) $\sigma=0$;
(b) $\pi/3$.
The blue line represents its first-order Fourier mode.}
\label{fig:5b3}
\end{figure}

Different modulated twisted solutions were also found in wider ranges of $\kappa$
 above the bifurcation points such that the $q$-twisted states were not observed.
In Fig.~\ref{fig:5b3}, the deviation, $u_k^n(t)-2\pi q k/n$,
  of two of such different states
  from the $q$-twisted ones are plotted as small red disks for $q=1$.
Its first-order Fourier mode estimated as in \eqref{eqn:ss}
 for the numerical results
 is plotted as a blue line in each figure.

\subsection{Cases~(ii) and (iii): Random graphs}

\begin{figure}[t]
\begin{minipage}[t]{0.495\textwidth}
\begin{center}
\includegraphics[scale=0.27]{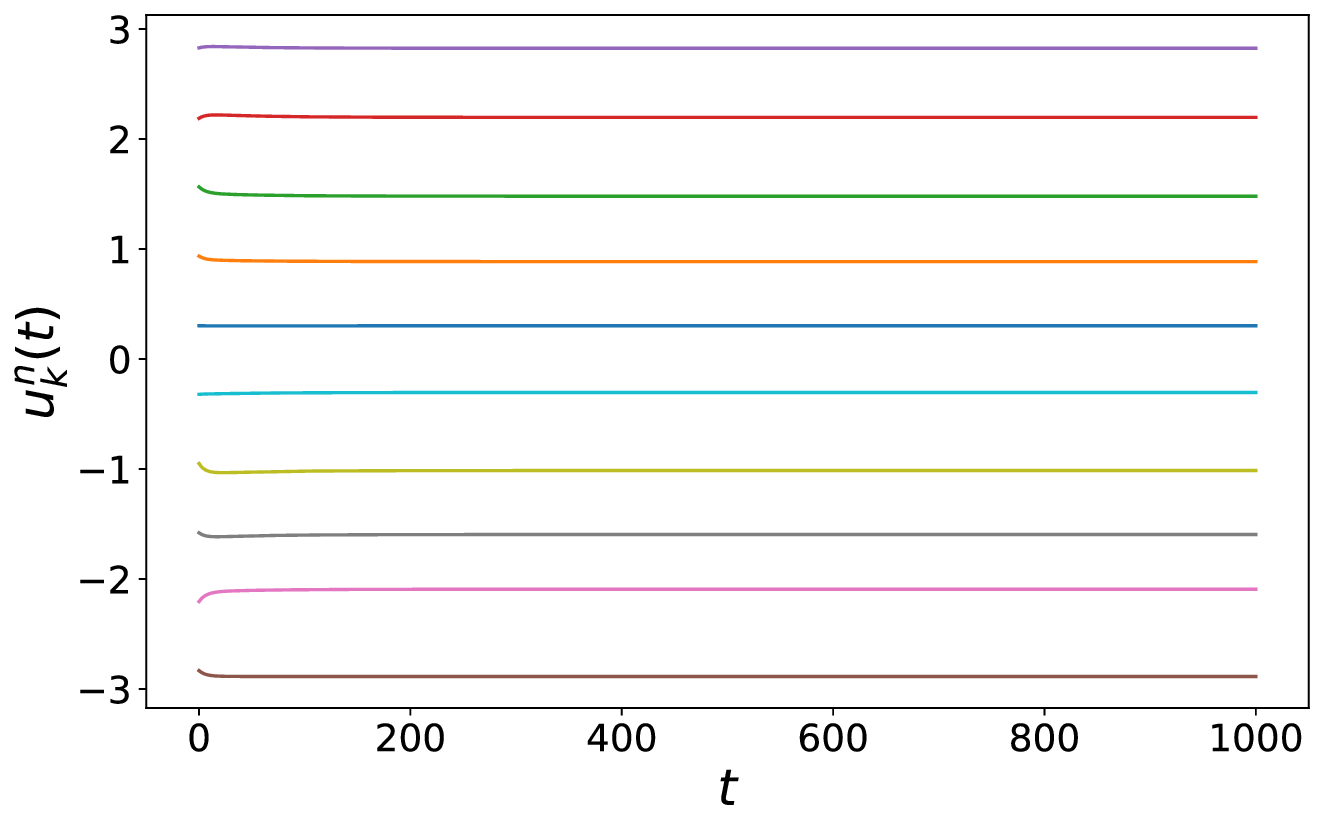}\\[-1ex]
{\footnotesize(a)}
\end{center}
\end{minipage}
\begin{minipage}[t]{0.495\textwidth}
\begin{center}
\includegraphics[scale=0.27]{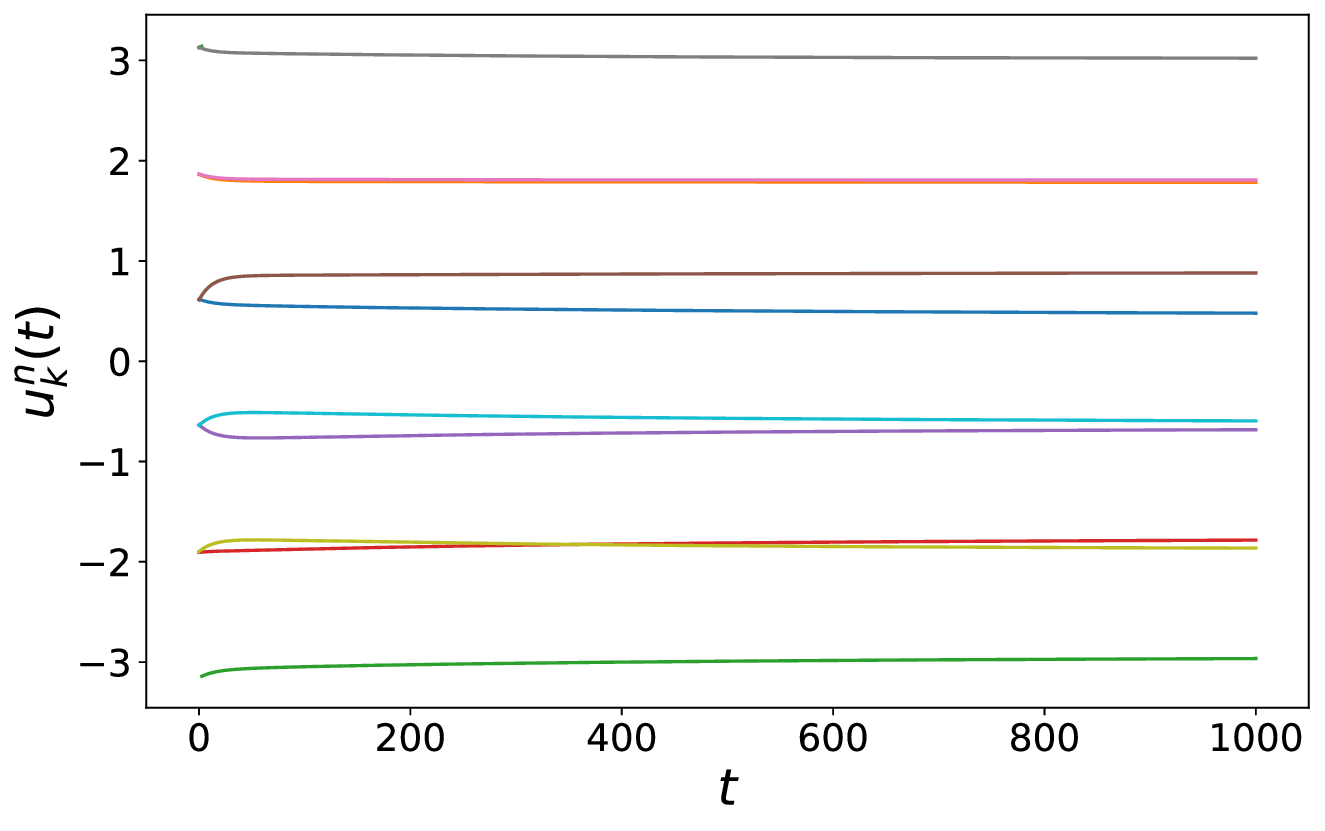}\\[-1ex]
{\footnotesize(b)}
\end{center}
\end{minipage}
\vspace*{1ex}

\begin{minipage}[t]{0.495\textwidth}
\begin{center}
\includegraphics[scale=0.27]{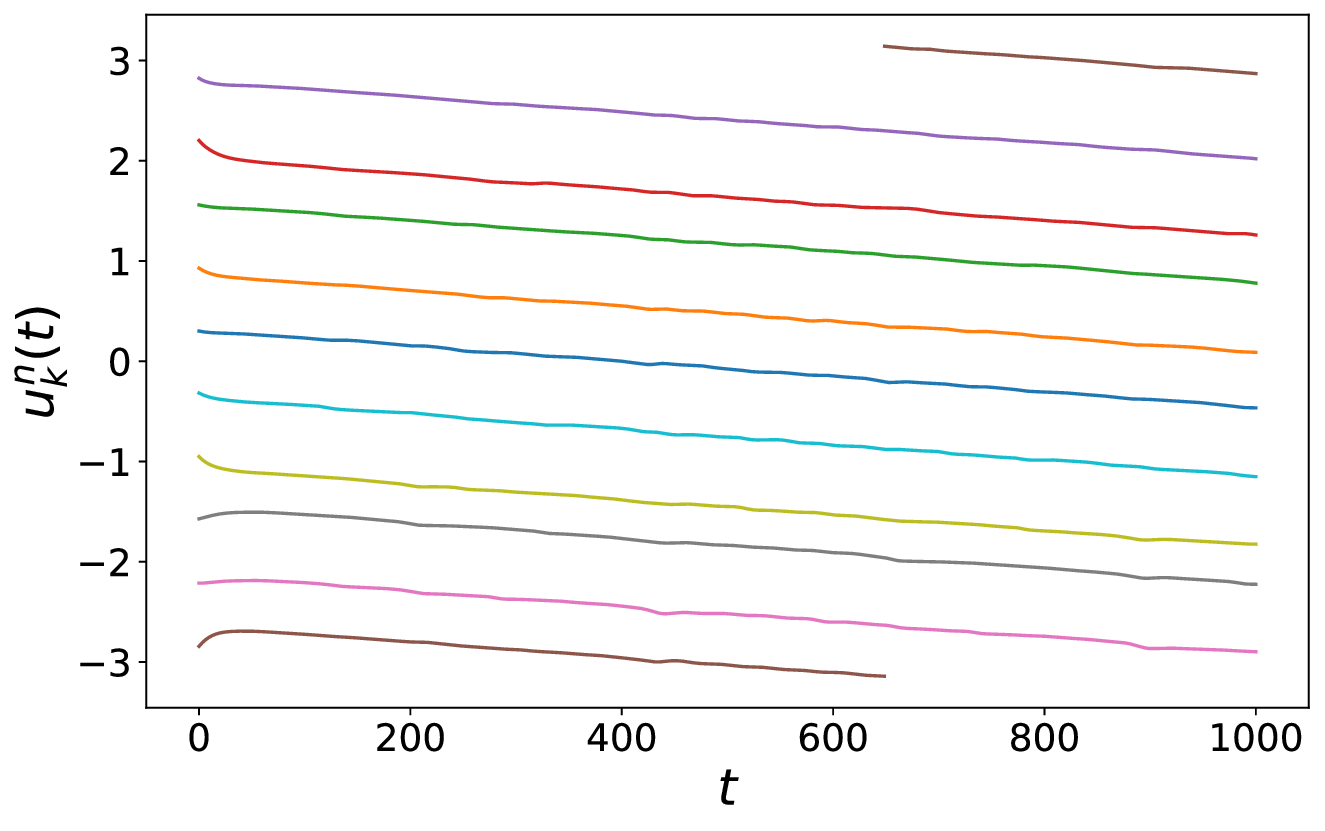}\\[-1ex]
{\footnotesize(c)}
\end{center}
\end{minipage}
\begin{minipage}[t]{0.495\textwidth}
\begin{center}
\includegraphics[scale=0.27]{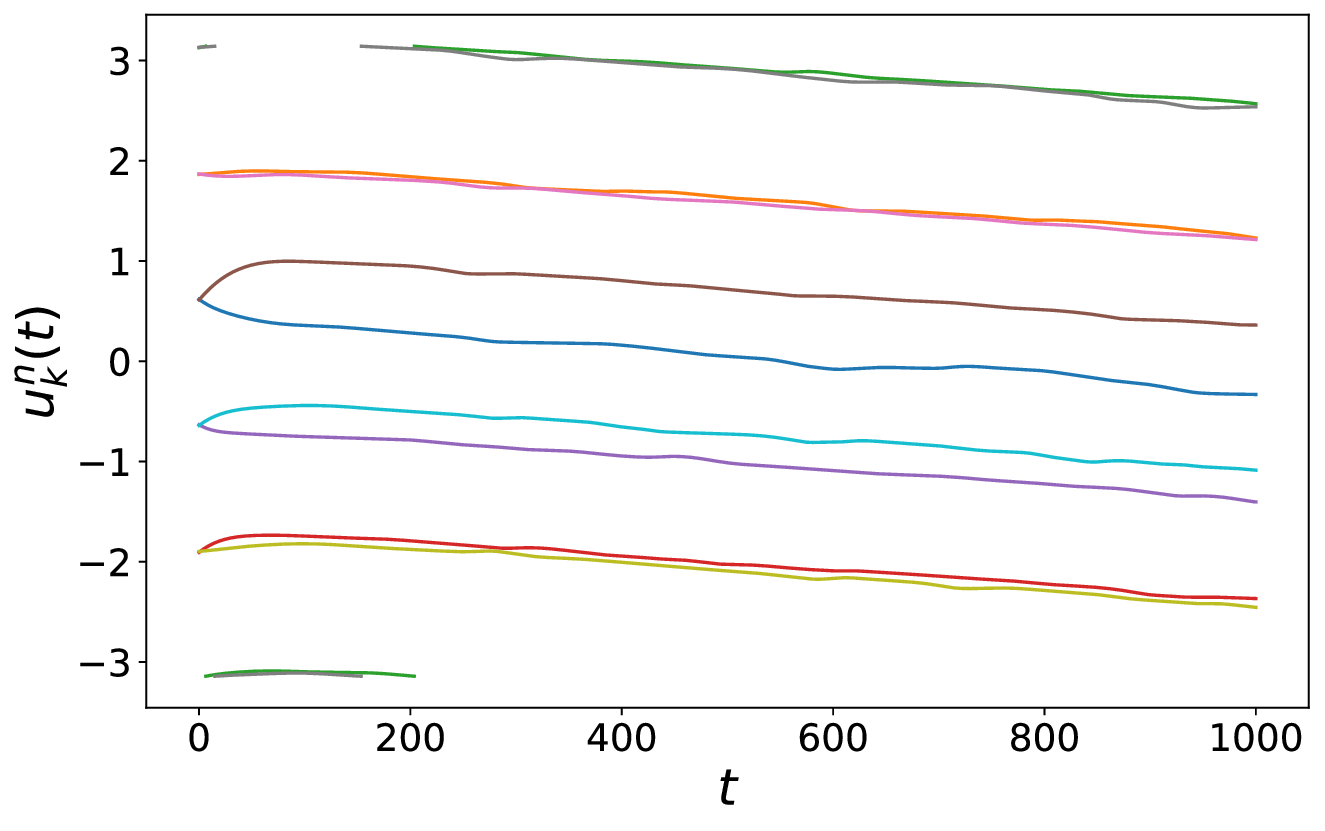}\\[-1ex]
{\footnotesize(d)}
\end{center}
\end{minipage}
\caption{Time histories of the KM \eqref{eqn:dsys}
 with $n=1000$ in case~(ii):
(a) $(\kappa,\sigma)=(0.31,0)$;
(b) $(0.15,0)$;
(c)  $(0.31,\pi/3)$;
(d) $(0.15,\pi/3)$.
The natural frequency $\omega$ was zero for $\sigma=0$
 and given by \eqref{eqn:omega} for $\sigma=\pi/3$.
 See also the caption of Fig.~\ref{fig:5b1}.}
\label{fig:5d1}
\end{figure}

\begin{figure}[t]
\begin{minipage}[t]{0.495\textwidth}
\begin{center}
\includegraphics[scale=0.27]{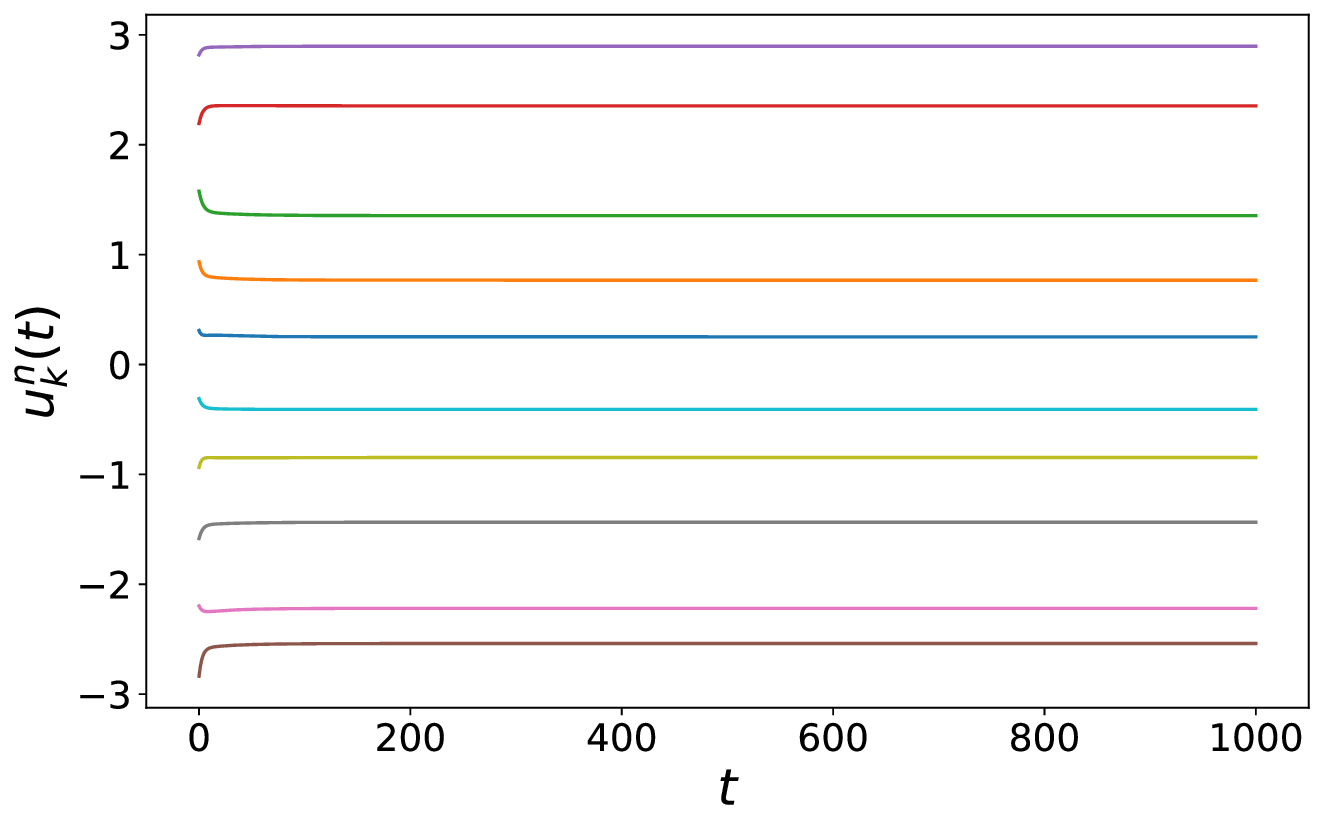}\\[-1ex]
{\footnotesize(a)}
\end{center}
\end{minipage}
\begin{minipage}[t]{0.495\textwidth}
\begin{center}
\includegraphics[scale=0.27]{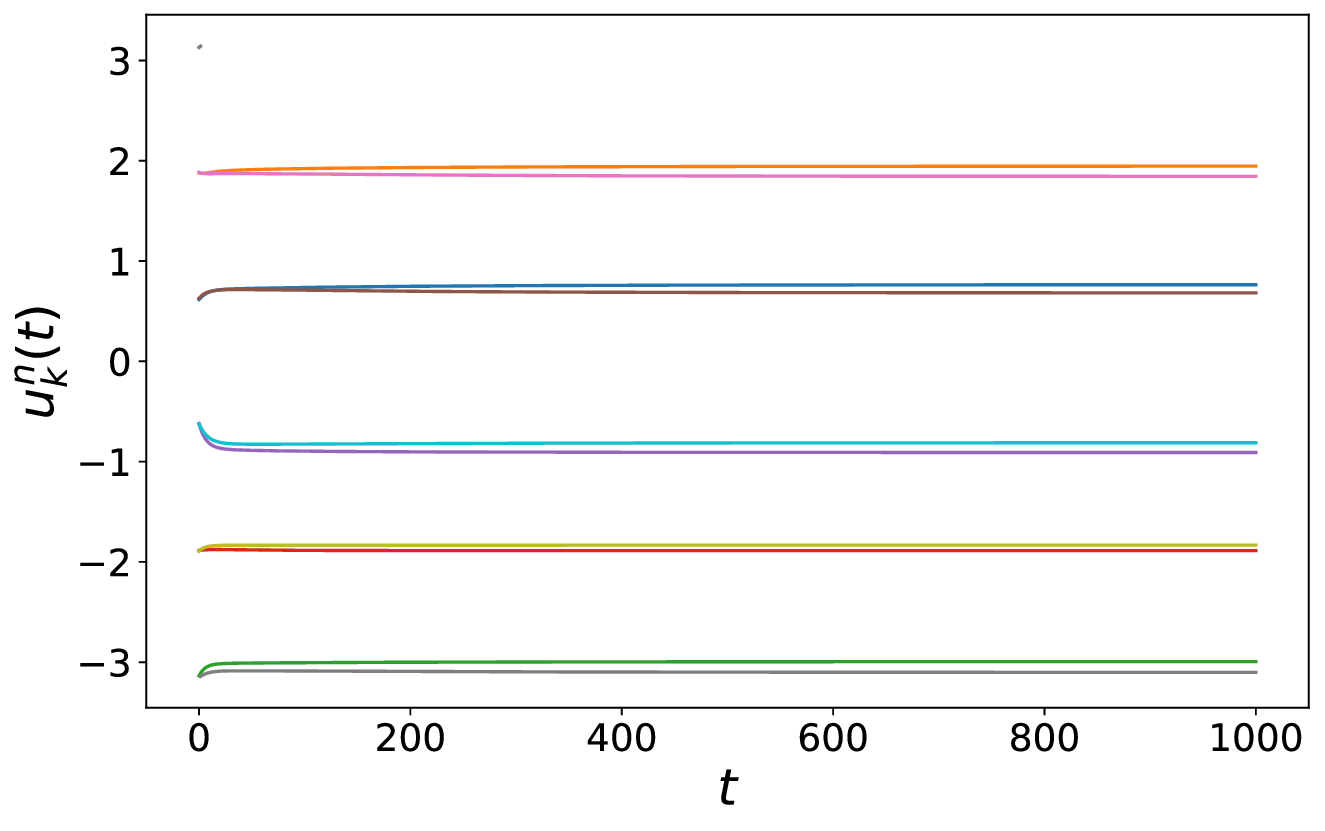}\\[-1ex]
{\footnotesize(b)}
\end{center}
\end{minipage}
\vspace*{1ex}

\begin{minipage}[t]{0.495\textwidth}
\begin{center}
\includegraphics[scale=0.27]{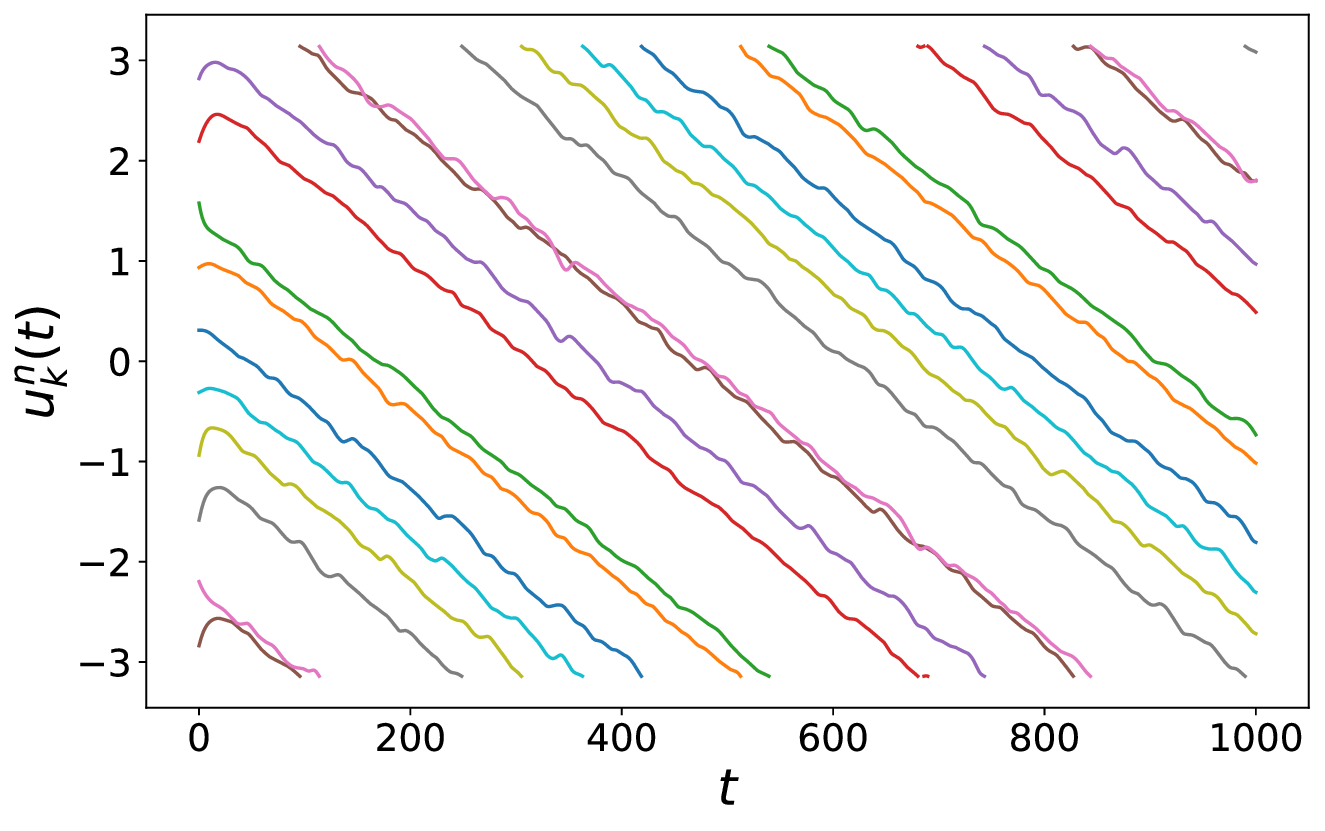}\\[-1ex]
{\footnotesize(c)}
\end{center}
\end{minipage}
\begin{minipage}[t]{0.495\textwidth}
\begin{center}
\includegraphics[scale=0.27]{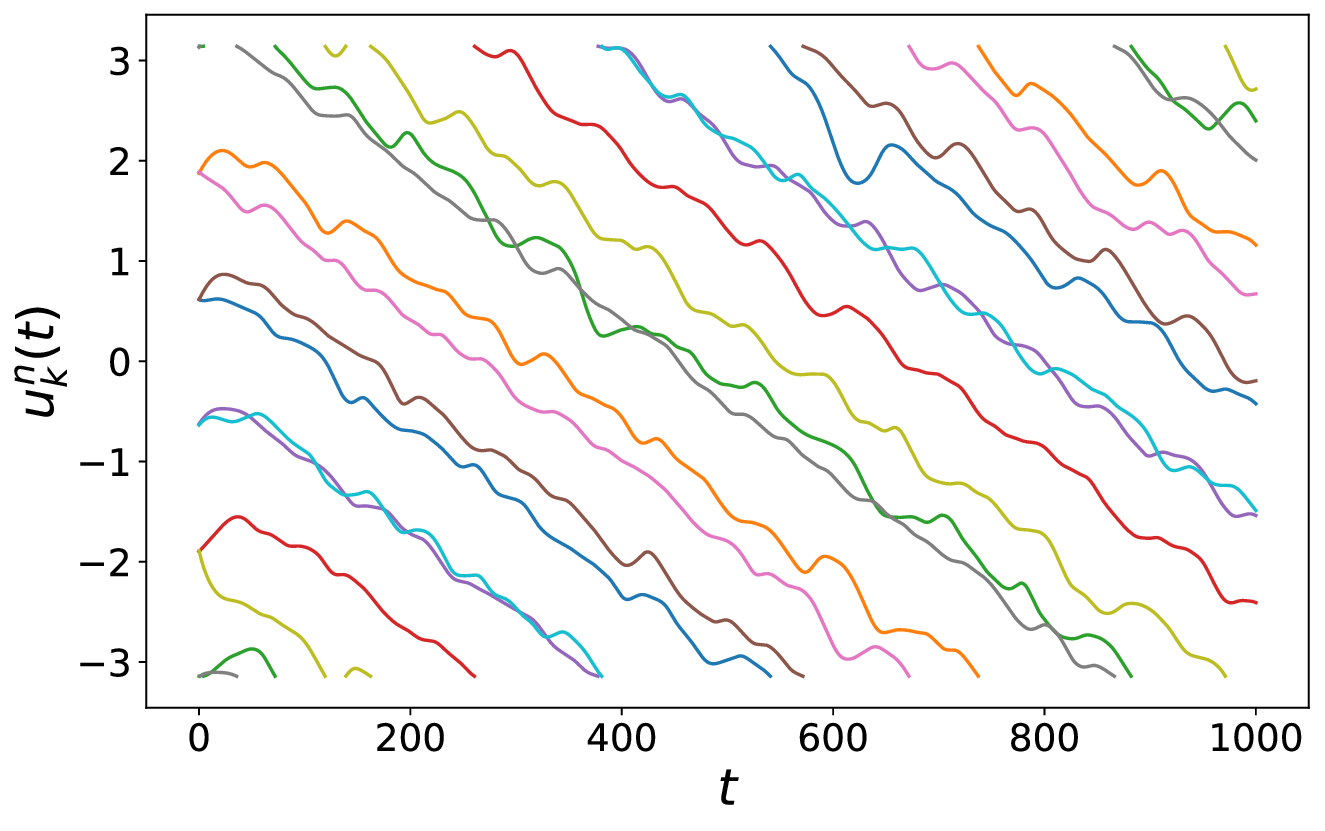}\\[-1ex]
{\footnotesize(d)}
\end{center}
\end{minipage}
\caption{Time histories of the KM \eqref{eqn:dsys}
 with $n=2000$ in case~(iii):
(a) $(q,\kappa,\sigma)=(1,0.31,0)$;
(b) $(2,0.15,0)$;
(c)  $(1,0.31,\pi/3)$;
(d) $(2,0.15,\pi/3)$.
They are plotted for every $200$th node (from 100th to 1900th).
See also the captions of Figs.~\ref{fig:5b1} and \ref{fig:5d1}.}
\label{fig:5e1}
\end{figure}

We next give numerical results for cases~(ii) and (iii).
Figures~\ref{fig:5d1} and \ref{fig:5e1}
 show the time-histories of every $100$th node (from 50th to 950th)
 and every $200$th node (from 100th to 1900th)
 in cases~(ii) and (iii), respectively.
Here $\sigma=0$ in plates~(a) and (b) of both figures
 and $\sigma=\pi/3$ in plates~(c) and (d),
 while the initial values were chosen near the twisted state $u_k^n=2\pi qk/n$ for $q=1,2$,
 as in Section~5.1.
The values of $\kappa$ are considered to be smaller than the bifurcation points,
 which are approximated by $\kappa_{1q}$, $q=1,2$.
We observe that the responses remain near the twisted states
 in Figs.~\ref{fig:5d1} and \ref{fig:5e1},
 and exhibit slow rotation due to the finite size effect
 in plates (c) and (d) of both figures, as in Fig.~\ref{fig:5c1},
 although fluctuations due to randomness are found.

\begin{figure}[t]
\begin{minipage}[t]{0.495\textwidth}
\begin{center}
\includegraphics[scale=0.3]{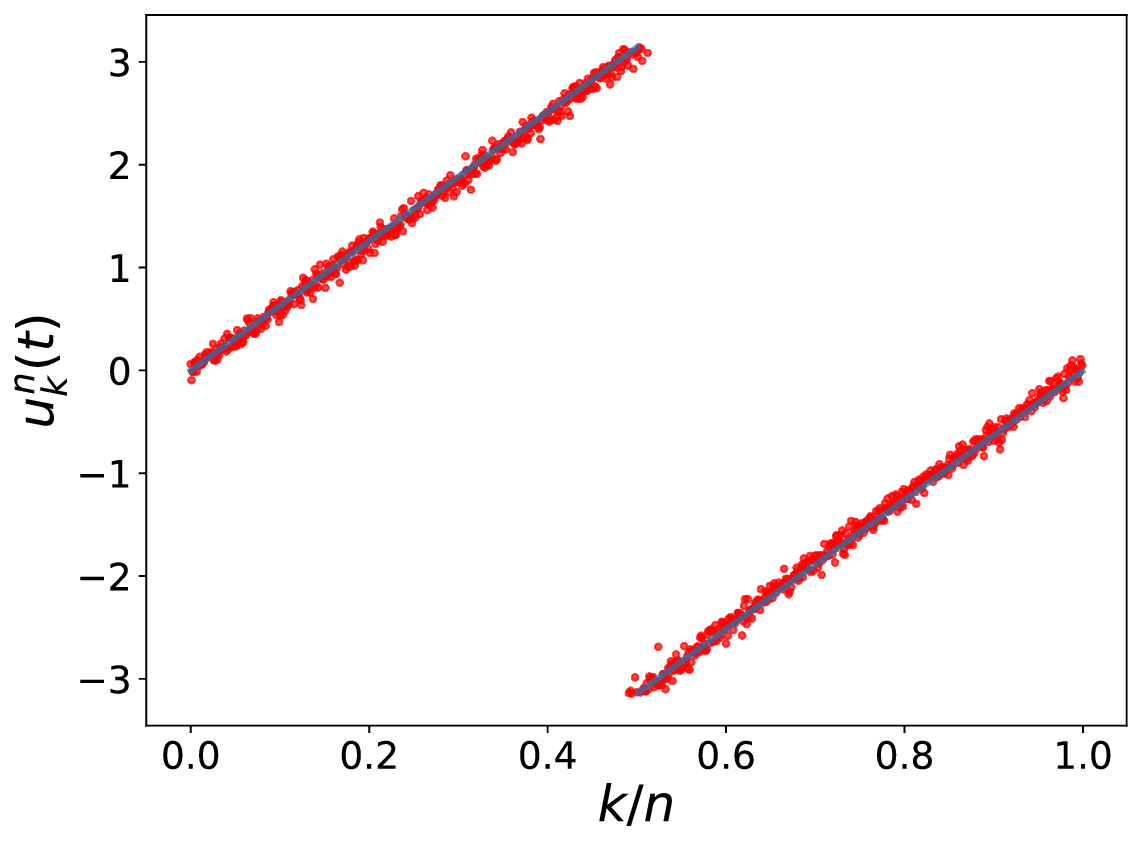}\\[-1ex]
{\footnotesize(a)}
\end{center}
\end{minipage}
\begin{minipage}[t]{0.495\textwidth}
\begin{center}
\includegraphics[scale=0.3]{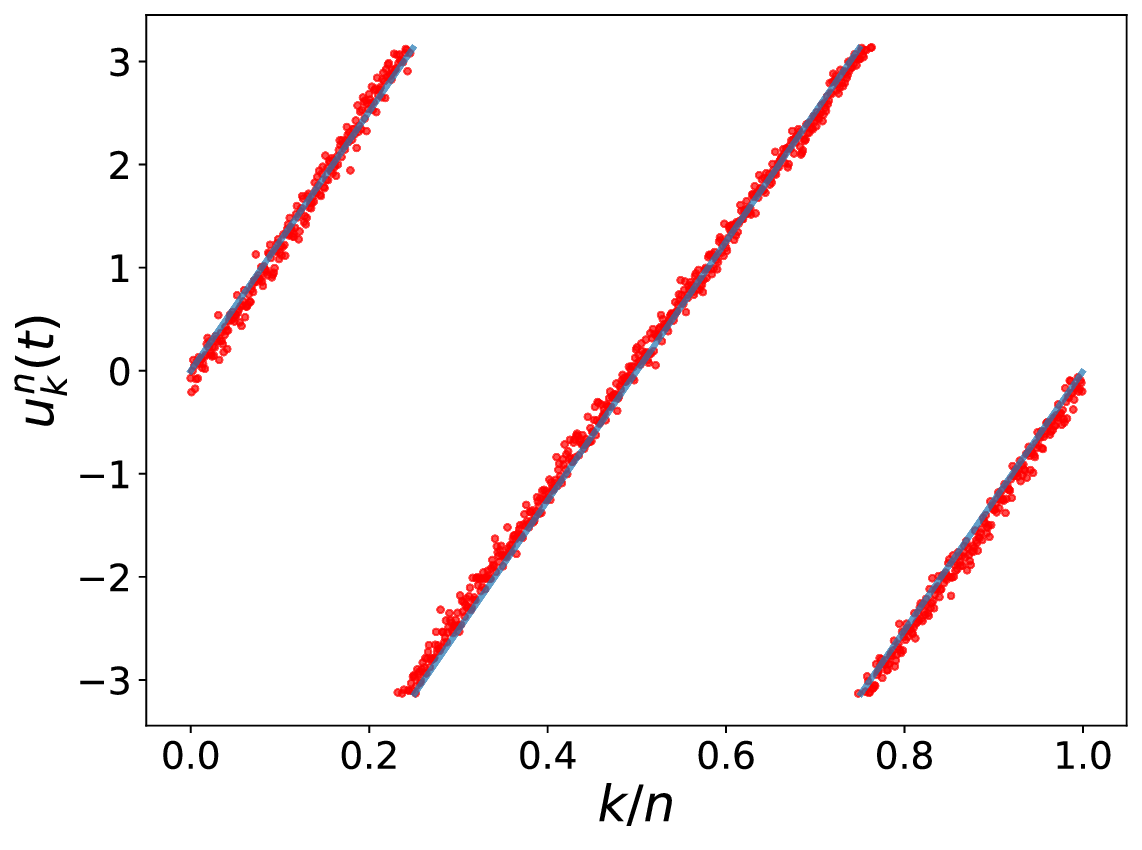}\\[-1ex]
{\footnotesize(b)}
\end{center}
\end{minipage}
\vspace*{1ex}

\begin{minipage}[t]{0.495\textwidth}
\begin{center}
\includegraphics[scale=0.3]{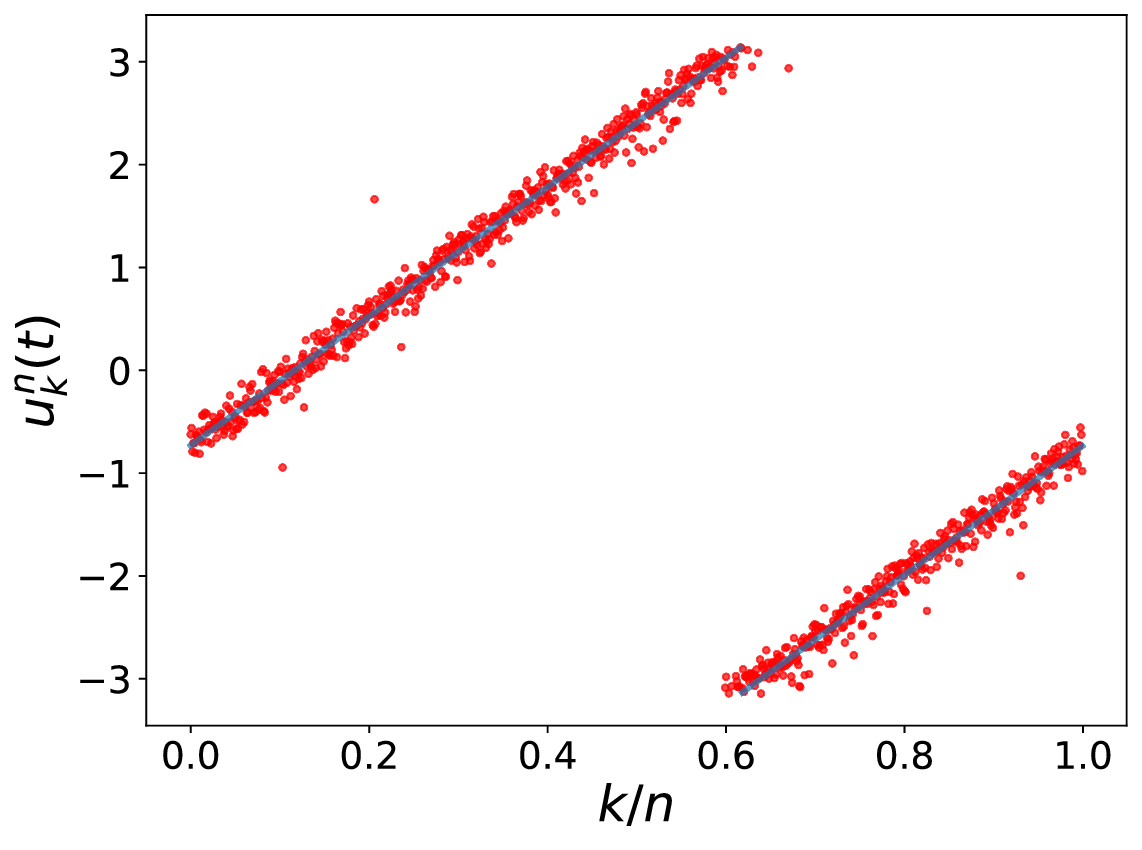}\\[-1ex]
{\footnotesize(c)}
\end{center}
\end{minipage}
\begin{minipage}[t]{0.495\textwidth}
\begin{center}
\includegraphics[scale=0.3]{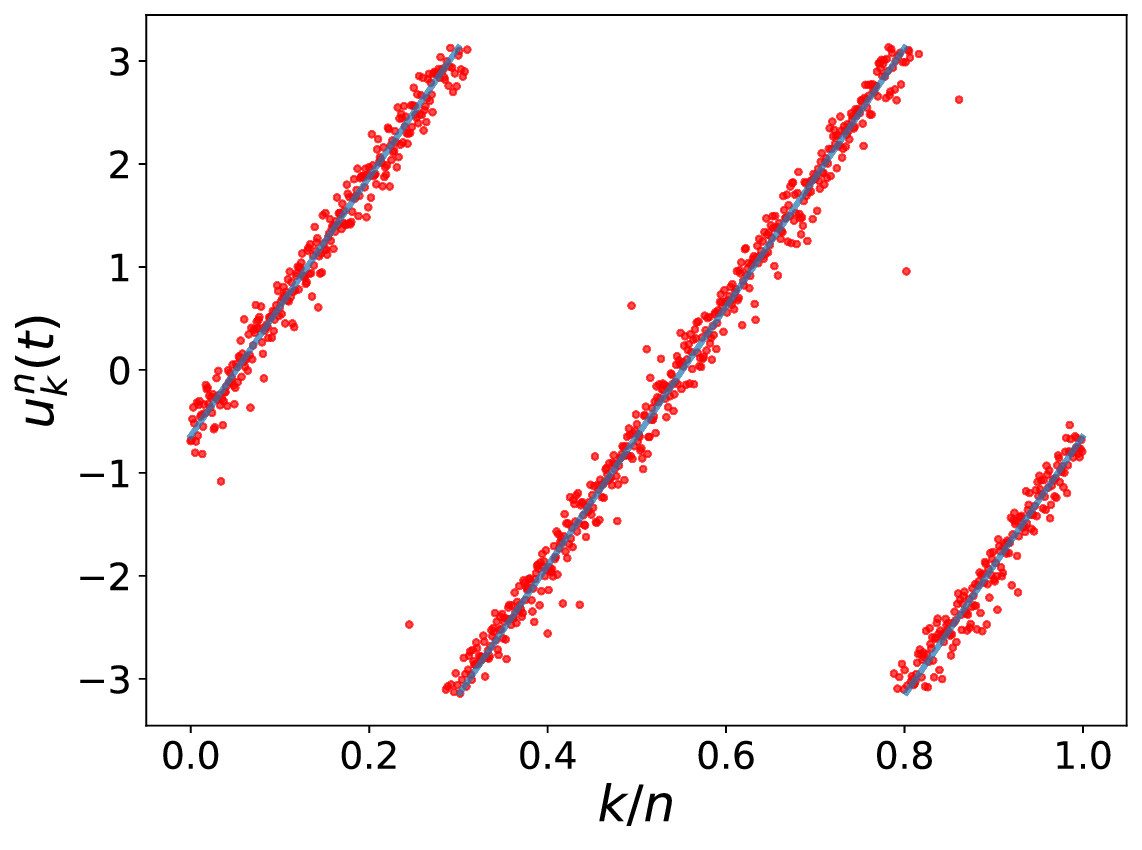}\\[-1ex]
{\footnotesize(d)}
\end{center}
\end{minipage}
\caption{Steady states of the KM \eqref{eqn:dsys}
 with $n=1000$ and $t=1000$ in case~(ii):
(a) $(q,\kappa,\sigma)=(1,0.31,0)$;
(b) $(2,0.15,0)$;
(c)  $(1,0.31,\pi/3)$;
(d) $(2,0.15,\pi/3)$.
See also the captions of Figs~\ref{fig:5b2} and \ref{fig:5d1}.}
\label{fig:5d2}
\end{figure}

\begin{figure}[t]
\begin{minipage}[t]{0.495\textwidth}
\begin{center}
\includegraphics[scale=0.3]{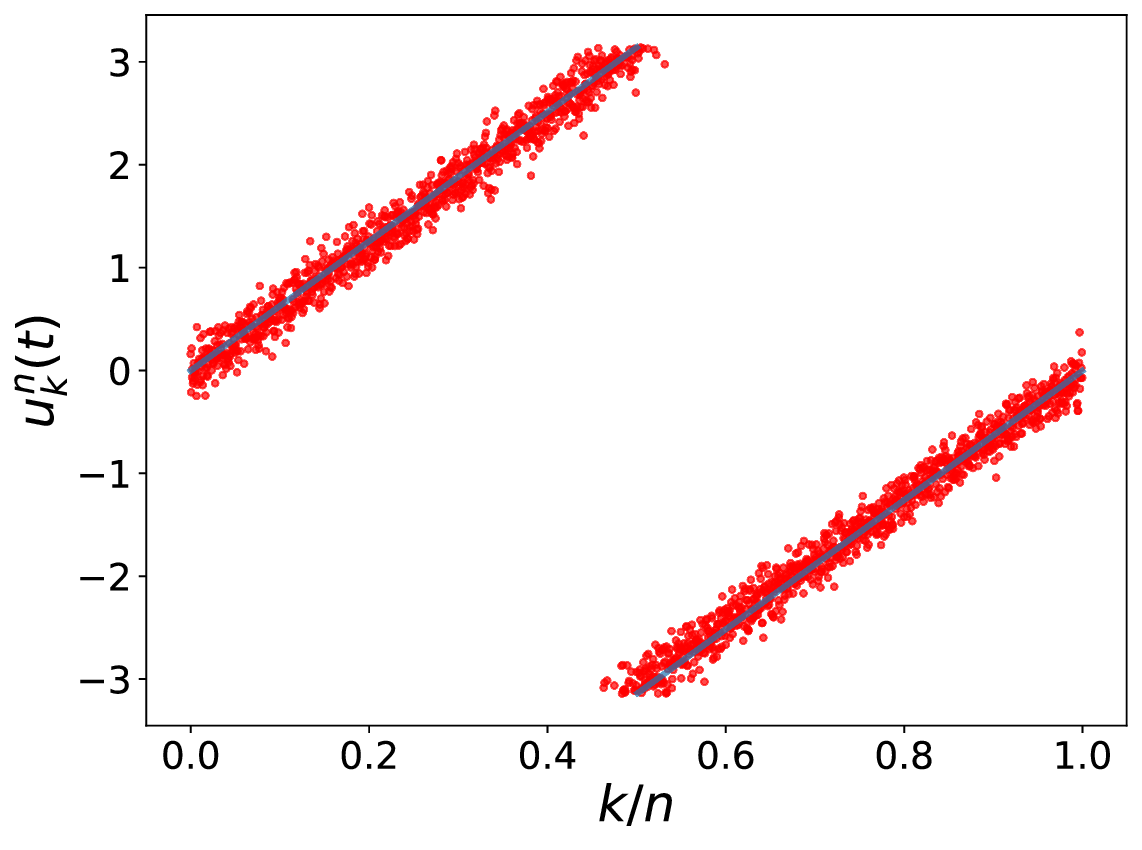}\\[-1ex]
{\footnotesize(a)}
\end{center}
\end{minipage}
\begin{minipage}[t]{0.495\textwidth}
\begin{center}
\includegraphics[scale=0.3]{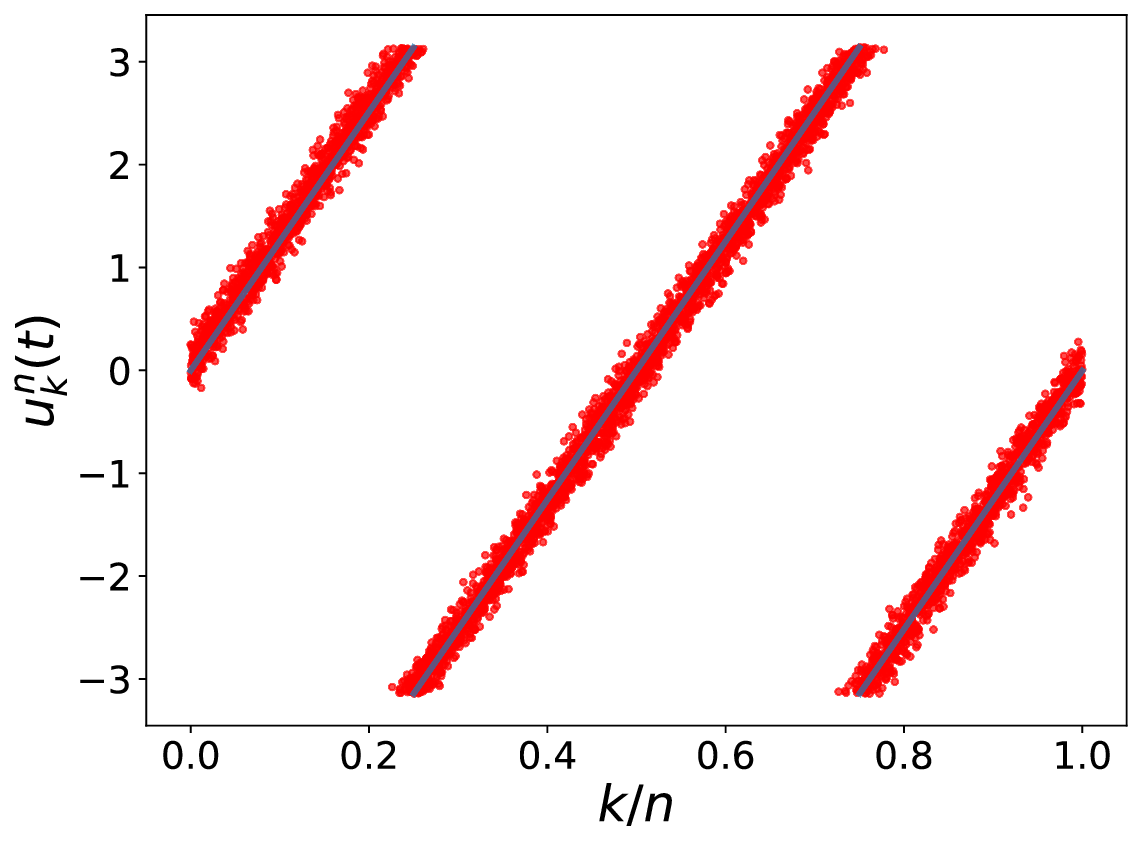}\\[-1ex]
{\footnotesize(b)}
\end{center}
\end{minipage}
\vspace*{1ex}

\begin{minipage}[t]{0.495\textwidth}
\begin{center}
\includegraphics[scale=0.3]{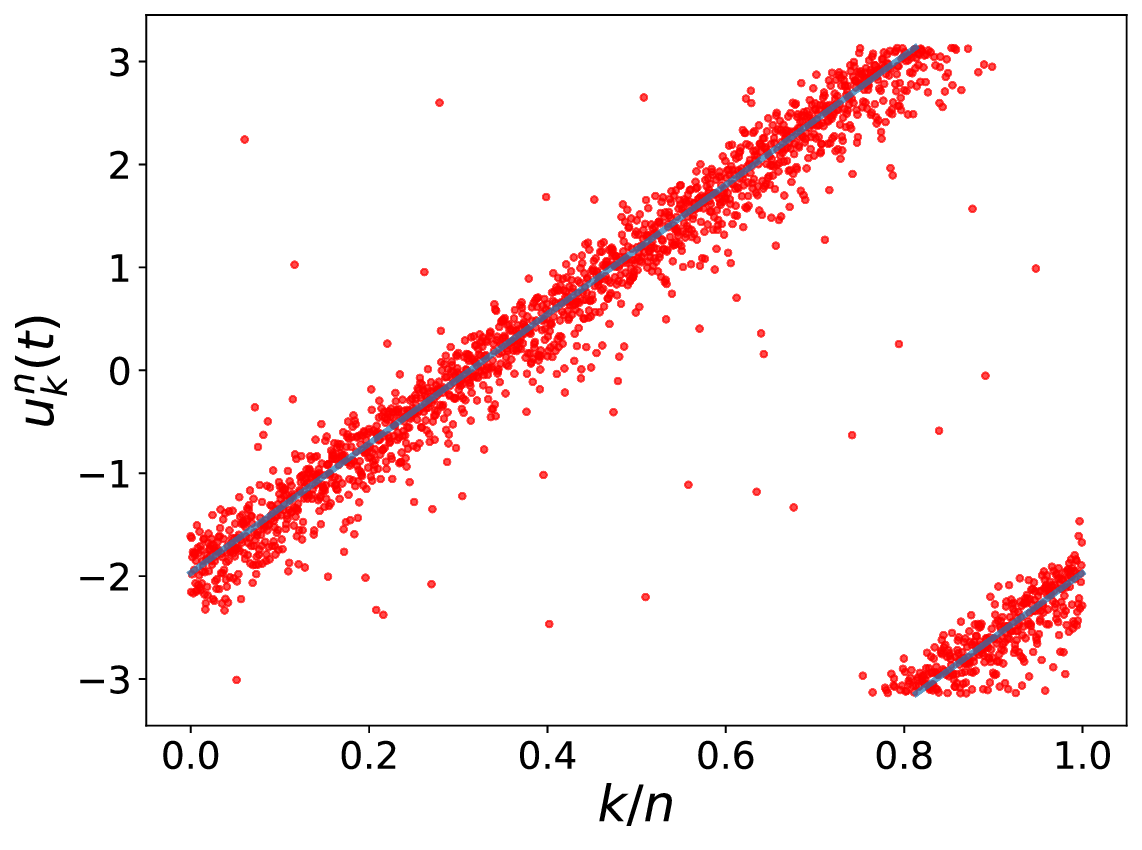}\\[-1ex]
{\footnotesize(c)}
\end{center}
\end{minipage}
\begin{minipage}[t]{0.495\textwidth}
\begin{center}
\includegraphics[scale=0.3]{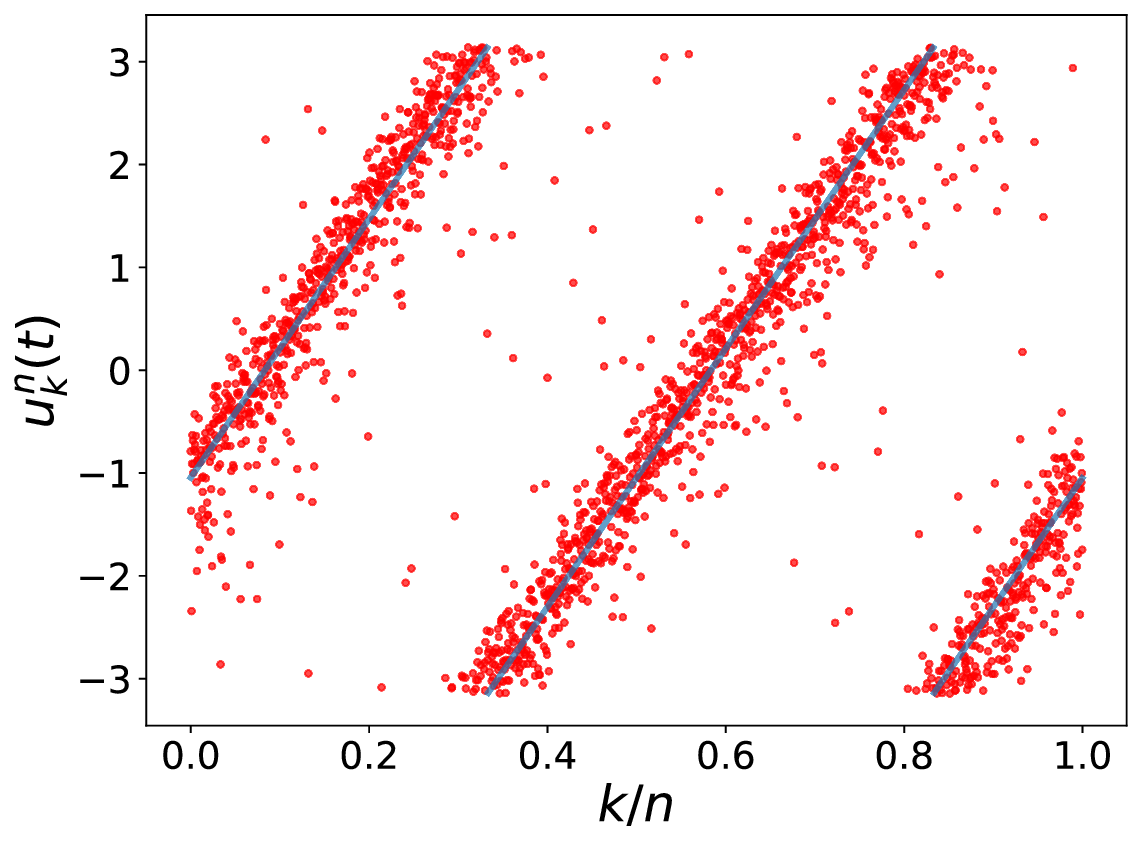}\\[-1ex]
{\footnotesize(d)}
\end{center}
\end{minipage}
\caption{Steady states of the KM \eqref{eqn:dsys}
 with $n=2000$ and $t=1000$ in case~(iii):
(a) $(q,\kappa,\sigma)=(1,0.31,0)$;
(b) $(2,0.15,0)$;
(c) $(1,0.31,\pi/3)$;
(d) $(2,0.15,\pi/3)$.
See also the captions of Figs.~\ref{fig:5b2} and \ref{fig:5d1}.}
\label{fig:5e2}
\end{figure}

In Figs.~\ref{fig:5d2} and \ref{fig:5e2},
 $u_k^n(t)$, $k\in[n]$, are plotted as small red disks
 for cases~(ii) and (iii), respectively,
 where the time $t=1000$ was chosen
 such that they may be regarded as the steady states,
 and the same values as in Figs.~\ref{fig:5d1} and \ref{fig:5e1}
 for $n$, $\kappa$ and $u_k^n(0)$, $k\in[n]$ were used.
Here $\sigma=0$ in plates~(a) and (b) of both figures
 and $\sigma=\pi/3$ in plates~(c) and (d).
The most probable twisted state \eqref{eqn:tsol}
  for the CL \eqref{eqn:csys} from the numerical result
  is also plotted as a blue line in each figure,
  as in Figs.~\ref{fig:5b2} and \ref{fig:5c2}.
 We observe that the KM \eqref{eqn:dsys} exhibits such twisted states
 as detected by Theorems~\ref{thm:4a} and \ref{thm:4b} for the CL \eqref{eqn:csys},
 although small fluctuations due to randomness are found.
Oscillating twisted solutions
 approximately given by \eqref{eqn:thm4b}
 were not observed in numerical simulations for random graphs of cases~(ii) and (iii)
 when $q=2$,
 since they are very subtle even in deterministic graphs of case (i).

\section*{Author Contributions}
K.Y. wrote the manuscript and prepared all figures.

\section*{Funding}
This work was partially supported by the JSPS KAKENHI Grant Number JP23K22409.

\section*{Data Availability}
Data sets generated during the current study are available
 from the author on reasonable request.

\section*{Conflict of Interest}
The author declares no conflict of interest.


\appendix

\renewcommand{\theequation}{\Alph{section}.\arabic{equation}}

{\color{black}
\section{Eigenvalues of $\L$ when $q=0$}

Let $q=0$.
The linear operator $\L:L^2(I)\to L^2(I)$ is written as 
\begin{align}
\L\phi(x)
=& p\cos\sigma\int_{x-\kappa}^{x+\kappa}\phi(y)\d y-2p\kappa\cos\sigma\phi(x).
\label{eqn:ep0}
\end{align}
Obviously, $\phi(x)=1$ is an eigenfunction for the zero eigenvalue.
Moreover,
\[
\phi(x)=\cos 2 \pi\ell x,\quad
\sin 2\pi\ell x
\]
are eigenfunctions for the eigenvalue
\[
\lambda=-p\cos\sigma\left(2\kappa-\frac{\sin 2\pi\ell\kappa}{\pi\ell}\right)
\]
for each $\ell\in\Nset$.
Thus, when $\kappa$ is changed,
 the zero eigenvalue is always simple 
 and the remaining eigenvalues are negative or positive,
 depending on whether $\cos\sigma$ is positive or negative.}

\section{Proof of Lemma~\ref{lem:3a}}

Let
\[
\psi_1(\zeta)=\frac{\sin\zeta}{\zeta},\quad
\psi_2(z)=\frac{2z^2-2z-1}{z-2}.
\]
Since
\begin{equation}
\frac{\d\psi_1}{\d\zeta}(\zeta)
=\frac{1}{\zeta}\left(\cos\zeta-\frac{\sin\zeta}{\zeta}\right),
\label{eqn:appa1}
\end{equation}
we see that $\psi_1(\zeta)$ has a local minimum on $((2j-1)\pi,2j\pi)$
 and a local maximum on $(2j\pi,(2j+1)\pi)$ for $j\in\Nset$.
We also have $\psi_2(-1)=-1$, $\psi_2(1)=1$ and
\[
\frac{\d\psi_2}{\d z}(z)=\frac{2z^2-8z+5}{(z-2)^2}=2-\frac{3}{(z-2)^2},\quad
\frac{\d^2\psi_2}{\d z^2}(z)=\frac{6}{(z-2)^3}>0,
\]
so that $\psi_2(z)$ has a local maximum $6-2\sqrt{6}>1$
 at $z=z_0:=2-\tfrac{1}{2}\sqrt{6}$ on $[-1,1]$.
Hence, $\psi_2(\cos\zeta)$ has local maxima at $\zeta=\arccos z_0$ and $2\pi-\arccos z_0$
 and local minima at $\zeta=0,\pi$ on $[0,2\pi)$, since
\[
\frac{\d}{\d\zeta}\psi_2(\cos\zeta)=-\frac{\d\psi_2}{\d z}(\cos\zeta)\sin\zeta.
\]
Since
\[
\sin\zeta<\zeta,\quad
1-\tfrac{1}{2}\zeta^2<\cos\zeta
\]
on $(0,\tfrac{1}{2}\pi)$, we have
\[
\frac{\d^2\psi_1}{\d\zeta^2}(\zeta)
=-\frac{\zeta\cos\zeta-(1-\tfrac{1}{2}\zeta^2)\sin\zeta}{2\zeta^3}<0,
\]
and
\begin{align*}
\frac{\d^2}{\d\zeta^2}\psi_2(\cos\zeta)
 =&\frac{\d^2\psi_2}{\d z^2}(\cos\zeta)\sin^2\zeta-\frac{\d\psi_2}{\d z}(\cos\zeta)\cos\zeta\\
=& -\left(\frac{2z^4-12z^3+27z^2-10z-6}{(z-2)^3}\right)_{z=\cos\zeta}<0
\end{align*}
on $(\tfrac{1}{3}\pi,\tfrac{1}{2}\pi)$, so that
\[
\frac{\d\psi_1}{\d\zeta}(\zeta)>\frac{\d\psi_1}{\d\zeta}(\tfrac{1}{2}\pi)
=-\frac{4}{\pi^2}=-0.40528\ldots
\]
and
\[
\frac{\d}{\d\zeta}\psi_2(\cos\zeta)<-\frac{\d\psi_2}{\d z}(\tfrac{1}{2})\sin\tfrac{1}{3}\pi
=-\frac{\sqrt{3}}{3}=-0.57735\ldots
\]
there.
Hence, the graphs of $\psi_1(\zeta)$ and $\psi_2(\cos\zeta)$
 have no tangency on $(\tfrac{1}{3}\pi,\tfrac{1}{2}\pi)$.
Since
\[
\psi_1(\zeta)-\psi_2(\cos\zeta)=\frac{3\sqrt{3}}{2\pi}-1<0
\]
at $\zeta=\tfrac{1}{3}\pi$ and
\[
\psi_1(\zeta)-\psi_2(\cos\zeta)=\frac{2}{\pi}-\tfrac{1}{2}>0
\]
at $\zeta=\tfrac{1}{2}\pi$, we conclude that $\zeta_1\in(\tfrac{1}{3}\pi,\tfrac{1}{2}\pi)$,
 at which $\varphi(\zeta)$ has a local maximum.

It remains to show that $\zeta_j\in((j-1)\pi,j\pi)$
 and $\d\varphi/\d\zeta\neq 0$ at $\zeta=\zeta_j$ for $j\ge 2$.
We first see that when $\zeta>\pi$,
\[
|\psi_1(\zeta)|<\frac{1}{\pi},\quad
\left|\frac{\d\psi_1}{\d\zeta}(\zeta)\right|\le\frac{1}{\pi}+\frac{1}{\pi^2}=0.41963\ldots
\]
by \eqref{eqn:appa1}.
On the other hand, when $|\psi_2(z)|<1/\pi$,
 we have $z\in(-\tfrac{7}{12},-\tfrac{1}{8})$ and
\[
\left|\frac{\d\psi_2}{\d z}(z)\sqrt{1-z^2}\right|
>\frac{\d\psi_2}{\d z}(-\tfrac{1}{8})\sqrt{1-\left(\tfrac{7}{12}\right)^2}
=\frac{193\sqrt{95}}{1734}=1.0848\ldots.
\]
Thus, the graphs of $\psi_1(\zeta)$ and $\psi_2(\cos\zeta)$
 have no tangency on $(\pi,\infty)$.
Noting that $\psi_1(\zeta)<0$ on $(2j-1)\pi,2j\pi)$ and $\psi_1(\zeta)>0$ on $(2j\pi,(2j+1)\pi)$,
 we obtain the desired result.
 
\section{Derivation of \eqref{eqn:ifex}}

We first rewrite the CL \eqref{eqn:csys}
 in the rotational frame with speed $\Omega$
 given by \eqref{eqn:Omega} as
\begin{align}
\frac{\partial}{\partial t}u(t,x)
=&\omega-\Omega+p\left(\cos u(t,x)\int_{x-\kappa}^{x+\kappa}\sin u(t,y)\d y\right.\notag\\
&\qquad\qquad
\left. -\sin u(t,x)\int_{x-\kappa}^{x+\kappa}\cos u(t,y)\d y\right)\cos\sigma\notag\\
&
+p\left(\sin u(t,x)\int_{x-\kappa}^{x+\kappa}\sin u(t,y)\d y\right.\notag\\
&\qquad\qquad
\left. +\cos u(t,x)\int_{x-\kappa}^{x+\kappa}\cos u(t,y)\d y\right)\sin\sigma.
\label{eqn:csys1}
\end{align}
Letting \eqref{eqn:solex} with $\Omega=0$, we have
\begin{align*}
\cos u(t,x)
=&\cos 2\pi qx-\sin 2\pi qx
\biggl(\xi_0+\sum_{j=1}^\infty(\xi_j\cos 2\pi jx+\eta_j\sin2\pi jx)\biggr)\\
&
-\cos 2\pi qx\bigl(\tfrac{1}{4}((\xi_1^2+\eta_1^2)
 +(\xi_1^2-\eta_1^2)\cos 4\pi x+2\xi_1\eta_1\sin 4\pi x)\\
&\quad
 +\tfrac{1}{2}\xi_0^2+\xi_0(\xi_1\cos2\pi x+\eta_1\sin2\pi x+\xi_2\cos4\pi x+\eta_2\sin4\pi x)\\
&\quad
+\tfrac{1}{2}((\xi_1\xi_2+\eta_1\eta_2)\cos2\pi x+(\xi_1\eta_2-\xi_2\eta_1)\sin2\pi x\\
&\quad
+(\xi_1\xi_2-\eta_1\eta_2)\cos6\pi x+(\xi_1\eta_2+\xi_2\eta_1)\sin6\pi x)\bigr)\\
&
+\sin 2\pi qx\bigl(
 \tfrac{1}{8}(\xi_1^2+\eta_1^2)(\xi_1\cos2\pi x+\eta_1\sin 2\pi x)\\
&\quad
+\tfrac{1}{24}((\xi_1^2-3\eta_1^2)\xi_1\cos 6\pi x
 +(3\xi_1^2-\eta_1^2)\eta_1\sin 6\pi x)\\
&\quad
+\tfrac{1}{6}\xi_0^3 +\tfrac{1}{2}\xi_0^2(\xi_1\cos2\pi x+\eta_1\sin2\pi x)\\
&\quad
+\tfrac{1}{4}\xi_0((\xi_1^2+\eta_1^2)
 +(\xi_1^2-\eta_1^2)\cos 4\pi x+2\xi_1\eta_1\sin 4\pi x)\bigr)+\cdots
\end{align*}
and
\begin{align*}
\sin u(t,x)
=& \sin 2\pi qx
+\cos 2\pi qx\biggl(\xi_0+\sum_{j=1}^\infty(\xi_j\cos 2\pi jx+\eta_j\sin2\pi jx)\biggr)\\
&
-\sin 2\pi qx\bigl(\tfrac{1}{4}((\xi_1^2+\eta_1^2)
 +(\xi_1^2-\eta_1^2)\cos 4\pi  x+2\xi_1\eta_1\sin 4\pi x)\\
&\quad
 +\tfrac{1}{2}\xi_0^2+\xi_0(\xi_1\cos2\pi x+\eta_1\sin2\pi x+\xi_2\cos4\pi x+\eta_2\sin4\pi x)\\
&\quad
+\tfrac{1}{2}((\xi_1\xi_2+\eta_1\eta_2)\cos2\pi x+(\xi_1\eta_2-\xi_2\eta_1)\sin2\pi x\\
&\quad
+(\xi_1\xi_2-\eta_1\eta_2)\cos6\pi x+(\xi_1\eta_2+\xi_2\eta_1)\sin6\pi x)\bigr)\\
&
-\cos 2\pi qx\bigl(
 \tfrac{1}{8}(\xi_1^2+\eta_1^2)(\xi_1\cos2\pi x+\eta_1\sin 2\pi x)\\
&\quad
+\tfrac{1}{24}((\xi_1^2-3\eta_1^2)\xi_1\cos 6\pi x
 +(3\xi_1^2-\eta_1^2)\eta_1\sin 6\pi x)\\
&\quad
+\tfrac{1}{6}\xi_0^3 +\tfrac{1}{2}\xi_0^2(\xi_1\cos2\pi x+\eta_1\sin2\pi x)\\
&\quad
+\tfrac{1}{4}\xi_0((\xi_1^2+\eta_1^2)
 +(\xi_1^2-\eta_1^2)\cos 4\pi x+2\xi_1\eta_1\sin 4\pi x)\bigr)+\cdots,
\end{align*}
where `$\cdots$' represents higher-order terms of
\[
O\left(\sqrt{\xi_0^8+\xi_1^8+\eta_1^8
 +\xi_2^4+\eta_2^4
 +\sum_{j=3}^\infty(\xi_j^2+\eta_j^2)^{4/3}}\right).
\]
We compute the integrals in \eqref{eqn:csys1} as
\begin{align*}
&
\int_{x-\kappa}^{x+\kappa}\cos u(t,y)\d y\\
&
=-a_2(q,0)\cos2\pi qx
 -a_2(q,0)\xi_0\bigl(\tfrac{1}{6}\xi_0^2+\tfrac{1}{4}(\xi_1^2+\eta_1^2)-1\bigr)\sin2\pi qx\\
&\quad
-\sum_{j=1}^\infty(a_1(q,j)(\xi_j\sin2\pi jx-\eta_j\cos2\pi jx)\cos2\pi qx\\
&\qquad
-a_2(q,j)(\xi_j\cos2\pi j+\eta_j\sin2\pi jx)\sin2\pi qx)\\
&\quad
+\tfrac{1}{4}a_2(q,0)(2\xi_0^2+\xi_1^2+\eta_1^2)\cos 2\pi qx\\
&\quad
+\tfrac{1}{4}a_1(q,2)((\xi_1^2-\eta_1^2+4\xi_0\xi_2)\sin4\pi x
 -2(\xi_1\eta_1+2\xi_0\eta_2)\cos4\pi x)\sin2\pi q x\\
&\quad
+\tfrac{1}{4}a_2(q,2)((\xi_1^2-\eta_1^2+4\xi_0\xi_2)\cos4\pi x
+2(\xi_1\eta_1+2\xi_0\eta_2)\sin4\pi x)\cos2\pi qx
\end{align*}
\begin{align*}
&\quad
+\tfrac{1}{8}a_1(q,1)(4\xi_0^2+\xi_1^2+\eta_1^2)(\xi_1\sin2\pi x-\eta_1\cos2\pi x)\cos2\pi qx\\
&\quad
-\tfrac{1}{8}a_2(q,1)(4\xi_0^2+\xi_1^2+\eta_1^2)(\xi_1\cos2\pi x+\eta_1\sin2\pi x)\sin2\pi qx\\
&\quad
+\tfrac{1}{24}a_1(q,3)
 ((\xi_1^2-3\eta_1^2)\xi_1\sin6\pi x-(3\xi_1^2-\eta_1^2)\eta_1\cos6\pi x)\cos2\pi qx\\
&\quad
-\tfrac{1}{24}a_2(q,3)
 ((\xi_1^2-3\eta_1^2)\xi_1\cos6\pi x+(3\xi_1^2-\eta_1^2)\eta_1\sin6\pi x)\sin2\pi qx\\
&\quad
+\tfrac{1}{2}a_1(q,1)((2\xi_0\xi_1+\xi_1\xi_2+\eta_1\eta_2)\sin2\pi x\\
&\qquad
 -(2\xi_0\eta_1+\xi_1\eta_2-\xi_2\eta_1)\cos2\pi x)\sin2\pi q x\\
&\quad
+\tfrac{1}{2}a_2(q,1)((2\xi_0\xi_1+\xi_1\xi_2+\eta_1\eta_2)\cos2\pi x\\
&\qquad
 +(2\xi_0\eta_1+\xi_1\eta_2-\xi_2\eta_1)\sin2\pi x)\cos2\pi q x\\
&\quad
+\tfrac{1}{2}a_1(q,3)((\xi_1\xi_2-\eta_1\eta_2)\sin6\pi x
 -(\xi_1\eta_2+\xi_2\eta_1)\cos6\pi x)\sin2\pi q x\\
&\quad
+\tfrac{1}{2}a_2(q,3)((\xi_1\xi_2-\eta_1\eta_2)\cos6\pi x
 +(\xi_1\eta_2+\xi_2\eta_1)\sin6\pi x)\cos2\pi q x\\
&\quad
+\tfrac{1}{4}a_1(q,2)\xi_0((\xi_1^2-\eta_1^2)\sin4\pi x-2\xi_1\eta_1\cos4\pi x)\cos2\pi q x\\
&\quad
-\tfrac{1}{4}a_2(q,2)\xi_0((\xi_1^2-\eta_1^2)\cos4\pi x+2\xi_1\eta_1\sin4\pi x)\sin2\pi qx
 +\cdots
\end{align*}
and
\begin{align*}
&
\int_{x-\kappa}^{x+\kappa}\sin u(t,y)\d y\\
&
=-a_2(q,0)\sin2\pi qx
 +a_2(q,0)\xi_0\bigl(\tfrac{1}{6}\xi_0^2+\tfrac{1}{4}(\xi_1^2+\eta_1^2)-1\bigr)\cos2\pi qx\\
&\quad
-\sum_{j=1}^\infty(a_1(q,j)(\xi_j\sin2\pi jx-\eta_j\cos2\pi jx)\sin2\pi qx\\
&\qquad
+a_2(q,j)(\xi_j\cos2\pi j+\eta_j\sin2\pi jx)\cos2\pi qx)\\
&\quad
+\tfrac{1}{4}a_2(q,0)(2\xi_0^2+\xi_1^2+\eta_1^2)\sin 2\pi qx\\
&\quad
-\tfrac{1}{4}a_1(q,2)((\xi_1^2-\eta_1^2+4\xi_0\xi_2)\sin4\pi x
 -2(\xi_1\eta_1+2\xi_0\eta_2)\cos4\pi x)\cos2\pi q x\\
&\quad
+\tfrac{1}{4}a_2(q,2)((\xi_1^2-\eta_1^2+4\xi_0\xi_2)\cos4\pi x
 +2(\xi_1\eta_1+2\xi_0\eta_2)\sin4\pi x)\sin2\pi qx\\
&\quad
+\tfrac{1}{8}a_1(q,1)(4\xi_0^2+\xi_1^2+\eta_1^2)(\xi_1\sin2\pi x-\eta_1\cos2\pi x)\sin2\pi qx\\
&\quad
+\tfrac{1}{8}a_2(q,1)(4\xi_0^2+\xi_1^2+\eta_1^2)(\xi_1\cos2\pi x+\eta_1\sin2\pi x)\cos2\pi qx\\
&\quad
+\tfrac{1}{24}a_1(q,3)
 ((\xi_1^2-3\eta_1^2)\xi_1\sin6\pi x-(3\xi_1^2-\eta_1^2)\eta_1\cos6\pi x)\sin2\pi qx\\
&\quad
+\tfrac{1}{24}a_2(q,3)
 ((\xi_1^2-3\eta_1^2)\xi_1\cos6\pi x+(3\xi_1^2-\eta_1^2)\eta_1\sin6\pi x)\cos2\pi qx\\
&\quad
-\tfrac{1}{2}a_1(q,1)((2\xi_0\xi_1+\xi_1\xi_2+\eta_1\eta_2)\sin2\pi x\\
&\qquad
-(2\xi_0\eta_1+\xi_1\eta_2-\xi_2\eta_1)\cos2\pi x)\cos2\pi q x\\
&\quad
+\tfrac{1}{2}a_2(q,1)((2\xi_0\xi_1+\xi_1\xi_2+\eta_1\eta_2)\cos2\pi x\\
&\qquad
+(2\xi_0\eta_1+\xi_1\eta_2-\xi_2\eta_1)\sin2\pi x)\sin2\pi q x\\
&\quad
-\tfrac{1}{2}a_1(q,3)((\xi_1\xi_2-\eta_1\eta_2)\sin6\pi x
 -(\xi_1\eta_2+\xi_2\eta_1)\cos6\pi x)\cos2\pi q x\\
&\quad
+\tfrac{1}{2}a_2(q,3)((\xi_1\xi_2-\eta_1\eta_2)\cos6\pi x
 +(\xi_1\eta_2+\xi_2\eta_1)\sin6\pi x)\sin2\pi q x\\
&\quad
+\tfrac{1}{4}a_1(q,2)\xi_0((\xi_1^2-\eta_1^2)\sin4\pi x-2\xi_1\eta_1\cos4\pi x)\sin2\pi q x\\
&\quad
+\tfrac{1}{4}a_2(q,2)\xi_0((\xi_1^2-\eta_1^2)\cos4\pi x+2\xi_1\eta_1\sin4\pi x)\cos2\pi qx
+\cdots.
\end{align*}
We substitute \eqref{eqn:solex} into \eqref{eqn:csys1}
 and integrate the resulting equation with respect to $x$ from $0$ to $1$
 after multiplying it with $\cos 2\pi j$ or $\sin 2\pi j$, $j\in\Nset$.
So we obtain \eqref{eqn:ifex} after lengthy calculations.



\begin{thebibliography}{99}

\bibitem{ABVRS05}
J.A.~Acebr\'on, L.L.~Bonilla, C.J.P.~Vicente, F.~Ritort and R.~Spigler,
The Kuramoto model: A simple paradigm for synchronization phenomena,
\textit{Rev. Mod. Phys.}, \textbf{77} (2005), 137--185.

\bibitem{ADKMZ08}
A.~Arenas, A.~Diaz-Guilera, J.~Kurths, Y.~Moreno and C.~Zhou,
Synchronization in complex networks,
\textit{Phys. Rep.},\textbf{469}(2008), 93--153.

\bibitem{C66}
 L.~Carleson, 
On convergence and growth of partial sums of  Fourier series,
\textit{Acta Mathematica}, \textbf{116} (1966), 135--157.

{\color{black}
\bibitem{C15}
H.~Chiba,
A proof of the Kuramoto conjecture for a bifurcation structure
 of the infinite-dimensional Kuramoto model,
\textit{Ergod. Theory Dyn. Syst.}, \textbf{35} (2015), 762--834.}

\bibitem{CL55}
E.A.~Coddington and N.~Levinson,
\textit{Theory of Ordinary Differential Equations},
McGraw-Hill, New York, 1955.


\bibitem{DB14}
F.~D\"orfler and F.~Bullo,
Synchronization in complex networks of phase oscillators: A survey,
\textit{Automatica}, \textbf{50} (2014), 1539--1564.

\bibitem{E85}
G.B.~Ermentrout,
Synchronization in a pool of mutually coupled oscillators with random frequencies
\textit{J. Math. Biol.}, \textbf{23} (1985), 55--74.

\bibitem{GHM12}
T.~Girnyk, M.~Hasler and Y.~Maistrenko,
Multistability of twisted states in non-locally coupled Kuramoto-type models,
\textit{Chaos}, \textbf{22} (2012), 013114.

\bibitem{GH83}
J.~Guckenheimer and P.~Holmes,
\textit{Nonlinear Oscillations, Dynamical Systems, and Bifurcations of Vector Fields},
Springer, New York, 1983.

\bibitem{HNW93}
E.~Hairer, S.P.~N{\o}rsett and G.~Wanner,
\textit{Solving Ordinary Differential Equations I: Nonstiff Problems}, 2nd ed.
Springer, Berlin, 1993.

\bibitem{HI11}
M.~Haragus and G.~Iooss,
\textit{Local Bifurcations, Center Manifolds, and Normal Forms
 in Infinite-Dimensional Dynamical Systems},
Springer, London, 2011.

\bibitem{IY23}
R.~Ihara and K.~Yagasaki,
Continuum limits of coupled oscillator networks depending on multiple sparse graphs,
\textit{J. Nonlinear Sci.}, \textbf{33} (2023), 62;
Correction, \textbf{35} (2025), 27.

\bibitem{KM17} 
D. Kaliuzhnyi-Verbovetskyi and G. S. Medvedev, 
The semilinear heat equation on sparse random graphs, 
\textit{SIAM J. Math. Anal.}, \textbf{49} (2017), no. 2, 1333-1355. 

\bibitem{K75} 
Y.~Kuramoto,
Self-entrainment of a population of coupled non-linear oscillators,
in \textit{International Symposium on Mathematical Problems in Theoretical Physics},
H.~Araki (ed.),
Springer, Berlin, 1975, pp. 420--422.
\bibitem{K84} 
Y.~Kuramoto,
\textit{Chemical Oscillations, Waves, and Turbulence}, Springer, Berlin, 1984. 

\bibitem{K04}
Y.A.~Kuznetsov,
\textit{Elements of Applied Bifurcation Theory},
Springer, New York, 2004.

\bibitem{L12} 
L.~Lov\'asz, 
\textit{Large Networks and Graph Limits}, 
AMS, Providence RI, 2012. 

\bibitem{M14a} 
G.S.~Medvedev, 
The nonlinear heat equation on dense graphs and graph limits, 
\textit{SIAM J. Math. Anal.}, \textbf{46} (2014), 2743--2766. 
\bibitem{M14b} 
G.S.~Medvedev, 
The nonlinear heat equation on W-random graphs, 
\textit{Arch. Ration. Mech. Anal.}, \textbf{212} (2014), 781--803. 
\bibitem{M14c} 
G.S.~Medvedev, 
Small-world networks of Kuramoto oscillators,
\textit{Phys. D}, \textbf{266} (2014), 13--22.
\bibitem{M19} 
G.S.~Medvedev, 
The continuum limit of the Kuramoto model on sparse random graphs, 
\textit{Comm. Math. Sci.}, \textbf{17} (2019), no. 4, 883--898.
\bibitem{MM22}
G.S.~Medvedev and  M.S.~Mizuhara,
Chimeras unfolded,
\textit{J. Stat. Phys.}, \textbf{186} (2022), 46.
\bibitem{MW17}
G.S.~Medvedev and  J.D.~Wright,
Stability of twisted states in the continuum Kuramoto model, 
\textit{SIAM J. Appl. Dyn. Syst.}, \textbf{16} (2017), 188--203. 


\bibitem{PR15}
A.~Pikovsky and M.~Rosenblum,
Dynamics of globally coupled oscillators: Progress and perspectives,
\textit{Chaos}, \textbf{25} (2015), 097616.

\bibitem{PRK01}
A. Pikovsky, M. Rosenblum, and J. Kurths,
\textit{Synchronization: A Universal Concept in Nonlinear Sciences},
Cambridge University Press, Cambridge, 2001.

\bibitem{RPJK16}
F.A.~Rodrigues, T.K.DM.~Peron, P.~Ji and J.~Kurths,
The Kuramoto model in complex networks,
\textit{Phys. Rep.}, \textbf{610} (2016), 1--98.

\bibitem{S00}
S.H.~Strogatz,
From Kuramoto to Crawford: Exploring the onset of synchronization
 in populations of coupled oscillators,
\textit{Phys. D}, \textbf{143} (2000), 1--20.

{\color{black}
\bibitem{SM91}
S.H.~Strogatz and R.E.~Mirollo,
Stability of incoherence in a population of coupled oscillators,
\textit{J. Stat. Phys.}, \textbf{63} (1991), 613--635.}

\bibitem{WSG06}
D.A.~Wiley, S.H.~Strogatz and M.~Girvan,
The size of the sync basin,
\textit{Chaos}, \textbf{16} (2006), 015103.

\bibitem{Y24a}
K.~Yagasaki,
Bifurcations and stability of synchronized solutions in the Kuramoto model
 with uniformly spaced natural frequencies,
\textit{Nonlinearity}, {\color{black}\textbf{38} (2025), 075032;
 Corrigendum, \textbf{38} (2025), 109501.}
\bibitem{Y24b}
K.~Yagasaki,
Bifurcations of synchronized solutions in a continuum limit
 of the Kuramoto model with two-mode interaction depending on two graphs,
\textit{SIAM J. Dyn. Syst.}, {\color{black}\textbf{24} (2025), 2345--2368}.
\bibitem{Y24c}
K.~Yagasaki,
Feedback control of twisted states in the Kuramoto model
 on nearest neighbor or uniform graphs,
submitted for publication.
\end{thebibliography}
\end{document}